\documentclass{amsart}

\usepackage{amsmath}
\usepackage{amssymb}
\usepackage{amsthm}

\usepackage[T1]{fontenc}
\usepackage{mathrsfs}
\usepackage{xcolor}
\usepackage{graphicx}
\usepackage{bbm}

\usepackage{hyperref}

\newcommand{\bE}{\mathbb{E}}

\newcommand{\bN}{\mathbb{N}}

\newcommand{\bP}{\mathbb{P}}

\newcommand{\bR}{\mathbb{R}}

\newcommand{\bZ}{\mathbb{Z}}

\newcommand{\cB}{\mathcal{B}}
\newcommand{\cC}{\mathcal{C}}
\newcommand{\cD}{\mathcal{D}}

\newcommand{\cF}{\mathcal{F}}

\newcommand{\cI}{\mathcal{I}}

\newcommand{\cK}{\mathcal{K}}

\newcommand{\cN}{\mathcal{N}}
\newcommand{\cO}{\mathcal{O}}

\newcommand{\cT}{\mathcal{T}}

\newcommand{\cW}{\mathcal{W}}
\newcommand{\cX}{\mathcal{X}}

\newcommand{\sT}{\mathscr{T}}

\newcommand{\N}{\bN}
\newcommand{\R}{\bR}

\newcommand{\rVar}{\mathrm{Var}}
\newcommand{\rCov}{\mathrm{Cov}}

\newcommand{\diff}{\mathrm{d}}

\newcommand{\scp}[2]{{\left\langle{#1},{#2}\right\rangle}}
\newcommand{\scpSmall}[2]{{\langle{#1},{#2}\rangle}}

\newcommand{\sscp}[2]{{\left\langle\left\langle{#1},{#2}\right\rangle\right\rangle}}
\newcommand{\sscpSmall}[2]{{\langle\langle{#1},{#2}\rangle\rangle}}

\newcommand{\norm}[1]{{\left\lVert{#1}\right\rVert}}
\newcommand{\normSmall}[1]{{\lVert{#1}\rVert}}

\newcommand{\abs}[1]{{\left|{#1}\right|}}
\newcommand{\absSmall}[1]{{|{#1}|}}

\newcommand{\NonparIndex}{{\gamma}}
\newcommand{\shift}{\sT}
\newcommand{\FuncZ}{\varphi}
\newcommand{\FuncY}{\psi}
\newcommand{\FuncX}{\chi}
\newcommand{\FuncPotential}{\bar\psi}
\newcommand{\multiindex}{\bar\gamma}

\newcommand{\assumptionB}{{$(B_{\beta, \underline{\vartheta}, C})$}}

\newtheorem{theorem}{Theorem}
\newtheorem{proposition}[theorem]{Proposition}
\newtheorem{lemma}[theorem]{Lemma}
\newtheorem{corollary}[theorem]{Corollary}
\newtheorem{definition}[theorem]{Definition}
\newtheorem{remark}[theorem]{Remark}

\numberwithin{theorem}{section}
\counterwithin{equation}{section}

\begin{document}


\title[Diffusivity Estimation from Noisy Observations]{Nonparametric Diffusivity Estimation for the Stochastic Heat Equation from Noisy Observations}
\author[G. Pasemann, M. Rei\ss{}]{Gregor Pasemann\footnotesize{$^{1}$} and Markus Rei\ss{}\footnotesize{$^{1}$}}
\email{gregor.pasemann@hu-berlin.de \\ reismark@hu-berlin.de}

\begin{abstract}
We estimate nonparametrically the spatially varying diffusivity of a stochastic heat equation from observations perturbed by additional noise.
To that end, we employ a two-step localization procedure, more precisely, we combine local state estimates into a locally linear regression approach.
Our analysis relies on quantitative Trotter--Kato type approximation results for the heat semigroup that are of independent interest.
The presence of observational noise leads to non-standard scaling behaviour of the model.
Numerical simulations illustrate the results.
\end{abstract}

\maketitle

\footnotetext[1]{Humboldt-Universit\"at zu Berlin}


\section{Introduction}

We consider a nonparametric estimation problem for the stochastic heat equation
\begin{align}\label{eq:introduction:spde}
	\partial_t X(t, x) = \nabla\cdot\vartheta(x)\nabla X(t, x) + \sigma\dot W(t, x),\quad t\in[0,T],\,x\in\cD,
\end{align}
driven by Gaussian space-time white noise $\dot W$
on a bounded domain $\cD\subset\R^d$, with unknown spatially varying diffusivity $\vartheta:\cD\rightarrow(0,\infty)$, where we observe a noisy trajectory of the form
\begin{align}\label{eq:introduction:observation}
	Y(t, x) = X(t, x) + \varepsilon\dot V(t, x),\quad t\in[0,T],\,x\in\cD,
\end{align}
at noise level $\varepsilon>0$, where the Gaussian space-time white noise $\dot V$ is independent of $\dot W$. We understand \eqref{eq:introduction:spde} and  \eqref{eq:introduction:observation} in the space-time weak sense, i.e. tested against a class of test functions.
The precise setting is discussed in Section \ref{sec:spde}.

This observation model is the continuous analogue of the regression model 
\begin{equation}\label{eq:introduction:regrmodel}
Y_{ij}=X(t_i, x_j)+\varepsilon_{ij},\quad i=1,\ldots,n_t,\; j=1,\ldots,n_x,
\end{equation}
with independent error variables $\varepsilon_{ij}\sim N(0,\eta^2)$
when the design points $(t_i, x_j)$ are uniformly spread over $[0,T]\times \cD$  and the noise level is set to $\varepsilon=\eta (T|\cD|/N)^{1/2}$ with total sample size $N=n_tn_x$. Observables $\int_0^T\int_{\cD}\varphi(t,x)Y(t,x)\,dxdt$ for $\varphi\in C([0,T]\times\cD)$ in model \eqref{eq:introduction:observation} are replaced in model \eqref{eq:introduction:regrmodel} by  discrete sums $T|\cD|N^{-1}\newline\sum_{i,j}\varphi(t_i,x_j)Y_{ij}$, see also \cite{Reiss2011} for a strong asymptotic Le Cam equivalence result in a similar setting.

There are two sources of randomness in our setting: The \emph{dynamic noise} $\dot W$ drives the underlying state equation and models the intrinsic stochastic behaviour of the signal process $X$, e.g. due to unresolved external forces. On the other hand, the \emph{static noise} $\dot V$ captures noise in the measurement process itself. While mathematically both objects are modelled by space-time white noise, their impact on the statistical problem at hand will be fundamentally different.
The presence of dynamic noise can influence the bifurcation structure or other key properties of the dynamical system, e.g. the time to repolarization in a stochastic Meinhardt model \cite{AltmeyerBretschneiderJanakReiss2022}.
On the other hand, static noise is clearly a common feature of measurement data, see e.g. \cite{PasemannFlemmingAlonsoBetaStannat2021} for a discussion in the case of \textit{D. discoideum} giant cell microscopy data.
Models including dynamic as well as static noise are frequently used in the context of stochastic filtering, see e.g. \cite{BaCr2009}.

From a statistical perspective, assuming $X$ is observed (i.e. setting $\varepsilon=0$), \eqref{eq:introduction:spde} is related to nonparametric regression with response variable $\partial_tX$ and random design $X$. The latter depends in a non-trivial way on $\vartheta$ via its dynamic evolution, in particular, the observation at fixed $x\in\cD$ depends on $\vartheta$ in a non-local way due to wide-range interactions of the heat semigroup at positive times.

In order to construct an estimator $\hat\vartheta_\varepsilon(x_0)$ for $\vartheta$ at a given point $x_0\in\cD$, we employ a two-step procedure: First, we estimate the underlying state $X$ from $Y$ at a grid of points around $x_0$ by averaging the observation with a local kernel of size $\delta>0$. In a second step, the resulting values are aggregated
with a locally linear regression approach
in a neighborhood of size $h>0$ around $x_0$. Both steps are based on ``localization'', but they are very different in nature: Locally averaging $Y$ in order to estimate the state effectively filters out the static noise, and relates the parabolic nature of \eqref{eq:introduction:spde} to the intensity of $\varepsilon\dot V$. On the other hand, the local aggregation of various such state estimates optimizes the bias-variance tradeoff induced by the spatially heterogeneous nature of $\vartheta$. Interestingly, it seems to be necessary to perform such an iterated localization in order to achieve rate-optimality.
This two-step procedure is similar in taste to  pre-averaging approaches for estimating the volatility from financial time series under microstructure noise \cite{PodolskijVetter2009, JacodLiMyklandPodolskijVetter2009, Reiss2011}.
We derive in Theorem \ref{thm:abstract-rate}, Theorem \ref{thm:connect-conditions} 
and Remark \ref{rem:Scaling} 
that in dimension $d\geq 3$ under suitable conditions
\begin{align*}
	\hat\vartheta_\varepsilon(x_0) = \vartheta(x_0) + \cO_\bP\left(\left(\frac{\varepsilon}{\sigma}\right)^{\frac{(2+d)\beta}{4\beta+2d}}\right)
\end{align*}
holds, where $\beta$ is the H\"older regularity index of $\vartheta$.
In dimension $d\leq 2$ we are able to obtain Lipschitz rates.
We see that the presence of static noise inverts how the estimation error depends on some of the parameters in \eqref{eq:introduction:spde}, leading to (at first sight) counterintuitive scaling: Large dynamical noise intensity makes the estimator \emph{more} precise.
In addition, it is true that the relative estimation error \emph{increases} with the size of $\vartheta$. Both effects contrast well-known results for the stochastic heat equation (or even a scalar Ornstein--Uhlenbeck process) without static noise. Such a behaviour can be understood by scaling arguments, where $X$ and $Y$ are replaced by $\sigma^{-1}X, \sigma^{-1}Y$ (Remark \ref{rem:Scaling} below).
In order to explain basic ideas, we discuss the parametric case ($\vartheta$ constant) first (Theorem \ref{thm:parametric}), before giving a general picture for the nonparametric setup.

The proof of our results is based on reframing the localization appearing in the state estimation above into an inflation of the domain $\cD$ around $x_0$, until the domain finally approaches all of $\R^d$. In order to ensure convergence of the inflated state processes $X^{(\delta)}$ (to be defined properly in Section \ref{sec:localization}), we rely on refined approximation results for the heat semigroup (Theorem \ref{thm:semigroup-approximation-simplified}).
We complement our results with evidence from numerical simulations.
\\

In \cite{Kutoyants2019, KutoyantsZhou2021, Kutoyants2024}, parameter estimation problems for scalar processes under additional static noise have been considered in a small noise regime in the context of the K\'alm\'an--Bucy equations for stochastic filtering.
A related source detection problem is studied in \cite{Kutoyants2021-detector} under small noise.
Even without static noise, drift parameters of stochastic ordinary differential equations are not identified in finite time. This is a consequence of Girsanov's theorem. The situation changes drastically when we study stochastic partial differential equations (SPDEs).
It is our aim to explore the interplay between structural properties of SPDEs and the presence of static noise.
Nonparametric smoothing methods turn out to be much simpler and more explicit than stochastic filtering  while still attaining optimal convergence rates.

Statistical inference for
SPDEs has grown into a diverse field, see \cite{Cialenco2018} for a survey.
Starting from the fundamental observation that certain drift parameters of a parabolic SPDE can be identified in finite time \cite{HuebnerKhasminskiiRozovskii1993, HuebnerRozovskii1995}, different observation schemes related to spatial precision have been analyzed. This includes the spectral approach \cite{HuebnerKhasminskiiRozovskii1993, HuebnerRozovskii1995, Lototsky2009}, where a growing number of spatial eigenfrequencies of the generating process are observed, the discrete approach \cite{PospisilTribe2007, CialencoHuang2020, HildebrandtTrabs2021, HildebrandtTrabs2023}, where a set of point evaluations of the process are used, and the local approach \cite{AltmeyerReiss2021, AltmeyerCialencoPasemann2023, AltmeyerTiepnerWahl2022}, where local averages of the process with a kernel of shrinking bandwidth are observed.
More recently, a small diffusivity regime has been studied, which is linked by scaling properties to observing the process on a growing spatial domain \cite{GaudlitzReiss2022, Gaudlitz2024}.
In our work, the spatial heterogeneity of $\vartheta$ leads us naturally to employing the local approach.

A stochastic filtering problem for an underlying stochastic heat equation with unknown spatially varying diffusivity and finite dimensional noisy observations has been considered in a series of works \cite{AiharaSunahara1988, Aihara1992, Aihara1998}, with a focus on tracing the observed parts of the signal for large times. In contrast, we are interested in statistical properties of diffusivity estimation and work in finite time.

Other statistical models for diffusivity estimation that are considered in literature include elliptic partial differential equations with measurement noise \cite{AbrahamNickl2020, NicklVanDeGeerWang2020}.
In contrast to these works, in our model the underlying signal is not stationary.
\\

In Section \ref{sec:spde} we discuss the stochastic heat equation and the concept of space-time weak solutions.
Section \ref{sec:localization} is devoted to the concept of localization in time and space.
The statistical setting and the main results are presented in Section \ref{sec:Estimation}.
We first outline the parametric case ($\vartheta$ constant) in order to explain all relevant proof ideas in a simplified setting, before proceeding to the full nonparametric setup.
Section \ref{sec:semigroupapproximation} states the approximation results for heat semigroups.
Numerical examples are contained in Section \ref{sec:Numerics}.

\section{The Stochastic Heat Equation}\label{sec:spde}

Let $T>0$, $\cT=[0,T)$ and $\cD\subset\R^d$ be a bounded domain with smooth boundary.
In this work, we consider noisy observations $(Y_t)_{t\in\cT}$ of a solution $(X_t)_{t\in\cT}$ to the stochastic heat equation in $L^2(\cD)$, i.e.
\begin{align}
	\partial_t X_t &= \Delta_\vartheta X_t + \sigma\dot W_t, \label{eq:basic:SPDE} \\
	Y_t &= X_t + \varepsilon\dot V_t						 \label{eq:basic:Observation}
\end{align}
on $\cD$, with initial condition $X_0=\xi\in L^2(\cD)$ and Dirichlet boundary conditions.
The diffusivity function $\vartheta\in C^1(\bar\cD)$ is bounded away from zero, and $\Delta_\vartheta=\nabla\cdot\vartheta\nabla:W^{2,2}(\cD)\cap W^{1,2}_0(\cD)\rightarrow L^2(\cD)$ is the self-adjoint diffusion operator,
where $W^{2,2}(\cD)$ and $W^{1,2}_0(\cD)$ denote standard Sobolev spaces, see \cite{AdamsFournier2003}.
That operator generates the $C_0$-semigroup $(S(t))_{t\geq 0}$ on $L^2(\cD)$. More detailed conditions on $\vartheta$ will be given in Assumption {\assumptionB} below. $\dot W$ and $\dot V$ are independent isonormal Gaussian processes (see \cite{Dudley2002}) on $L^2(\cT\times\cD)$, formalizing the concept of space-time white noise, with intensities $\sigma, \varepsilon>0$.
We rely on $X$ being a solution to \eqref{eq:basic:SPDE} only in a \emph{space-time weak sense} on $\cC:=C^\infty_c(\cT\times\cD)$, i.e. for $\varphi\in\cC$:
\begin{align}\label{eq:basic:SPDE:spacetimeweak}
	-\sscp{X}{\dot\varphi} = \scp{\xi}{\varphi(0,\cdot)} + \sscp{X}{\Delta_\vartheta\varphi} + \sigma\sscpSmall{\dot W}{\varphi},
\end{align}
where $\sscp{\cdot}{\cdot}$ denotes the scalar product on $L^2(\cT\times\cD)$, and $\scp{\cdot}{\cdot}$ the scalar product on $L^2(\cD)$.
Equation \eqref{eq:basic:SPDE:spacetimeweak} arises from \eqref{eq:basic:SPDE} by means of formal integration by parts in space and time.
Note that the notation $\sscpSmall{\dot W}{\varphi}$ denotes evaluation of $\dot W$ at $\varphi\in L^2(\cT\times\cD)$, and that $\bE[\sscpSmall{\dot W}{\varphi}\sscpSmall{\dot W}{\psi}]=\sscp{\varphi}{\psi}$ for $\varphi,\psi\in L^2(\cT\times\cD)$, $\dot W$ being an isonormal Gaussian process, see \cite{Dudley2002}.

In order to construct such a solution,
we consider a formal variation of constants expression
\begin{align*}
	X_t = S(t)\xi + \sigma\int_0^tS(t-s)\dot W_s\diff s = S(t)\xi + \sigma\int_0^tS(t-s)\diff W_s
\end{align*}
of \eqref{eq:basic:SPDE}, where the integral is to be read in the It\^o sense \cite{DaPratoZabczyk2014}.
Motivated by this formula,
we \emph{define} $X$ to be a Gaussian process on $\cC$ with
\begin{align}
	\bE[\sscpSmall{X}{\varphi}] &= \sscp{S(\cdot)\xi}{\varphi}, \label{eq:field:Mean} \\
	\rCov(\sscpSmall{X}{\varphi}, \sscpSmall{X}{\psi})
		&= \frac{\sigma^2}{2}\int_0^\infty\int_0^\infty\int_{\abs{t-s}}^{t+s}\scp{\varphi(t, \cdot)}{S(r)\psi(s, \cdot)}\diff r\diff s\diff t. \label{eq:field:Cov}
\end{align}
Note that these terms are well-defined for $\varphi,\psi\in\cC$.

\begin{lemma}[well-posedness] \label{lem:FieldIsStweak}
	The process
	$X$
	defined by \eqref{eq:field:Mean}, \eqref{eq:field:Cov}
	is a space-time weak solution in the sense of \eqref{eq:basic:SPDE:spacetimeweak} for some isonormal Gaussian process $\dot W$ on $L^2(\cT\times\cD)$.
\end{lemma}

\begin{proof}
	See Appendix \ref{sec:proofs:spde}.
\end{proof}

The proof of this statement shows in particular that any mild solution in the sense of \cite{DaPratoZabczyk2014} is a space-time weak solution that satisfies \eqref{eq:field:Mean}, \eqref{eq:field:Cov}. Note, however, that for $d\geq 2$ a mild solution does not take values in $L^2(\cD)$.

\section{Localization}\label{sec:localization}

Our analysis relies heavily on localization techniques.
In this section, we provide a general picture.
Fix $x_0\in\cD$, and
define for $\delta>0$
\begin{align*}
	\cD_\delta = \{x\in\R^d:\delta x + x_0\in\cD\}\subset\R^d, \quad
	\cT_\delta = [0,\delta^{-2}T), \quad
	\cC_\delta = C^\infty_c(\cT_\delta\times\cD_\delta).
\end{align*}
Denote the scalar products on $L^2(\cD_\delta)$ and $L^2(\cT_\delta\times\cD_\delta)$ by $\scp{\cdot}{\cdot}_\delta$ and $\sscp{\cdot}{\cdot}_\delta$, i.e.
\begin{align*}
	\scp{f_1}{f_2}_\delta=\int_{\cD_\delta}f_1(x)f_2(x)\diff x,
	\quad\quad\quad
	\sscp{g_1}{g_2}_\delta = \int_0^{\delta^{-2}T}\int_{\cD_\delta}g_1(t, x)g_2(t, x)\diff x\diff t
\end{align*}
for $f_1,f_2\in L^2(\cD_\delta)$ and $g_1,g_2\in L^2(\cT_\delta\times\cD_\delta)$.
The localization mappings $(\cdot)_\delta:L^2(\cT_\delta\times\cD_\delta)\rightarrow L^2(\cT\times\cD)$ and $(\cdot)^{1/\delta}:L^2(\cT\times\cD)\rightarrow L^2(\cT_\delta\times\cD_\delta)$, given by
\begin{align*}
	\varphi_\delta(t, x) = \delta^{-1-d/2}\varphi(\delta^{-2}t, \delta^{-1} (x-x_0)), \quad
	\varphi^{1/\delta}(t, x) = \delta^{1+d/2}\varphi(\delta^2t, \delta x+x_0),
\end{align*}
are isometries.

\begin{remark}
	We also use space-only localization in the form $(\cdot)_\delta:L^2(\cD_\delta)\rightarrow L^2(\cD)$ and $(\cdot)^{1/\delta}:L^2(\cD)\rightarrow L^2(\cD_\delta)$, given by
	\begin{align*}
		\varphi_\delta(x) = \delta^{-d/2}\varphi(\delta^{-1}(x-x_0)), \quad
		\varphi^{1/\delta}(x) = \delta^{d/2}\varphi(\delta x + x_0).
	\end{align*}
	From the context it will be clear if a function defined on space and time or only on space will be localized, so there is no ambiguity.
\end{remark}

Define the localized diffusivity $\vartheta_\delta(x):=\vartheta(\delta x + x_0)$ on $\cD_\delta$, and the operator $\Delta_{\vartheta_\delta}=\nabla\cdot\vartheta_\delta\nabla$ on $W^{2,2}(\cD_\delta)$.
When restricted to $W^{2, 2}(\cD_\delta)\cap W^{1,2}_0(\cD_\delta)$, $\Delta_{\vartheta_\delta}$ generates a semigroup on $L^2(\cD_\delta)$, denoted by $(S_\delta(t))_{t\geq 0}$.

\begin{lemma}[localization of the semigroup] \label{lem:semigroup-scaling} \
		For $\varphi\in L^2(\cD_\delta)$ and fixed $t\geq 0$, we have $S(t)\varphi_\delta=(S_\delta(t\delta^{-2})\varphi)_\delta$.
\end{lemma}

\begin{proof}
	See Appendix \ref{sec:proofs:localization}.
\end{proof}

For an isonormal Gaussian process $\dot W$ on $L^2(\cT\times\cD)$ we define $\dot W^{1/\delta}$ on $L^2(\cT_\delta\times\cD_\delta)$ via
\begin{align*}
	\dot W^{1/\delta}(\varphi) = \dot W(\varphi_\delta).
\end{align*}
By isometry of the localization operation,
this is an isonormal Gaussian process, too.
The processes $X^{1/\delta}$ and $Y^{1/\delta}$ on $\cC_\delta$ are defined analogously.

\begin{lemma}[localization of the signal] \label{lem:localization-X}
The process $X^{(\delta)} = \delta^{-2}X^{1/\delta}$ solves
\begin{align}\label{eq:SHE:localized}
	\partial_t X^{(\delta)} = \Delta_{\vartheta_\delta} X^{(\delta)} + \sigma\dot W^{1/\delta}
\end{align}
on $\cC_\delta$ in the space-time weak sense, with initial condition $\xi^{(\delta)}=\delta^{-1}\xi^{1/\delta}\in L^2(\cD_\delta)$.
	Furthermore, we have for $\varphi, \psi\in \cC_\delta$:
	\begin{align}
		\bE[\sscpSmall{X^{(\delta)}}{\varphi}_\delta] &= \sscpSmall{S_\delta(\cdot)\xi^{(\delta)}}{\varphi}_\delta, \label{eq:mean-delta} \\
		\rCov(\sscpSmall{X^{(\delta)}}{\varphi}_\delta, \sscpSmall{X^{(\delta)}}{\psi}_\delta) &= \frac{\sigma^2}{2}\int_0^\infty\int_0^\infty\int_{\abs{t-s}}^{t+s}\scp{\varphi(t, \cdot)}{S_\delta(r)\psi(s, \cdot)}_\delta\diff r\diff s\diff t. \label{eq:covariance-delta}
	\end{align}
\end{lemma}

\begin{proof}
Use that $X$ is a space-time weak solution with initial condition $\xi$:
\begin{align*}
	-\sscpSmall{X^{1/\delta}}{\dot\varphi}_\delta &= -\sscp{X}{(\dot\varphi)_\delta} = -\delta^2\sscp{X}{\dot\varphi_\delta} \\
		&= \delta^2\scp{\xi}{\varphi_\delta(0,\cdot)} + \delta^2\sscp{X}{\Delta_{\vartheta}(\varphi_\delta)} + \delta^2\sigma\sscpSmall{\dot W}{\varphi_\delta} \\
		&= \delta\scp{\xi}{\varphi(0,\cdot)_\delta} + \sscp{X}{(\Delta_{\vartheta_\delta}\varphi)_\delta} + \delta^2\sigma\sscpSmall{\dot W}{\varphi_\delta} \\
		&= \delta\scpSmall{\xi^{1/\delta}}{\varphi(0,\cdot)}_\delta + \sscpSmall{X^{1/\delta}}{\Delta_{\vartheta_\delta}\varphi}_\delta + \delta^2\sigma\sscpSmall{\dot W^{1/\delta}}{\varphi}_\delta.
\end{align*}
Now multiply with $\delta^{-2}$.
The remaining statements are proven in Appendix \ref{sec:proofs:localization}.
\end{proof}

\begin{remark}\label{rem:localization-initialcondition} \
	Note that for the initial condition, $\xi^{(\delta)}(x)=\delta^{-1+d/2}\xi(\delta x + x_0)$.
	Thus in $d>2$, even if $\xi$ is continuous at $x_0$, one cannot assume that $\xi^{(\delta)}$ converges to $\xi(x_0)$ in a neighborhood of $x_0$ as $\delta\rightarrow 0$ in any reasonable mode of convergence on $\cC_\delta$.
\end{remark}

Now we consider the observation process:

\begin{lemma}[localization of the observation process]
The process $Y^{(\delta)}:=\delta^{-2}Y^{1/\delta}$ on $\cC_\delta$ satisfies
\begin{align}
	Y^{(\delta)} = X^{(\delta)} + \varepsilon\delta^{-2}\dot V^{1/\delta}.
\end{align}
\end{lemma}

\begin{proof}
We write
\begin{align*}
	\sscpSmall{Y^{1/\delta}}{\varphi}_\delta &= \sscp{Y}{\varphi_\delta}
		= \sscp{X}{\varphi_\delta} + \varepsilon\sscpSmall{\dot V}{\varphi_\delta} \\
		&= \sscpSmall{X^{1/\delta}}{\varphi}_\delta + \varepsilon\sscpSmall{\dot V^{1/\delta}}{\varphi}_\delta,
\end{align*}
and multiply this equation with $\delta^{-2}$.
\end{proof}

In particular, if $\delta = \sqrt{\varepsilon}$,
then the noise intensity in the former equation is
one.
Motivated by this observation, we shall use throughout
\begin{align}
	\delta=\sqrt{\varepsilon},
\end{align}
and based on this identification, we also write (by abuse of notation) all quantities appearing above in terms of $\varepsilon$, e.g. $X^{(\varepsilon)}$ for $X^{(\delta)}$ with $\delta=\sqrt{\varepsilon}$. \\

In order to define $X^{(\varepsilon)}$ in the limiting case $\varepsilon=\delta=0$, it is reasonable to set
$\cD_0:=\R^d$ and $\cT_0:=[0,\infty)$. Write $\scp{\cdot}{\cdot}_0$ for the scalar product on $L^2(\cD_0)$.
Let $\vartheta_0(x):=\vartheta(x_0)$ and $\Delta_{\vartheta_0}:=\nabla\cdot\vartheta_0\nabla=\vartheta_0\Delta$,
which generates the semigroup $(S_0(t))_{t\geq 0}$ on $L^2(\cD_0)$.
Considering Remark \ref{rem:localization-initialcondition}, we restrict to the case $X_0=0$ and write $\bar X^{(0)}$ for the limiting process, in alignment with the notation used in Appendix \ref{sec:Splitting} for processes starting at zero.
Extending \eqref{eq:covariance-delta}, this centered Gaussian process is determined by

\begin{align}\label{eq:covariance-zero}
	\bE[\sscpSmall{\bar X^{(0)}}{\varphi}_0\sscpSmall{\bar X^{(0)}}{\psi}_0] = \frac{\sigma^2}{2}\int_0^\infty\int_0^\infty\int_{\abs{t-s}}^{t+s}\scp{\varphi(t, \cdot)}{S_0(r)\psi(s, \cdot)}_0\diff r\diff s\diff t,
\end{align}
where $\sscpSmall{\bar X^{(0)}}{\cdot}_0$ refers to evaluation of $\bar X^{(0)}$ on the class of functions $\cC_0:=C^\infty_c(\cT_0\times\cD_0)$.

\section{Estimator and Main Results}\label{sec:Estimation}

We first motivate and construct our diffusivity estimator in Section \ref{sec:estimation:construction}.
In Section \ref{sec:estimation:parametric}, we give our results for the parametric case, and in Section \ref{sec:estimation:nonparametric} we treat the fully nonparametric case.
Here and in the sequel,
we write $a_\varepsilon\lesssim b_\varepsilon$ for $a_\varepsilon=\cO(b_\varepsilon)$.
$a_\varepsilon\sim b_\varepsilon$ is shorthand for $a_\varepsilon\lesssim b_\varepsilon$ and $b_\varepsilon\lesssim a_\varepsilon$, and
$a_\varepsilon\asymp b_\varepsilon$ denotes $a_\varepsilon/b_\varepsilon\rightarrow 1$.
When referring to convergence in probability of random variables, we write $A_\varepsilon\sim_\bP B_\varepsilon$ and $A_\varepsilon\asymp_\bP B_\varepsilon$.

\subsection{Construction of the Estimator}\label{sec:estimation:construction}

We are given a full realization of $Y$ on $\cC$.
Our goal is to infer the underlying diffusivity $\vartheta$ in a pointwise manner for $x_0\in\cD$.
In general, the laws of $Y$ for different $\vartheta$ are equivalent (in fact, always for $d\leq 5$ \cite{companionpaper}),
so we quantify the estimation error in terms of the noise level $\varepsilon$ as $\varepsilon\rightarrow 0$.
It is known (see \cite{HuebnerKhasminskiiRozovskii1993, HuebnerRozovskii1995, AltmeyerReiss2021}) that in the limit case $\varepsilon=0$, the diffusivity of a stochastic heat equation can be identified in finite time, i.e. the laws of $X$ for different $\vartheta$ are in fact singular.

\begin{remark}
	The static noise intensity $\varepsilon$ can be treated as known. In order to obtain $\varepsilon$ from $Y$, one may for example test $Y$ against any $\varphi\in L^2(\cD)$ of unit norm, and take the quadratic variation of the resulting process in time.
	This is the analogue of standard variance estimators in regression.
\end{remark}

A natural approach to estimation is to decompose the problem into two smaller ones: First, assuming that we had access to $X$ itself, we can construct a regression-type estimator for $\vartheta(x_0)$.
Working in a nonparametric setup, it is reasonable to consider various local averages of $X$ in a neighborhood of $x_0$ (in the sense of \cite{AltmeyerReiss2021}). For this we use a locally linear estimation approach in order to take account of the regularity of $\vartheta$.
The classical theory of this approach is given in \cite{Tsybakov2009}, and in \cite{StrauchTiepner2024}, it has been applied to a nonparametric advection estimation problem for a stochastic linear evolution equation.
The resulting estimator can be used as the basis for a plug-in approach, where the relevant functionals of $X$ are estimated from $Y$.
Let us note that regression-type estimators, coming from a least-squares rather than a maximum likelihood approach, have been used for inference for SPDEs with additive fractional noise \cite{MaslowskiTudor2013, KrizSnuparkova2022}.

As a motivation, consider the idealized case $\varepsilon=0$, where we observe $X$ directly.
We can (approximately) recover $\vartheta(x_0)$ from the regression problem
\begin{align}
	\sscp{\partial_t X}{\varphi_i} &= \sscp{\Delta_\vartheta X}{\varphi_i} + \sscpSmall{\sigma\dot W}{\varphi_i} \label{eq:estimator:regression} \\
		&\approx \vartheta(x_0)\sscp{\Delta X}{\varphi_i} + \sscpSmall{\sigma\dot W}{\varphi_i} \nonumber
\end{align}
for a set of test functions $(\varphi_i)_{i\in\cI}$ in $\cC$ with support close to $x_0$.
If the $\varphi_i$ have non-overlapping support and the same $L^2$-norm, then the $\sscpSmall{\sigma\dot W}{\varphi_i}$ form a sample of homoscedastic white noise.
Note, however, that the design $X$ depends on the noise $\dot W$ by construction, i.e. we have to tackle the endogeneity with an instrumental variable approach, see e.g. \cite{Lee2010}.
To this end, assume that we are given a second set $(\psi_i)_{i\in\cI}$ of test functions in $\cC$ such that $\sscp{\Delta X}{\psi_i}$ and $\sscpSmall{\sigma\dot W}{\varphi_i}$ are independent, while $\sscp{\Delta X}{\psi_i}$ and $\sscp{\Delta X}{\varphi_i}$ are closely correlated. In fact, $\psi_i$ will be constructed by shifting the support of $\varphi_i$ back in time, using that $X$ is an adapted process.
With these preparations, a natural regression estimator is
\begin{align}\label{eq:estimator:regression:referenceestimator}
	\tilde\vartheta_\cI(x_0) = \frac{\sum_{i\in\cI}\sscp{\Delta X}{\psi_i}\sscp{\partial_tX}{\varphi_i}}{\sum_{i\in\cI}\sscp{\Delta X}{\psi_i}\sscp{\Delta X}{\varphi_i}}.
\end{align}
Indeed, we use a modification of this regression estimator, defined below in \eqref{eq:estimator}.
Our test functions $(\varphi_i)_{i\in\cI}$ will be shifted and rescaled versions of a reference kernel $K$, with an additional time shift for the family $(\psi_i)_{i\in\cI}$.
Returning to our originial setting $\varepsilon>0$, the static noise induces additional error terms, which have to be handled by choosing the scaling of the kernel optimally.
Further, we will introduce weights in order to account for the regularity of $\vartheta$.

Fix some kernel $K\in C^{\infty}(\R\times\R^d)$ with compact support in $(0,1)\times(-1,1)^d$, and for $t\geq 0$ and $x\in\R^d$ set $K_{k, x}(t, y):=K(t-k, y-x)$.
We localize $K_{k, x}$ at $x_0$,
as described in Section \ref{sec:localization}:
\begin{align}\label{eq:kernel-localized-explicit}
	(K_{k, x})_{\varepsilon, x_0}(t, y) := \varepsilon^{-1/2-d/4}K_{k, x}(\varepsilon^{-1}t, \varepsilon^{-1/2}(y-x_0)).
\end{align}
$(K_{k, x})_{\varepsilon, x_0}$ has support in $(\varepsilon k,\varepsilon(k+1))\times(\varepsilon^{1/2}((-1,1)^d+x)+x_0)$.
We have $(K_{k, x})_{\varepsilon, x_0}\in\cC$ whenever its support is contained in $\cT\times\cD$.
In this case set
\begin{align*}
	\widehat X'_{\varepsilon, k, x} &:= \sscpSmall{\partial_tY}{(K_{k, x})_{\varepsilon, x_0}} = -\sscpSmall{Y}{\partial_t(K_{k, x})_{\varepsilon, x_0}}, \\
	\widehat X^\Delta_{\varepsilon, k, x} &:= \sscpSmall{\Delta Y}{(K_{k, x})_{\varepsilon, x_0}} = \sscpSmall{Y}{\Delta (K_{k, x})_{\varepsilon, x_0}}.
\end{align*}
These are estimators for
$\sscpSmall{\partial_t X}{(K_{k, x})_{\varepsilon, x_0}}$ as well as $\sscpSmall{\Delta X}{(K_{k, x})_{\varepsilon, x_0}}$, where $X$ is approximated by $Y$.
Figure \ref{fig:numerics:trajectory} visualizes the parabolically scaled local averages.
The simulation parameters  are identical to those of Figure \ref{fig:numerics:results} (left). As discussed in the introduction, for discrete observations from the regression model \eqref{eq:introduction:regrmodel} the statistics $\widehat X'_{\varepsilon, k, x}$ and $\widehat X^\Delta_{\varepsilon, k, x}$ are calculated by the corresponding Riemann sums on the right-hand side.
We will consider $\widehat X'_{\varepsilon, k, x},\widehat X^\Delta_{\varepsilon, k, x}$ simultaneously for
\begin{enumerate}
	\item[a)] $0\leq k\leq N_\varepsilon$, where $N_\varepsilon := \lfloor T\varepsilon^{-1}\rfloor - 1$,
	\item[b)] $x\in\cX_\varepsilon$, where $\cX_\varepsilon\subset\R^d$ is a finite set of points such that
	\begin{enumerate}
		\item[(i)] $\mathrm{supp}((K_{0, x})_{\varepsilon, x_0})\subset\cT\times\cD$ for $x\in\cX_\varepsilon$,
		\item[(ii)] $\mathrm{supp}((K_{0, x})_{\varepsilon, x_0})\cap \mathrm{supp}((K_{0, y})_{\varepsilon, x_0})=\emptyset$ for $x\neq y\in\cX_\varepsilon$.
	\end{enumerate}
\end{enumerate}

\begin{remark}
	Clearly, the size $\abs{\cX_\varepsilon}$ of the grid is at most of the order $\delta^{-d}=\varepsilon^{-d/2}$.
\end{remark}

\begin{figure}
	\includegraphics[width=0.32\textwidth]{"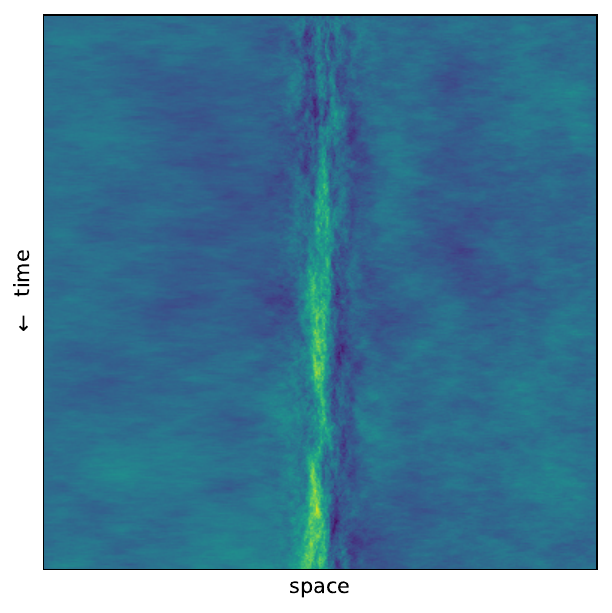"}
	\includegraphics[width=0.32\textwidth]{"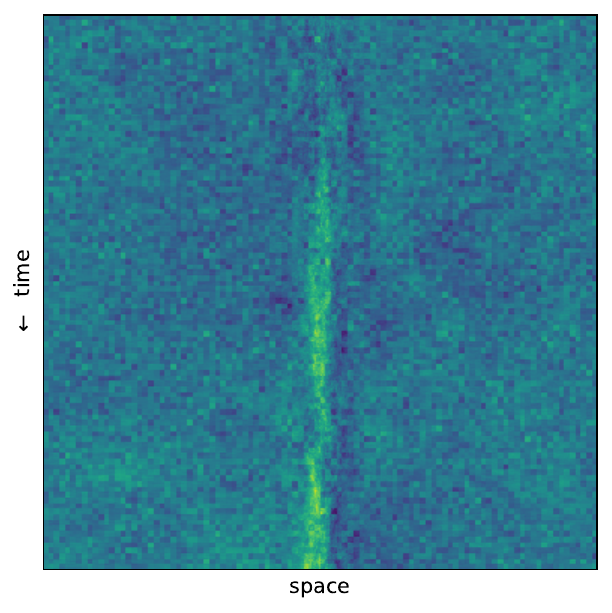"}
	\includegraphics[width=0.32\textwidth]{"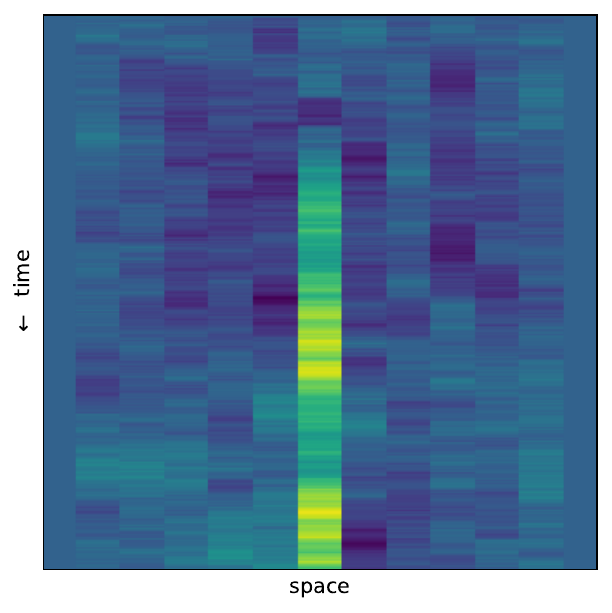"}
	\caption{
	\textbf{Left:} Realization of a stochastic heat equation with spatially heterogeneous diffusivity.
	\textbf{Center:} Same trajectory, but with additional static noise ($\varepsilon=0.16$) that is locally averaged for better visibility.
	\textbf{Right:} Smoothed trajectory with static noise, obtained by testing with $(K_{k, x})_{\varepsilon, x_0}$ for $x_0=0.5$ and $\varepsilon=0.0016$, i.e. $\delta=0.04$.
	}\label{fig:numerics:trajectory}
\end{figure}

Motivated by \eqref{eq:estimator:regression:referenceestimator}, we use a locally linear-type estimator with weights $w_\varepsilon^h(x)=w_\varepsilon^h(x,x_0, \cX_\varepsilon)$ for $x\in\cX_\varepsilon$, depending on an additional bandwidth parameter $h>0$.
Precise conditions on the weights are given below.
The standing assumption is
\begin{align}
	\delta = \cO(h) \;\Leftrightarrow\;
	\varepsilon = \cO(h^2),
\end{align}
i.e. the bandwidth $h$ is larger than the spatial precision $\delta=\varepsilon^{1/2}$ of a single kernel.
Putting things together, the formal regression problem \eqref{eq:estimator:regression} leads to a least squares estimator of the form
\begin{align}\label{eq:estimator}
	\hat\vartheta_\varepsilon(x_0) = \frac{\sum_{k=1}^{N_\varepsilon}\sum_{x\in\cX_\varepsilon}w_\varepsilon^h(x)\widehat X^\Delta_{\varepsilon, k-1, x}\widehat X'_{\varepsilon, k, x}}{\sum_{k=1}^{N_\varepsilon}\sum_{x\in\cX_\varepsilon}w_\varepsilon^h(x)\widehat X^\Delta_{\varepsilon, k-1, x}\widehat X^\Delta_{\varepsilon, k, x}}
\end{align}
where $\widehat X^\Delta_{\varepsilon, k-1, x}$ serves as an instrumental variable for $\widehat X^\Delta_{\varepsilon, k, x}$.

\begin{remark} \
	\begin{enumerate}
		\item[a)] Heuristically, the numerator of $\hat\vartheta_\varepsilon(x_0)$ has the form of a discrete stochastic integral of $\widehat X^\Delta_{\varepsilon, k, x}$ with respect to increments of $\widehat X_{\varepsilon, k, x}:=\sscp{Y}{(K_{k, x})_{\varepsilon, x_0}}$ (this is made rigorous in Lemma \ref{lem:errordecomposition}). Note, however, that $\widehat X'_{\varepsilon, k, x}$ behaves differently from such increments due to local averaging in time, which would lead to correlation with the stochastic integrand. In order to maintain the structure of a martingale, we introduce an artificial time shift and evaluate the integrand at $k-1$ instead of $k$ in the definition of $\hat\vartheta_\varepsilon(x_0)$.
		\item[b)] Interestingly, there are two different localization parameters active in \eqref{eq:estimator}: The precision of the testing kernel is bound to the parabolic scaling of the stochastic heat equation (thus relating the noise level $\varepsilon$ to local length units in space and time, see Section \ref{sec:localization}), whereas the bandwidth $h$ of the weights defines a neighborhood that is informative for estimating $\vartheta(x_0)$.
		Loosely speaking, our estimator is the ratio of ``averages of squared averages'',
		and thus conceptually related to  pre-averaging approaches in volatility estimation \cite{PodolskijVetter2009, JacodLiMyklandPodolskijVetter2009,Reiss2011}.
	\end{enumerate}
\end{remark}

\begin{lemma}[error decomposition] \label{lem:errordecomposition}
	The estimation error can be written as
	\begin{align}
		\hat\vartheta_\varepsilon(x_0) - \vartheta(x_0)
			&= M_{\varepsilon, N_\varepsilon}/I_{\varepsilon, N_\varepsilon} + B_{\varepsilon, N_\varepsilon}/I_{\varepsilon, N_\varepsilon},
	\end{align}
	where
	\begin{align*}
		I_{\varepsilon, N_\varepsilon} &= \sum_{k=1}^{N_\varepsilon}\sum_{x\in\cX_\varepsilon}w_\varepsilon^h(x)\widehat X^\Delta_{\varepsilon, k-1, x}\widehat X^\Delta_{\varepsilon, k, x}, \\
		B_{\varepsilon, N_\varepsilon} &= \sum_{k=1}^{N_\varepsilon}\sum_{x\in\cX_\varepsilon}w_\varepsilon^h(x)\widehat X^\Delta_{\varepsilon, k-1, x}\sscpSmall{[\Delta_{\vartheta}-\Delta_{\vartheta(x_0)}] X}{(K_{k, x})_{\varepsilon, x_0}},
	\end{align*}
	and $(M_{\varepsilon, N})_{N=0,\ldots N_\varepsilon}$ is a time-discrete $L^2$-martingale for fixed $\varepsilon$ with quadratic variation
	\begin{align*}
		\langle M_\varepsilon\rangle_N = \sigma_K^2\sum_{k=1}^{N}\sum_{x\in\cX_\varepsilon}w_\varepsilon^h(x)^2(\widehat X^\Delta_{\varepsilon, k-1, x})^2,
	\end{align*}
	where $\sigma_K^2=\sigma^2\norm{K}_{L^2(\R^d)}^2+\norm{\partial_tK+\Delta_{\vartheta(x_0)} K}_{L^2(\R^d)}^2$.
\end{lemma}

\begin{proof}
	See Appendix \ref{sec:proofs:estimation}.
\end{proof}

In this decomposition, $I_{\varepsilon, N_\varepsilon}$ plays the role of an empirical Fisher information, $M_{\varepsilon, N_\varepsilon}$ captures the stochastic error and $B_{\varepsilon, N_\varepsilon}$ is the bias induced by spatial heterogeneity of $\vartheta$.

For later use, we define the $\Delta$-order of a kernel:

\begin{definition}
	$K\in C_c^\infty(\R\times\R^d)$ is \textit{of $\Delta$-order $k\in\N$} if $K = (-\Delta)^k\bar K$ for some $\bar K\in C_c^\infty(\R\times\R^d)$.
\end{definition}

\subsection{Parametric Theory}\label{sec:estimation:parametric}

It is useful to consider the parametric case first, as most ideas appear already here. To this end, assume that $\vartheta(x)\equiv\vartheta$ is constant, in particular $B_{\varepsilon, N_\varepsilon}=0$.
For simplicity of presentation, we restrict to $X_0=0$ in this section, as well as $x_0=0$, assuming $0\in\cD$.
Abbreviate $\hat\vartheta_\varepsilon=\hat\vartheta_\varepsilon(x_0)$.
As the bandwidth of the weights only affects nonparametric estimation, we assume $h = 1$, i.e. the $w_\varepsilon^h(x)$ do in fact not depend on $h$, and we write $w_\varepsilon(x)=w^h_\varepsilon(x)$.
\\

\noindent\textbf{Assumption $(W)$:} The weights $w_\varepsilon(x)$ satisfy
	\begin{enumerate}
		\item[a)] $w_\varepsilon(x)\geq 0$ and $\lim_{\varepsilon\rightarrow 0}\sum_{x\in\cX_\varepsilon}w_\varepsilon(x) = 1$,
		\item[b)] there are closed balls $\cB\subset\cB^*$ in $\cD$
		around $x_0=0$ such that
		$\mathrm{supp}(w_\varepsilon)\subseteq\varepsilon^{-1/2}\cB$,
		and
		$\mathrm{supp}((K_{0, x})_{\varepsilon, x_0})\subseteq\cT\times\cB^*$ for all
		$x\in \mathrm{supp}(w_\varepsilon)$.
	\end{enumerate}
	
The second condition says that we evaluate $Y$ away from the boundary of the domain. A specific and natural choice is given by $w_\varepsilon(x)\equiv 1/\abs{\cX_\varepsilon}$ for some grid $\cX_\varepsilon$ such that $\varepsilon^{1/2}x\in\cB$ and  $\mathrm{supp}((K_{0, x})_{\varepsilon, x_0})\subseteq\cT\times\cB^*$ for $x\in\cX_\varepsilon$.
As the grid is determined by the statistician, this imposes no restriction.
As a last preparation, write
\begin{align}\label{eq:limitingconstant}
	C_\infty(\varphi, \psi) = \frac{\sigma^2}{2}\int_0^\infty\int_0^\infty\int_{\abs{t-s}}^{\infty}\scp{\varphi(t,\cdot)}{S_0(r)\psi(s,\cdot)}_0\diff r\diff s\diff t,\quad\quad \varphi,\psi\in\cC_0.
\end{align}
This is finite if $\varphi$ or $\psi$ have $\Delta$-order one by Lemma \ref{lem:bound:semigroup:heatkernel}.

\begin{theorem}[parametric central limit theorem] \label{thm:parametric}
	Assume $X_0=0$. If
	Assumption $(W)$ holds and
	$C_\infty(\Delta K, \Delta K_{-1, 0})>0$,
	then
	\begin{align}
		R_\varepsilon^{-1}(\hat\vartheta_\varepsilon-\vartheta)\xrightarrow{d}\cN(0, 1),
	\end{align}
	where $R_\varepsilon := \langle M_\varepsilon\rangle_{N_\varepsilon}^{1/2}/I_{\varepsilon, N_\varepsilon}$ satisfies
	with $\cW_\varepsilon:=\sum_{x\in\cX_\varepsilon} w_\varepsilon(x)^2$:
	\begin{align}
		R_\varepsilon \asymp_\bP \sigma_K\frac{\sqrt{C_\infty(\Delta K, \Delta K) + \norm{\Delta K}_{L^2}^2}}{C_\infty(\Delta K, \Delta K_{-1, 0})}N_\varepsilon^{-1/2}\cW_\varepsilon^{1/2}.
	\end{align}
	In particular, if $w_\varepsilon(x)\equiv 1/\abs{\cX_\varepsilon}$ and $\abs{\cX_\varepsilon}\asymp C_\cX\varepsilon^{-d/2}$ for some $C_\cX>0$, then
	\begin{align}\label{eq:rate:parametric}
		\varepsilon^{-\frac{1}{2}-\frac{d}{4}}\left(\hat\vartheta_\varepsilon-\vartheta\right)\xrightarrow{d}\cN\left(0, \sigma_K^2\frac{C_\infty(\Delta K, \Delta K) + \norm{\Delta K}_{L^2}^2}{T C_\cX C_\infty(\Delta K, \Delta K_{-1, 0})^2}\right).
	\end{align}
\end{theorem}

\begin{proof}
	It suffices to show
	\begin{align}
		\bE I_{\varepsilon, N_\varepsilon} &\asymp C_\infty(\Delta K, \Delta K_{-1, 0})N_\varepsilon, \nonumber\\
		\bE\langle M_\varepsilon\rangle_{N_\varepsilon} &\asymp \sigma_K^2(C_\infty(\Delta K, \Delta K) + \norm{\Delta K}_{L^2}^2)N_\varepsilon\cW_\varepsilon, \label{eq:parametric:proof:moments} \\
			\rVar(I_{\varepsilon, N_\varepsilon}) &= o(N_\varepsilon^2)\;\;\;\mathrm{and}\;\;\; \rVar(\langle M_\varepsilon\rangle_{N_\varepsilon}) = o\left(N_\varepsilon^2\cW_\varepsilon^2\right). \nonumber
	\end{align}
In fact, the claim follows from Lemma \ref{lem:errordecomposition} and Slutsky's lemma with a standard central limit theorem for triangular martingale schemes in $N_\varepsilon,\varepsilon$  (see e.g. \cite[Theorem VII.8.4]{Shiryaev1996}), noting 
\[N_\varepsilon^{-1}I_{\varepsilon, N_\varepsilon}\xrightarrow{P} C_\infty(\Delta K, \Delta K_{-1, 0}), \quad (N_\varepsilon \cW_\varepsilon)^{-1}\langle M_\varepsilon\rangle_{N_\varepsilon}\xrightarrow{P} \sigma_K^2(C_\infty(\Delta K, \Delta K) + \norm{\Delta K}_{L^2}^2)
\]
as a consequence of \eqref{eq:parametric:proof:moments}.
In Lemma \ref{lem:conditionallyapunov}, we prove a Lyapunov condition, which implies the conditional Lindeberg condition needed in order to apply the central limit theorem for triangular martingale schemes \cite{GaensslerStrobelStute78}.

	\begin{enumerate}
	\item[a)] \emph{Asymptotic expansion for the expected values.}\\
	Using $\Delta(K_{k-1, x})_{\varepsilon, x_0}=\varepsilon^{-1}(\Delta K_{k-1, x})_{\varepsilon, x_0}$,
	\begin{align*}
		\bE\langle M_\varepsilon\rangle_{N_\varepsilon}
			&= \sigma_K^2\sum_{k=1}^{N_\varepsilon}\sum_{x\in\cX_\varepsilon}w_\varepsilon(x)^2\left(\bE\sscpSmall{X}{\Delta(K_{k-1, x})_{\varepsilon, x_0}}^2 \right. \\
			&\left.\hspace{5cm}+ \varepsilon^2\bE\sscpSmall{\dot V}{\Delta(K_{k-1, x})_{\varepsilon, x_0}}^2\right) \\
			&= \sigma_K^2\sum_{k=1}^{N_\varepsilon}\sum_{x\in\cX_\varepsilon}w_\varepsilon(x)^2\left(\varepsilon^{-2}\bE\sscpSmall{X}{(\Delta K_{k-1, x})_{\varepsilon, x_0}}^2 + \norm{(\Delta K_{k-1, x})_{\varepsilon, x_0}}^2\right) \\
			&= \sigma_K^2\norm{\Delta K}_{L^2(\R\times\R^d)}^2N_\varepsilon\cW_\varepsilon + \sigma_K^2\sum_{k=1}^{N_\varepsilon}\sum_{x\in\cX_\varepsilon}w_\varepsilon(x)^2\bE\sscpSmall{X^{(\varepsilon)}}{\Delta K_{k-1, x}}_\varepsilon^2,
	\end{align*}
	where $X^{(\varepsilon)}$ is defined in Lemma \ref{lem:localization-X}.
	As $\varepsilon\rightarrow 0$, this process approaches the space-time weak solution $\bar X^{(0)}$ to the stochastic heat equation on $[0,\infty)\times\R^d$ (defined in \eqref{eq:covariance-zero}), in the sense that the covariance converges uniformly in $x\in\cX_\varepsilon$, $k\leq N_\varepsilon$ (see Lemma \ref{lem:cov:conv:eps}). Thus
	\begin{align*}
		\sum_{k=1}^{N_\varepsilon}\sum_{x\in\cX_\varepsilon}w_\varepsilon(x)^2\bE\sscpSmall{X^{(\varepsilon)}}{\Delta K_{k-1, x}}_\varepsilon^2 \hspace{-2cm}& \\
			&= \sum_{k=1}^{N_\varepsilon}\sum_{x\in\cX_\varepsilon}w_\varepsilon(x)^2\left(\bE\sscpSmall{\bar X^{(0)}}{\Delta K_{k-1, x}}_0^2 + o(1)\right) \\
			&= \left(C_\infty(\Delta K, \Delta K) + o(1)\right)N_\varepsilon\cW_\varepsilon,
	\end{align*}
	where the ergodic limit $C_\infty(\Delta K, \Delta K)$ is identified by Lemma \ref{lem:cov:conv:T}.
	The calculation for $\bE I_{\varepsilon, N_\varepsilon}$ is identical, but easier as there is no term involving the static noise $\dot V$.
	
	\item[b)] \emph{Variance bounds.}\\
	Concerning $\rVar(\langle M_\varepsilon\rangle_{N_\varepsilon})$,
	write $\bar X^\Delta_{\varepsilon, k, x}:=\sscpSmall{X}{\Delta(K_{k, x})_{\varepsilon, x_0}}$ and $\dot V^\Delta_{\varepsilon, k, x}:=\sscpSmall{\varepsilon\dot V}{\Delta(K_{k, x})_{\varepsilon, x_0}}$.
	Then 
	\begin{align*}
		\rVar(\langle M_\varepsilon\rangle_{N_\varepsilon})
			&= \rVar\left(\sigma_K^2\sum_{k=1}^{N_\varepsilon}\sum_{x\in\cX_\varepsilon}w_\varepsilon(x)^2(\widehat X^\Delta_{\varepsilon, k-1, x})^2\right) \\
			&\hspace{-2cm} \leq 8\sigma_K^4\left[\rVar\left(\sum_{k=1}^{N_\varepsilon}\sum_{x\in\cX_\varepsilon}w_\varepsilon(x)^2(\bar X^\Delta_{\varepsilon, k-1, x})^2\right) + \rVar\left(\sum_{k=1}^{N_\varepsilon}\sum_{x\in\cX_\varepsilon}w_\varepsilon(x)^2(\dot V^\Delta_{\varepsilon, k-1, x})^2\right)\right] \\
			&\hspace{-2cm} =: 8\sigma_K^4(I + II)
	\end{align*}
	follows from resolving the square of $\widehat X^\Delta_{\varepsilon, k-1, x}$, then bounding the variance of the sum of the resulting four terms by four times the sum of the variances, using Lemma \ref{lem:abstract-variance-bounds} and grouping the terms together.
	For the first term $I$, we exploit the decorrelation in time of $\bar X^\Delta$. More precisely, Lemma \ref{lem:cov:bound:kl} yields
	for any $\kappa<d/2$:
	\begin{align*}
		I &= 2\sum_{k,\ell=1}^{N_\varepsilon}\sum_{x, y\in\cX_\varepsilon}w_\varepsilon(x)^2w_\varepsilon(y)^2\bE[\bar X^\Delta_{\varepsilon, k-1, x}\bar X^\Delta_{\varepsilon, \ell-1, y}]^2
			\lesssim \cW_\varepsilon^2\sum_{k,\ell=1}^{N_\varepsilon}(1\wedge\abs{k-\ell}^{-\kappa})^2.
	\end{align*}
	With $\kappa>1/2$, this is $\cO(\cW_\varepsilon^2N_\varepsilon)$ in $d\geq 2$. In $d=1$, this is $\cO(\cW_\varepsilon^2N_\varepsilon^{1+\eta})$ for any $\eta>0$.
	In both cases, $I=o(\cW_\varepsilon^2N_\varepsilon^2)$, as desired. A direct evaluation shows
	\begin{align*}
		II &= 2\sum_{k,\ell=1}^{N_\varepsilon}\sum_{x, y\in\cX_\varepsilon}w_\varepsilon(x)^2w_\varepsilon(y)^2\bE[\dot V^\Delta_{\varepsilon, k-1, x}\dot V^\Delta_{\varepsilon, \ell-1, y}]^2
			= 2\norm{\Delta K}_{L^2}^2N_\varepsilon\sum_{x\in\cX_\varepsilon}w_\varepsilon(x)^4,
	\end{align*}
	which is $\cO(N_\varepsilon\cW_\varepsilon^2)$.
	In total, $I+II = o(N_\varepsilon^2\cW_\varepsilon^2)$.
	Again, the argument for $\rVar(I_{\varepsilon, N_\varepsilon})$ is identical.
	\end{enumerate}
\end{proof}

\begin{remark} \label{rem:parametric} \
	\begin{enumerate}
		\item[a)] If we define the effective sample size $N_\mathrm{eff}$ to be the number of spatiotemporal shifts of the localized testing kernel $K$ that can be used for estimation, parabolic scaling yields $N_\mathrm{eff}\sim\delta^{-d}\times\delta^{-2}\sim\varepsilon^{-d/2-1}$, so Theorem \ref{thm:parametric} says that $\vartheta$ can be identified at rate $1/\sqrt{N_\mathrm{eff}}$.
		\item[b)] The parametric rate in \eqref{eq:rate:parametric} is optimal, see \cite{companionpaper}.
		\item[c)] The central limit theorem from Theorem \ref{thm:parametric} can be used to construct confidence intervals in a standard way, by substituting all unknown quantities in the asymptotic variance by consistent estimates.
		\item[d)] It is straightforward to extend the argument in order to understand the dependence of the asymptotic result on the dynamic noise level $\sigma$, assuming this is known.
		Letting $\delta=\sqrt{\varepsilon/\sigma}$ and repeating the proof of Theorem \ref{thm:parametric}, we obtain a central limit theorem with convergence rate $(\varepsilon/\sigma)^{1/2+d/4}$, with an asymptotic variance independent of $\sigma$.
		In fact, this seemingly more general case can be reduced to the present one by rescaling both $X$ and $Y$ with $\sigma^{-1}$, see Remark \ref{rem:Scaling} for a discussion. 
		\item[e)] Even if we consider the case of continuous observations, where the statistician can always choose $\delta=\sqrt{\varepsilon}$, it is illustrative to consider the effects of other relations to the proof of Theorem \ref{thm:parametric}. If $\delta=o(\sqrt{\varepsilon})$, the effect of the static noise on $\bE\langle M_\varepsilon\rangle_{N_\varepsilon}$ becomes dominant and grows faster than $N_\varepsilon\cW_\varepsilon$. As a consequence, $\hat\vartheta_\varepsilon$ will converge more slowly. If $\varepsilon=o(\delta^2)$, then the analysis remains the same if we formally replace $\varepsilon$ by $\delta^2$. In this case, the estimator does not optimally exploit the static noise level present in the data.
		\item[f)] The dependence of the asymptotic variance on $\vartheta$ is non-trivial due to the form of $\sigma_K^2$ and $C_\infty(\cdot, \cdot)$.
		\item[g)] The approximation of the covariances of $X^{(\varepsilon)}$ and $\bar X^{(0)}$ in Lemma \ref{lem:cov:conv:eps}, which has been used in the proof, relies on a Trotter--Kato type semigroup approximation result. This is discussed in Section \ref{sec:semigroupapproximation}.
		\item[h)]
		The separation of the variances of signal and noise is possible due to a general property of Gaussian processes, as discussed in Lemma \ref{lem:abstract-variance-bounds}.
		\item[i)]
		Integrating out the heat semigroup, we can write in \eqref{eq:limitingconstant}
		\begin{align*}
			C_\infty(\Delta K, \Delta K)
				&= -\frac{\sigma^2}{2\vartheta(x_0)}\int_0^\infty\int_0^\infty\scp{K(t, \cdot)}{S_0(\abs{t-s})\Delta K(s, \cdot)}_0\diff s\diff t \\
				&= \frac{\sigma^2}{2\vartheta(x_0)}\int_0^\infty\int_0^\infty\scp{\nabla K(t, \cdot)}{S_0(\abs{t-s}) \nabla K(s, \cdot)}_{L^2(\R^d)^d}\diff s\diff t.
		\end{align*}
		As we consider kernels in space and time, this can be considered as a temporally smoothed version of the constant $\Psi(\Delta K, \Delta K)$ appearing in \cite[Section 3.2]{AltmeyerReiss2021}.
	\end{enumerate}
\end{remark}

\subsection{Nonparametric Theory}\label{sec:estimation:nonparametric}

We first give general conditions that imply natural nonparametric rates, and in a second step we examine to what extent these conditions can be verified.
More precisely, in analogy to the moment conditions \eqref{eq:parametric:proof:moments} in the proof of Theorem \ref{thm:parametric}, the consistency of $\hat\vartheta_\varepsilon(x_0)$ follows if the following natural moment bounds are satisfied with some parameter $\NonparIndex\geq 1$, which quantifies the size of the bias:
\\

\noindent\textbf{Assumption $(A_\NonparIndex)$:}
For $h\sim\varepsilon^{\frac{2+d}{4\NonparIndex+2d}}$, it holds:
\begin{enumerate}
	\item[a)]
	$\bE I_{\varepsilon, N_\varepsilon}\sim N_\varepsilon$ and $\rVar(I_{\varepsilon, N_\varepsilon}) = o(N_\varepsilon^2)$,
	\item[b)]
	$\bE\langle M_\varepsilon\rangle_{N_\varepsilon} = \cO(N_\varepsilon\abs{\cX_\varepsilon}^{-1}h^{-d})$,
	\item[c)]
	$\bE B_{\varepsilon, N_\varepsilon}^2 =\cO(h^{2\NonparIndex} N_\varepsilon^2)$.
\end{enumerate}

\begin{theorem}[nonparametric rates in general] \label{thm:abstract-rate}
	Let Assumption $(A_\NonparIndex)$ hold for some $\NonparIndex\geq 1$. Then
	\begin{align}\label{eq:thm:abstract:generalrate}
		\hat\vartheta(x_0) - \vartheta(x_0) = \cO_\bP\left(\frac{1}{\sqrt{N_\varepsilon\abs{\cX_\varepsilon}h^d}} + h^\NonparIndex\right).
	\end{align}
	In particular,
	with $N_\varepsilon\sim\varepsilon^{-1}$ and
	$\abs{\cX_\varepsilon}\sim\varepsilon^{-d/2}$,
	we obtain
	\begin{align}
		\hat\vartheta(x_0) - \vartheta(x_0) = \cO_\bP\left(\varepsilon^{1/2+d/4}h^{-d/2}\right) + \cO_\bP\left(h^\NonparIndex\right).
	\end{align}
	With the choice $h\sim\varepsilon^{\frac{2+d}{4\NonparIndex+2d}}$, this expression is optimized, and
	\begin{align}\label{eq:thm:abstract:finalrate}
		\hat\vartheta_\varepsilon(x_0) = \vartheta(x_0) + \cO_\bP\left(\varepsilon^{\frac{(2+d)\NonparIndex}{4\NonparIndex+2d}}\right).
	\end{align}
\end{theorem}

\begin{proof}
	Assumption $(A_\NonparIndex)$ a)
	implies $I_{\varepsilon, N_\varepsilon}\sim_\bP N_\varepsilon$.
	Next, since $(M_{\varepsilon,N})_{N\geq 0}$ is a centered martingale, $\bE[M_{\varepsilon, N_\varepsilon}^2]=\bE[\langle M_\varepsilon\rangle_{N_\varepsilon}]$, so $M_{\varepsilon, N_\varepsilon} = \cO_{L^2}(N_\varepsilon^{1/2}\abs{\cX_\varepsilon}^{-1/2}h^{-d/2})$ by Assumption $(A_\NonparIndex)$ b). The claim now follows from Lemma \ref{lem:errordecomposition}, using Assumption $(A_\NonparIndex)$ c) for the remaining term.
\end{proof}

\begin{remark} \
	\begin{enumerate}
		\item[a)] The rate in \eqref{eq:thm:abstract:finalrate} is indeed the optimal rate if $\NonparIndex$ is the H\"older regularity of $\vartheta$, see \cite{companionpaper}.
		\item[b)] Formally letting $\NonparIndex\rightarrow\infty$ in Theorem \ref{thm:abstract-rate}, we obtain the parametric rate $\cO_\bP(\varepsilon^{1/2+d/4})$ from Theorem \ref{thm:parametric}, together with $h\sim 1$.
		\item[c)] For $\NonparIndex=1$, the optimal bandwidth is just $h\sim\varepsilon^{1/2}$, and no additional smoothing is needed.
		For $\NonparIndex<1$, the optimal bandwidth would be smaller than the precision of the test kernel, $h=o(\varepsilon^{1/2})$, which seems unreasonable.
		Indeed, lower bounds from \cite{companionpaper} suggest an ellbow effect in the rate even on the level of statistical experiments for $\NonparIndex<1$, such that different tools have to be applied in that case.
		\item[d)] It is straightforward to derive rates from \eqref{eq:thm:abstract:generalrate} for the situation that the evaluation points in space are distributed only along some submanifold of dimension $d'$, such that $\abs{\cX_\varepsilon}\sim\varepsilon^{-d'/2}=o(\varepsilon^{-d/2})$. This result is of independent interest. We emphasize, however, that our goal is to study the case of continuous observation $Y$ in space and time, in particular, the statistician can choose the grid $\cX_\varepsilon$, and in order to maximize the information, this will be done as in Theorem \ref{thm:abstract-rate}.
	\end{enumerate}
\end{remark}

Next, we give verifiable conditions that imply Assumption $(A_\NonparIndex)$.
We will use for $1<\beta<2$ and $0<\underline{\vartheta}<C$:
\\

\noindent\textbf{Assumption {\assumptionB}:}
\begin{enumerate}
	\item[a)] $\vartheta\in \Theta(\beta, \underline{\vartheta}, C)$,
	where $\Theta(\beta, \underline{\vartheta}, C)=\{\theta\in C^1(\bar\cD)\;|\;\underline\vartheta\leq\vartheta(x)\leq C,\; \abs{\nabla\vartheta(x)}\leq C\;\mathrm{and}\;\abs{\nabla\vartheta(x)-\nabla\vartheta(y)}\leq C\abs{x-y}^{\beta-1}\;\mathrm{for}\;x,y\in\bar\cD\}$.
	\item[b)] The weights satisfy:
	\begin{enumerate}
		\item[(i)] $w^h_\varepsilon(x)=0$ if $\abs{x}>Ch\varepsilon^{-1/2}$,
		\item[(ii)] $\sup_{x\in\cX_\varepsilon}\abs{w^h_\varepsilon(x)} \leq C\abs{\cX_\varepsilon}^{-1}h^{-d}$,
		\item[(iii)] $\sum_{x\in\cX_\varepsilon}\abs{w^h_\varepsilon(x)} \leq C$,
		\item[(iv)] $\sum_{x\in\cX_\varepsilon}w^h_\varepsilon(x) = 1$, and $\sum_{x\in\cX_\varepsilon}x_i w^h_\varepsilon(x) = 0$ for $i=1,\dots,d$.
	\end{enumerate}
	\item[c)] The kernel $K$ is of $\Delta$-order $2$.
	In addition, $K$ is an odd or even function in all coordinate directions $x_i$, $i=1,\dots,d$,
	and $C_\infty(\Delta K, \Delta K_{-1, 0})>0$.
\end{enumerate}

\begin{remark}\label{rem:weights}
		Assumption {\assumptionB} is motivated by similar conditions in \cite{StrauchTiepner2024}.
		In particular, the conditions on the weights $w^h_\varepsilon(x)$ appear as a canonical nonparametric generalization (see \cite[Section 1.6]{Tsybakov2009}) of those imposed in
		Assumption $(W)$.
		It is possible to construct weights satisfying Assumption {\assumptionB} b) under natural assumption on the design $\cX_\varepsilon$. This is discussed in generality in
		\cite[Lemma 3.6/Example 3.7]{StrauchTiepner2024},
		following the lines of \cite{Tsybakov2009}.
		Here, we mention one important special case: For $\cD=(0,1)\subset\R$,
		with the uniform grid $\cX_\varepsilon = \{0, 1, \dots, \lfloor\varepsilon^{-1/2}\rfloor \}-x_0\varepsilon^{-1/2}$ and $h<(x_0\wedge(1-x_0))/2$,
		we can set $$w^h_\varepsilon(x) = \abs{I_{\varepsilon, h}}^{-1}\mathbbm{1}(x\in I_{\varepsilon, h})$$
		with $I_{\varepsilon, h}=\{x\in\cX_\varepsilon\;|\;\abs{x}<h\varepsilon^{-1/2}/2\}$.
		
\end{remark}

\begin{theorem}[sufficient conditions for nonparametric rates] \label{thm:connect-conditions}
	Let Assumption {\assumptionB} hold for some $\beta\in(1,2)$ and $0<\underline{\vartheta}<C$,
	and assume $X_0\in L^p(\cD)$ for some $p>2$.
	\begin{enumerate}
		\item[a)] In $d\geq 3$, assume $\beta<2-2/(4+d)$ and $p>p_0$, where $p_0 = p_0(d, \beta) \geq 2$ is given by \eqref{eq:p-initialcondition}. Then Assumption $(A_\NonparIndex)$ holds with $\NonparIndex=\beta$.
		\item[b)] In $d=2$, Assumption $(A_\NonparIndex)$ holds with $\NonparIndex=1$.
		\item[c)] In $d=1$, assume $\beta\geq 3/2$. Then Assumption $(A_\NonparIndex)$ holds with $\NonparIndex=1$.
	\end{enumerate}
	In particular, the conclusion from Theorem \ref{thm:abstract-rate} holds true,
	and for $h\sim\varepsilon^{\frac{2+d}{4\NonparIndex+2d}}$ we obtain $$\hat\vartheta_\varepsilon(x_0)=\vartheta(x_0)+\cO_\bP(\varepsilon^{\frac{(2+d)\NonparIndex}{4\NonparIndex+2d}}).$$
\end{theorem}

\begin{proof}
	\label{proof:connect-conditions}
	Apply Proposition \ref{prop:splitting-reduction} in order to aggregate relevant moment bounds from Lemma \ref{lem:parametric:moments} (empirical Fisher information), Lemma \ref{lem:moments:bias} (bias), Lemma \ref{lem:bounds-initialcondition} (initial condition) and Lemma \ref{lem:moments-noise} (static noise).
\end{proof}

\begin{remark} \
	\begin{enumerate}
		\item[a)] As $d$ gets large, the regularity restrictions in Theorem \ref{thm:connect-conditions} become weaker.
		In fact, the restrictions on $\beta$ in the statement of Theorem \ref{thm:connect-conditions} stem from semigroup approximation errors.
		This is directly linked to dimension-dependent hypercontractivity bounds of the heat semigroup.
		The gap between $\beta$ and $\NonparIndex$ in $d=1$ also appears in \cite{AltmeyerReiss2021}.
		\item[b)] The locally linear approach, together with quantitative Trotter--Kato bounds (see Section \ref{sec:semigroupapproximation}), can lead to a refined analysis even in the noiseless case ($\varepsilon=0$) discussed in \cite{AltmeyerReiss2021}.
		\item[c)]
		Our approach is based on a Taylor expansion of $\delta\mapsto\Delta_{\vartheta_\delta}$.
		In particular, this leads naturally to bounds
		with respect to weighted Sobolev spaces instead of the standard ones, see Lemma \ref{lem:bias:Laplacian}. As our proof of Theorem \ref{thm:connect-conditions} relies on estimates that are uniform in a neighborhood of shifted kernels $K_{k, x}$, the weights in the Sobolev norms cause additional divergence.
		This explains the upper bound $\beta<2-2/(4+d)$ rather than $\beta<2$ in $d\geq 3$.
		\item[d)]
		In contrast to the parametric case, where no such assumption is needed, the $\Delta$-order of $K$ allows to control the approximation error between the heat semigroups as well as their generators on $\cD$ and $\R^d$.
		Tracing the proof of Theorem \ref{thm:connect-conditions}, we see that $K$ can even have $\Delta$-order $1$ in $d=2$.
		In $d=1$, such improvement is not to be expected due to the gap between $\beta$ and $\NonparIndex$. On the other hand, in $d\geq 3$, we need $\Delta$-order two in order to obtain higher order quantitative bounds on the bias term $B_{\varepsilon, N_\varepsilon}$.
		As $K$ is chosen by the statistician, we can always assume that it has sufficiently large $\Delta$-order.
	\end{enumerate}
\end{remark}

In the sequel, we will fix the class $\Theta(\beta, \underline{\vartheta}, C)$ and implicitly understand that all estimates in the proofs will be uniform in $\vartheta\in\Theta(\beta, \underline{\vartheta}, C)$.

We finish this section with a remark on non-standard scaling properties of our model.

\begin{remark}\label{rem:Scaling}
	A scaling argument reveals two features of this statistical experiment which may seem surprising at first sight:
	\begin{enumerate}
		\item[a)] \emph{Large dynamic noise level helps.} Replacing $X$, $Y$ by $\sigma^{-1}X$, $\sigma^{-1}Y$ in \eqref{eq:basic:SPDE}, \eqref{eq:basic:Observation} leads to a normalized system where $\sigma$ is set to one, and $\varepsilon$ is replaced by $\varepsilon/\sigma$.
		Thus letting $\sigma$ grow will improve the estimate and identify the parameter!
		Constructing the estimator $\hat\vartheta_\varepsilon(x_0)$ with bandwidth $\delta=\sqrt{\varepsilon}$ in the normalized system corresponds to choosing $\delta=\sqrt{\varepsilon/\sigma}$ in the original system, as discussed in Remark \ref{rem:parametric}. 
		There is a different heuristic which helps to understand this effect:
		We consider the initial condition fixed in a way that the operator $\partial_t-\nabla\cdot\vartheta\nabla$ has a unique inverse $G_\vartheta$, and $G_\vartheta f$ is given by solving a parabolic (S)PDE with inhomogeneity $f$ (which may be white noise), i.e. $G_\vartheta=(\partial_t-\nabla\cdot\vartheta\nabla)^{-1}$.
		Formally rewriting \eqref{eq:basic:SPDE} as $X=\sigma G_\vartheta \dot W$,
		we can rephrase \eqref{eq:basic:Observation} as $Y=\sigma G_\vartheta\dot W+\varepsilon\dot V$, and increasing $\sigma$ will put more weight on the term that is not agnostic to $\vartheta$ (even if that term itself is random).
		A similar phenomenon for the one-dimensional Ornstein--Uhlenbeck process has been studied in \cite{Kutoyants2020}.

		\item[b)] \emph{The relative estimation error grows for large $\vartheta(x_0)$.}
		It is natural to conjecture that Theorem \ref{thm:connect-conditions} extends to
		\begin{align}\label{eq:scalingremark:T}
			\hat\vartheta_\varepsilon(x_0) = \vartheta(x_0)+\cO_\bP((\varepsilon^{\frac{1}{2} + \frac{d}{4}}T^{-\frac{1}{2}})^{\frac{2\NonparIndex}{2\NonparIndex + d}})
		\end{align}
		if the dependence on $T$ is explicitly traced in the nonparametric setup. While the approximation error of the heat semigroup prevents us from proving this rigorously in our setup,
		complementary lower bounds (see \cite{companionpaper}) indeed match.
		Replacing $X_t$ by the rescaled process $\sqrt{\vartheta(x_0)/\sigma^2}X_{t/(\vartheta(x_0))}$, and equally for $Y$, the resulting processes satisfy \eqref{eq:basic:SPDE}, \eqref{eq:basic:Observation} on $[0,\vartheta(x_0)T)$ with $\vartheta(x_0)$ and $\sigma$ normalized to one, and $\varepsilon$ replaced by $\vartheta(x_0)\varepsilon/\sigma$.
		That is, the effective static noise level and the effective observation horizon both scale linearly in the diffusivity.
		This implies that in \eqref{eq:scalingremark:T}, the error gets larger with $\vartheta(x_0)$: While $T$ grows, leading to more precise averaging in time, the increasing static noise level affects both time \emph{and} space!
		Note that after rescaling, the target parameter is normalized to one in \eqref{eq:scalingremark:T}, so it should be interpreted as relative rather than absolute estimation error.
		
		This is notably different from the classical setting without static noise,
		where the relative error \emph{shrinks} when $\vartheta(x_0)$ gets large.
		More precisely, consider the central limit theorems from \cite{AltmeyerReiss2021, AltmeyerBretschneiderJanakReiss2022, AltmeyerTiepnerWahl2022} within the scope of the local approach to drift estimation for SPDEs without static noise. In any of these works, the asymptotic variance is proportional to $1/T$, depends on $\vartheta$ and is independent of $\sigma$. Thus, scaling the latter quantities to unity as above, we obtain a dependence of the form $1/(\vartheta(x_0)T)$, all remaining terms being unchanged.
		This phenomenon is therefore intrinsically linked to both the presence of static noise, and spatially extended dynamics, which we consider in this work.

	\end{enumerate}
\end{remark}

\section{Semigroup Approximation}\label{sec:semigroupapproximation}

The proof of
our main results
relies heavily on uniform approximation results between the heat semigroup on $\cD_\delta$ and on $\R^d$. Pointwise approximation has been established in \cite{AltmeyerReiss2021}, and approximation results in different, but related settings can be found in \cite{StrauchTiepner2024} and \cite{AltmeyerTiepnerWahl2022}.
The most important feature in our setting is that the generators $\Delta_{\vartheta_\delta}$ and $\Delta_{\vartheta_0}$ deviate already in their highest order parts (as compared to possible lower order advection and reaction terms), which makes the approximation intrinsically hard.
We provide
a uniform Trotter--Kato type result that seems to be of independent interest.
As this section is analytic in nature and provides results that may be of interest beyond the statistical setup from Section \ref{sec:Estimation}, we write all terms in dependence of the localization parameter $\delta$ instead of the static noise level $\varepsilon$.

\begin{theorem}[uniform Trotter--Kato theorem] \label{thm:semigroup-approximation-simplified}
	Let $\delta\lesssim h\lesssim 1$.
	Fix two closed balls $\cB\subset\cB^*$ in $\cD$ centered at $x_0$,
	write $\cB[\delta/h]:=\cB_{\delta/h}=h\delta^{-1}(\cB-x_0)$ and use analogous notation for $\cB^*$.
	Write $\shift_y\varphi(x):=\varphi(x-y)$ for $x, y\in\R^d$.
	Fix $T>0$ and
	$\varphi,\psi\in C^\infty_c(\R^d)$ such that
	$\mathrm{supp}(\shift_y\varphi),\mathrm{supp}(\shift_y\psi)\subseteq\cB^*[\delta]$ for $y\in\cB[\delta/h]$.
	\begin{enumerate}
		\item[a)]
		For $p\geq 2$,
		\begin{align}
			\sup_{0\leq t\leq T}\sup_{y\in\cB[\delta/h]}\norm{(S_\delta(t) - S_0(t))\shift_y\varphi}_{L^p(\R^d)} \lesssim h.
		\end{align}
		If $\vartheta$ is constant, then there is some $c>0$, depending on $T$, with
		\begin{align}
			\sup_{0\leq t\leq T}\sup_{y\in\cB[\delta/h]}\norm{(S_\delta(t) - S_0(t))\shift_y\varphi}_{L^p(\R^d)} \lesssim e^{-c\delta^{-2}}.
		\end{align}
		
		\item[b)] Let $d\geq 3$. If $\varphi$ and $\psi$ have $\Delta$-order two and three, then for $C>0$ and $\alpha>1$:
		\begin{align}
			\int_0^{C\delta^{-2}}\sup_{y_1,y_2\in\cB[\delta/h]}\abs{\scp{\shift_{y_1}\varphi}{(S_\delta(t) - S_0(t))\shift_{y_2}\psi}_0}\diff t \lesssim h(h\delta^{-1})^{\alpha}.
		\end{align}
	\end{enumerate}
\end{theorem}

\begin{proof}
	See Supplement \ref{sec:proofs:semigroupapproximation}.
\end{proof}

Both statements control the approximation error between $S_\delta$ and $S_0$ in a general, albeit different way. By restricting to an analytically weak formulation in the second part of the proposition, we are able to find a quantitative bound that is integrable in time, which we need given that the final time horizon grows upon localization of the stochastic heat equation.
We mention that the constants in the estimates in Theorem \ref{thm:semigroup-approximation-simplified} are proportional to a weighted Sobolev norm of $\varphi$ and $\psi$, see Supplement \ref{sec:TrotterKato} for a detailed discussion.
In the proof, which is based on the general theory exposed in \cite{ItoKappel1998}, we have to control the errors arising from the heterogeneity of $\vartheta$ (compared to a constant diffusivity of the form $\vartheta(x_0)$), the boundedness of the domain $\cD$ (compared to $\R^d$) and the Dirichlet boundary conditions in the dynamics of the heat semigroup $(S_\delta(t))_{t\geq 0}$.
	For our statistical analysis we state the following specialized version, which translates the results from Theorem \ref{thm:semigroup-approximation-simplified} to the setting of Section \ref{sec:Estimation}:

\begin{corollary}[cumulative Trotter--Kato approximation] \label{cor:semigroup-for-statistics} \
	Fix $R>0$.
	Let $\varphi,\psi\in \cC_0$ such that
	$\mathrm{supp}(\shift_y\varphi),\mathrm{supp}(\shift_y\psi)\subseteq[0,R]\times\cB^*[\delta]$ for $y\in\cB[\delta/h]$.
	\begin{enumerate}
		\item[a)]
	In the setting of Theorem \ref{thm:connect-conditions}, if $\varphi$ or $\psi$ has $\Delta$-order $(3-d)_+$, we have for $C>0$:
	\begin{align}\label{eq:uniform-covariance-convergence}
		\int_0^R\int_0^R\int_{0}^{C\varepsilon^{-1}}\sup_{x\in\cB[\varepsilon^{1/2}/h]}\abs{\scp{\shift_x\varphi_t}{(S_\varepsilon(r)-S_0(r))\shift_x\psi_s}_0}\diff r\diff s\diff t\rightarrow 0
	\end{align}
	as $\varepsilon\rightarrow 0$.
	In $d\geq 3$, if in addition $\varphi$ and $\psi$ have $\Delta$-order two and three, then the left-hand side is of order $\cO(h^{\beta-1})$.
		\item[b)] In the setting of Theorem \ref{thm:parametric}, \eqref{eq:uniform-covariance-convergence} is true
		if $\varphi$ or $\psi$ has $\Delta$-order one in $d\leq 2$, and without conditions on the $\Delta$-order if $d\geq 3$.
	\end{enumerate}
\end{corollary}

\begin{proof}
	See Supplement \ref{sec:proofs:semigroupapproximation}.
\end{proof}

	Under Assumption {\assumptionB} and $h=o(1)$,
	we have that $\mathrm{supp}(w^h_\varepsilon)$ is
	asymptotically contained in
	any $\cB^*[\delta]$, which means that the support condition from Corollary \ref{cor:semigroup-for-statistics} is satisfied. Thus the results from this section can be applied. For parametric estimation, it is demanded explicitly in Assumption $(W)$ that the support of $K$ shifted by any point in
	$\mathrm{supp}(w^h_\varepsilon)$
	is separated from $\partial\cD$.

\section{A Numerical Example}\label{sec:Numerics}

\begin{figure}
	\includegraphics[width=0.49\textwidth]{"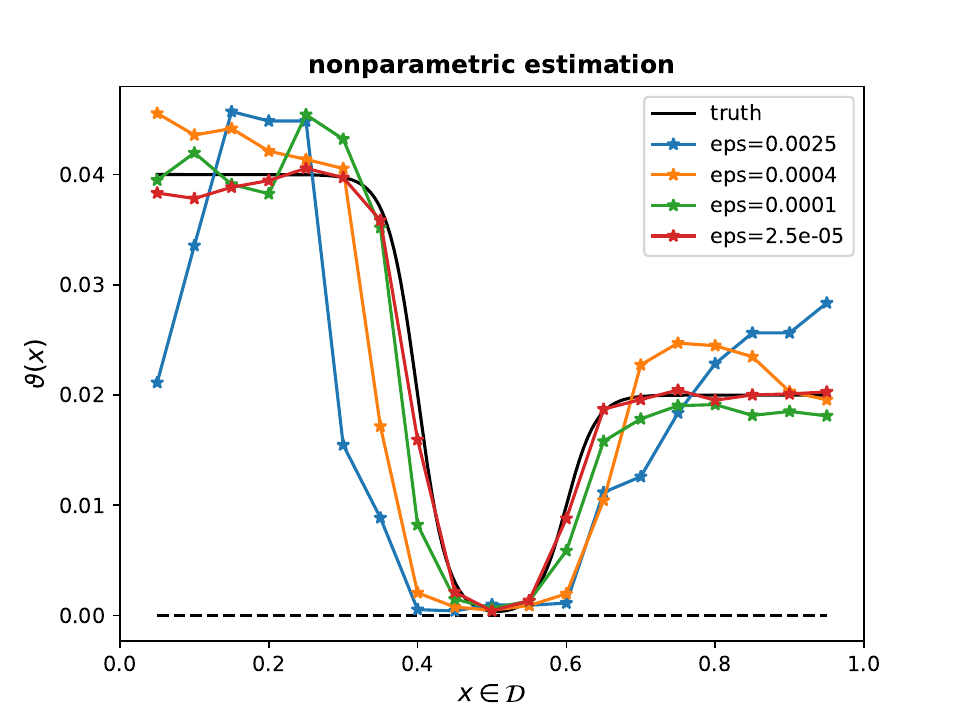"}
	\includegraphics[width=0.49\textwidth]{"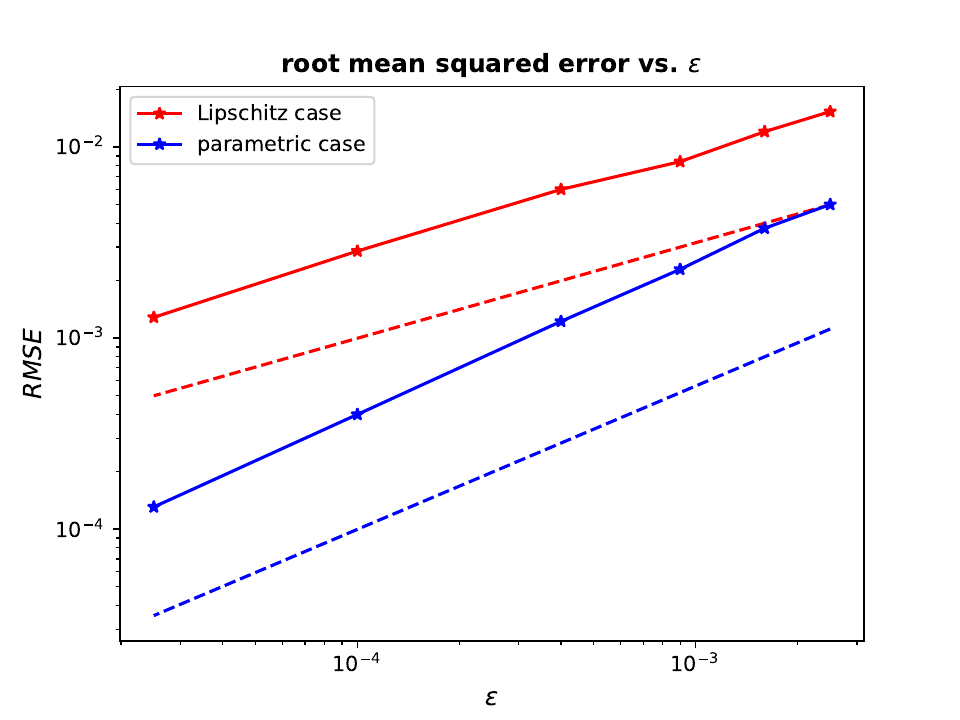"}
	\caption{
	\emph{Left:} Reconstruction of a spatially varying diffusivity at different points in space.
	\emph{Right:} RMSE vs. $\varepsilon$, formal comparison of the convergence rate of the Lipschitz ($\NonparIndex=1$) and the parametric ($\NonparIndex=\infty$) estimation problem.
	}\label{fig:numerics:results}
\end{figure}

We simulate the stochastic heat equation on $\cD=(0,1)$, $\cT=[0,1)$, where we use a discrete grid in space and time of size $\diff x=10^{-3}$ and $\diff t = 10^{-6}$, starting from initial condition
zero
and with $\sigma=10$, using an implicit Euler--Maruyama scheme with finite difference discretization in space \cite{LordPowellShardlow2014}.
Static white Gaussian noise $\dot V$ is simulated as an array of independent Gaussian variables on the grid points, with mean zero and variance $\diff t^{-1}\diff x^{-1}$.
We set $K(t, x)=2K(2t-1)K(x)$, where $K(x)$ is the normalization in $L^1(-1,1)$ of $x\mapsto\exp(-10(x+1)^{-2}(1-x)^{-2})$. While such a kernel does not satisfy the condition on the $\Delta$-order in Assumption {\assumptionB}, we will see that it leads to good numerical results nonetheless.
Our grid $\cX_\varepsilon$ consists of all points $x_k=0.5+2\delta k$, $k\in\bZ$, such that the localized kernel support $(x_k-\delta,x_k+\delta)$ is separated from the boundary of $\cD$ with distance at least $0.1\delta$.
In accordance with Remark \ref{rem:weights}, and noting that $\hat\vartheta_\varepsilon$ remains invariant under rescaling of the weights $w^h_\varepsilon(x)$ with a constant, we set them to one in the active neighborhood of size $h$ around $x_0$, and zero otherwise.

In Figure \ref{fig:numerics:results} (left), we estimate a spatially heterogeneous function $\vartheta(x) = 0.04 \psi(x-0.4) + 0.02\psi(0.6-x)$ with $\psi(y)=1/(1+\exp(50y))$ 
at equidistant grid points $x_0\in\{0.05i\;:\;i=1,\dots,19\}$
from one realization of the stochastic heat equation 
, where we set $h=\varepsilon^{1/3}$, slightly larger than the Lipschitz case $h=\varepsilon^{1/2}$.
The shape of $\vartheta$ is similar to the function used in \cite{AltmeyerReiss2021}, but our choice of parameters cannot be directly compared to the theirs due to the different scaling of our problem, see Remark \ref{rem:Scaling}.
We see that magnitude and shape of $\vartheta$ are correctly recovered, with increasing precision as $\varepsilon\rightarrow 0$. Estimates at adjacent points $x_0$ appear to be correlated as they are constructed from partially overlapping data.
Estimation near the boundary is expected to perform worse, as only those spatial shifts of $K$ are taken into account which stay within the domain, reducing
the information used in order to estimate $x_0$.
The reported values for $\varepsilon$ correspond to $\delta\in\{0.05, 0.02, 0.01, 0.005\}$, taking into account the relation $\varepsilon=\delta^2$.
In particular, at the coarsest resolution level, corresponding to the blue line in Figure \ref{fig:numerics:results} (left), the testing kernel has spatial diameter $0.1$.
The values for the effective sample size from Remark \ref{rem:parametric} are given by $N_\mathrm{eff}\in\{3600, 57500, 4.9\times 10^5, 3.96\times 10^6\}$ for the four cases considered.

In Figure \ref{fig:numerics:results} (right), the root mean squared error (RMSE) of $\hat\vartheta_\varepsilon$ estimated at $x_0=0.5$ is plotted as $\varepsilon\rightarrow 0$, based on $M=1000$ Monte Carlo runs, each with constant diffusivity $\vartheta=0.02$.
In order to compare the parametric case with estimation under Lipschitz assumptions, we plot the results for $\hat\vartheta_\varepsilon$ if either full spatial information is used ($h=1$, parametric case), or just information in a neighborhood of $x_0$ is used ($h=\sqrt{\varepsilon}$, Lipschitz case).
The dashed lines represent (up to a constant) the theoretically expected size of the error, i.e. $\varepsilon^{1/2+d/4}$ in the parametric case, and $\varepsilon^{1/2}$ in the Lipschitz case.
Note that although the simulation itself is parametric here, we choose $h$ such that Lipschitz behaviour is emulated. 
In both cases, the rate of decay of the RMSE matches theoretical predictions.

\appendix

\section{Notation}\label{sec:notation}

We use dependence on $\varepsilon$ or $\delta$ interchangeably, e.g. $\Delta_{\vartheta_\varepsilon}=\Delta_{\vartheta_\delta}$ (with $\delta\sim\sqrt{\varepsilon}$) by abuse of notation. As a rule of thumb, sections on \emph{statistics} (such as Section \ref{sec:Estimation}) use dependence on $\varepsilon$, whereas sections on \emph{structure} (such as Section \ref{sec:semigroupapproximation}) use dependence on $\delta$.
\\

\begin{description}

	\item[\emph{Sobolev spaces.}]
	We use the notation $W^{k, p}(\cD)$, $W^{k, p}_0(\cD)$ and $W^{k, p}(\R^d)$ for standard Sobolev spaces, see \cite{AdamsFournier2003} for details.
	For $p\geq 1$, and $\alpha\geq 0$ define the weighted $L^p$-norm
	\begin{align}
		\abs{\FuncZ}_{p, \alpha} := \norm{\abs{x}^\alpha \FuncZ}_{L^p(\R^d)},
	\end{align}
	and for $k\in\N_0$,
	writing $\multiindex\in\N_0^d$ for a multiindex and $\partial_{\multiindex}=\partial_{x_1}^{\multiindex_1}\cdots\partial_{x_d}^{\multiindex_d}$,
	the weighted Sobolev norm
	\begin{align}
		\norm{\FuncZ}_{\bar W^{k, p}_\alpha} := \sum_{\abs{\multiindex}\leq k}\left(\abs{\partial_{\multiindex} \FuncZ}_{p, 0} + \abs{\partial_{\multiindex} \FuncZ}_{p, \alpha}\right).
	\end{align}
	$\bar W_\alpha^{k, p}$ is the space of all $\FuncZ\in W^{k, p}(\R^d)$ such that this norm is finite.
	In particular, we have $\bar W_0^{k, p}=W^{k, p}(\R^d)$.
	For any two such spaces $\bar W_{\alpha_1}^{k_1, p_1}$ and $\bar W_{\alpha_2}^{k_2, p_2}$, we equip the Banach space $\bar W_{\alpha_1}^{k_1, p_1}\cap \bar W_{\alpha_2}^{k_2, p_2}$ with the norm defined by $\norm{\varphi}_{\bar W_{\alpha_1}^{k_1, p_1}\cap \bar W_{\alpha_2}^{k_2, p_2}}=\norm{\varphi}_{\bar W_{\alpha_1}^{k_1, p_1}}+\norm{\varphi}_{\bar W_{\alpha_2}^{k_2, p_2}}$.

	\item[\emph{Heat kernel.}] $q(x)=e^{-\abs{x}^2/2}/(2\pi)^{d/2}$ is the density of the standard normal distribution on $\R^d$, $q_t(x)=(2t)^{-d/2}q((2t)^{-1/2}x)$ denotes the heat kernel on $\R^d$. The diffusion semigroup $t\mapsto e^{t\Delta}$ on $L^p(\R^d)$, generated by $\Delta$, acts via convolution with $q_t$. This differs from $S_0$ (see below) only by its rescaling in time with $\vartheta(x_0)$.
	\item[\emph{Laplacian.}] $\Delta_{\vartheta_\delta}=\nabla\cdot\vartheta_\delta\nabla:W^{2, p}(\cD_\delta)\rightarrow L^p(\cD_\delta)$ is a second-order differential operator, where $\vartheta_\delta(x) = \vartheta(\delta x + x_0)$ maps $\cD_\delta\rightarrow (0,\infty)$ for $\delta\geq 0$.
	In particular, as $\vartheta_0$ is a constant function, the notation $\Delta_{\vartheta_0}$ may refer to an operator on a bounded or unbounded domain, depending on the context.
	\item[\emph{Semigroup.}] $A_\delta:D_p(A_\delta)\rightarrow L^p(\cD_\delta)$ for $p>1$ is the restriction of $\Delta_{\vartheta_\delta}$ to the domain $D_p(A_\delta)$,
	which includes boundary conditions.
	$A_\delta$ generates the semigroup $t\mapsto S_\delta(t)$ on $L^p(\cD_\delta)$, for $\delta\geq 0$.
	For $\delta>0$, $D_p(A_\delta)=W^{2, p}(\cD_\delta)\cap W^{1,p}_0(\cD_\delta)$ \cite[Section 2.4]{Yagi2010}.
	For $\delta=0$, $D_p(A_\delta)=W^{2,p}(\R^d)$ \cite[Section VII.4]{Werner2005}.
	\item[\emph{Linearization of $\Delta_{\vartheta_\delta}$.}] The operator $\Delta'_{(\nabla\vartheta)_0}$ acts for $p\geq 1$ as $\bar W^{2, p}_1\rightarrow L^p(\R^d)$ or $W^{2, p}(\cD_\delta)\rightarrow L^p(\cD_\delta)$ ($\delta> 0$), it is given by
\begin{align}
	\Delta'_{(\nabla\vartheta)_0}\FuncZ(x) := \nabla\vartheta(x_0)\cdot x\Delta \FuncZ(x) + \nabla\vartheta(x_0)\cdot\nabla \FuncZ(x).
\end{align}
	\item[\emph{Truncation.}] $\pi_\delta: W^{k, p}(\R^d)\rightarrow W^{k,p}(\cD_\delta)$ is the truncation operator for $\delta>0$, i.e. $\pi_\delta\FuncZ=\FuncZ|_{\cD_\delta}$.
	\item[\emph{Zero boundary projection.}] $\bar\pi_\delta:W^{2, p}(\R^d)\rightarrow D_p(A_\delta)$ is the projection onto the space of functions on $\cD_\delta$ satisfying Dirichlet boundary conditions, which is obtained by forcing the trace of $\pi_\delta\FuncZ$ on $\partial\cD_\delta$ to be zero for $\FuncZ\in W^{2, p}(\R^d)$.
	It is given as the solution to
	\begin{align}\label{eq:DirichletProjection}
		\left\{\begin{matrix}
			\Delta_{\vartheta_\delta}\bar\pi_\delta\FuncZ &=& \Delta_{\vartheta_\delta}\pi_\delta\FuncZ & \quad\mathrm{on}\quad & \cD_\delta, \\
			\hspace{0.7cm}\bar\pi_\delta\FuncZ &=& 0 \quad\quad\quad & \quad\mathrm{on}\quad & \partial\cD_\delta,
		\end{matrix}\right.
	\end{align}
	or equivalently, $\bar\pi_\delta \FuncZ = \pi_\delta \FuncZ - \FuncPotential$, where $\FuncPotential$ is the solution to
	\begin{align}\label{eq:HarmonicPotential}
		\left\{\begin{matrix}
			\Delta_{\vartheta_\delta}\FuncPotential &=& 0 & \quad\mathrm{on}\quad & \cD_\delta, \\
			\hspace{0.7cm}\FuncPotential &=& \FuncZ & \quad\mathrm{on}\quad & \partial\cD_\delta.
		\end{matrix}\right.
	\end{align}
	According to \cite[Theorem 9.15]{GilbargTrudinger2001},
	there is a unique solution $\FuncPotential\in W^{2, p}(\cD_\delta)$ to \eqref{eq:DirichletProjection} and \eqref{eq:HarmonicPotential}, where the boundary equations in \eqref{eq:DirichletProjection} and \eqref{eq:HarmonicPotential} should be interpreted as $\bar\pi_\delta\varphi\in W^{1, p}_0(\cD_\delta)$
	and $\pi_\delta\FuncZ-\FuncPotential\in W^{1,p}_0(\cD_\delta)$. So $\bar\pi_\delta$ is well-defined.
	\item[\emph{Shift.}] $\shift_y:L^p(\R^d)\rightarrow L^p(\R^d)$ is the shift operator given by $\shift_y\FuncZ(x)=\FuncZ(x-y)$.
	
	\item[\emph{Space-time kernels.}]
	As before, we write $\FuncZ_{k, x}=\FuncZ(\cdot-k, \cdot-x)$ for the space-time shift of some kernel $\varphi$.
	In addition, we abbreviate $\FuncZ_t=\FuncZ(t, \cdot)$ for its trace at time $t\geq 0$.
\end{description}

\section{Derivation of the Main Result}\label{sec:Proofs}

In this section,
we present all statements and proofs in a top-down manner, reflecting the importance of concepts.

\subsection{Postponed Proofs}

\subsubsection{Proofs from Section \ref{sec:spde}}\label{sec:proofs:spde}

\begin{proof}[Proof of Lemma \ref{lem:FieldIsStweak}] \label{proof:FieldIsStweak}
	Fix $\alpha>d/2$. Then $B:=(-\Delta_\vartheta)^{-\alpha/2}:L^2\rightarrow D((-\Delta_\vartheta)^{\alpha/2})\subset L^2$ is a Hilbert--Schmidt operator on $L^2$, and for a cylindrical Brownian motion $W_t$ there is a mild solution $Z$ to
	\begin{align*}
		\diff Z_t = \Delta_\vartheta Z_t\diff t + \sigma B\diff W_t
	\end{align*}
	with initial condition $B\xi$, see \cite{DaPratoZabczyk2014}. Defining the Gaussian process $X'$ on $\cC$ via
	\begin{align*}
		\sscpSmall{X'}{\varphi} := \sscp{Z}{B^{-1}\varphi},
	\end{align*}
	where $B^{-1}$ maps $\cC\subset D((-\Delta_\vartheta)^{\alpha/2})\rightarrow L^2$, we have that $X\sim X'$ in distribution, since
	\begin{align*}
		\sscpSmall{X'}{\varphi} = \int_0^\infty\scp{\varphi(t,\cdot)}{S(t)\xi + \sigma\int_0^tS(t-r)\diff W_r}\diff t,
	\end{align*}
	and so $\bE[\sscpSmall{X'}{\varphi}]=\sscpSmall{S(\cdot)\xi}{\varphi}$ and
	\begin{align*}
		\rCov(\sscpSmall{X'}{\varphi},\sscpSmall{X'}{\psi}) \hspace{-2cm} & \\
			&= \sigma^2\int_0^\infty\int_0^\infty\scp{\varphi(t, \cdot)}{\left(\int_0^{t\wedge s}S(t-r)S(s-r)\diff r\right)\psi(s, \cdot)}\diff s\diff t \\
			&= \sigma^2\int_0^\infty\int_0^\infty\int_0^{t\wedge s}\scp{\varphi(t, \cdot)}{S(t+s-2r)\psi(s, \cdot)}\diff r\diff s\diff t \\
			&= \frac{\sigma^2}{2}\int_0^\infty\int_0^\infty\int_{\abs{t-s}}^{t+s}\scp{\varphi(t, \cdot)}{S(r)\psi(s, \cdot)}\diff r\diff s\diff t.
	\end{align*}

	Next, note that $Z$ is a weak solution, i.e. for $\psi\in D(\Delta_\vartheta)$:
	\begin{align}\label{eq:SPDE:weak}
		\scp{Z_t}{\psi} = \scp{B\xi}{\psi} + \int_0^t\scp{Z_s}{\Delta_\vartheta\psi}\diff s + \sigma\scp{BW_t}{\psi}.
	\end{align}
	This implies that $X'$ is a space-time weak solution with respect to the isonormal Gaussian process induced by $W_t$.
	We may restrict to test functions of the form $\varphi_t=\rho\otimes\psi$, where $\rho\in C_c^\infty(\cT)$ and $\psi\in C_c^\infty(\cD)$. Then
	\begin{align*}
		\scp{\xi}{\varphi(0,\cdot)} &= \scp{B\xi}{B^{-1}\psi}\rho_0 = - \int_0^T\scp{B\xi}{B^{-1}\psi}\dot\rho_t\diff t, \\
		\sscp{X'}{\Delta_\vartheta\varphi} &= \int_0^T\scp{Z_t}{\Delta_\vartheta B^{-1}\psi}\rho_t\diff t = - \int_0^T\left(\int_0^t\scp{Z_s}{\Delta_\vartheta B^{-1}\psi}\diff s\right)\dot\rho_t\diff t, \\
		\sigma\sscpSmall{\dot W}{\varphi} &= \sigma\int_0^T\rho_t\scp{B^{-1}\psi}{B\diff W_t} = - \int_0^T\sigma\scp{BW_t}{B^{-1}\psi}\dot\rho_t\diff t.
	\end{align*}
	The claim follows from plugging $B^{-1}\psi$ into \eqref{eq:SPDE:weak} and testing with $-\dot\rho$ in $L^2(\cT)$.
	To conclude the proof, note that
	$\varphi\mapsto -(\sscpSmall{X}{\dot\varphi}+\scpSmall{\xi}{\varphi(0,\cdot)}+\sscpSmall{X}{\Delta_\vartheta\varphi})/\sigma$
	has the same finite-dimensional laws when replacing $X$ by $X'$ (which results in
	$\sscpSmall{\dot W}{\varphi}$),
	so it can be extended to an isonormal Gaussian process on $L^2(\cT\times\cD)$ given that $\cC$ is dense, and $X$ is a space-time weak solution driven by that process.
\end{proof}

\subsubsection{Proofs from Section \ref{sec:localization}} \label{sec:proofs:localization}

\begin{proof}[Proof of Lemma \ref{lem:semigroup-scaling}] \label{proof:semigroup-scaling}
		For $T_\delta(r)\varphi:=(S(r\delta^{2})\varphi_\delta)^{1/\delta}$ it holds that $$\partial_tT_\delta(r)\varphi = \delta^{2}(\Delta_{\vartheta}S(r\delta^{2})\varphi_\delta)^{1/\delta} = \Delta_{\vartheta_\delta}(S(r\delta^{2})\varphi_\delta)^{1/\delta} = \Delta_{\vartheta_\delta}T_\delta(r)\varphi,$$ so $T=S$ by uniqueness.
		Set $r=t\delta^{-2}$.
\end{proof}

\begin{proof}[Remaining proof of Lemma \ref{lem:localization-X}] \label{proof:localization-X}
	With \eqref{eq:field:Mean}, \eqref{eq:field:Cov}
	and Lemma \ref{lem:semigroup-scaling},
	\begin{align*}
		\bE[\sscpSmall{X^{(\delta)}}{\varphi}_\delta]
			&= \delta^{-2}\bE[\sscp{X}{\varphi_\delta}]
			= \delta^{-2}\sscp{S(\cdot)\xi}{\varphi_\delta}
			= \delta^{-2}\sscpSmall{(S(\cdot)\xi)^{1/\delta}}{\varphi}_\delta \\
			&= \delta^{-1}\int_0^\infty\scpSmall{(S(\delta^2t)\xi)^{1/\delta}}{\varphi(t, \cdot)}_\delta\diff t
			= \delta^{-1}\int_0^\infty\scpSmall{S_\delta(t)\xi^{1/\delta}}{\varphi(t, \cdot)}_\delta\diff t \\
			&= \sscpSmall{S_\delta(\cdot)\xi^{(\delta)}}{\varphi}_\delta,
	\end{align*}
	as well as
	\begin{align*}
		\rCov(\sscpSmall{X^{(\delta)}}{\varphi}_\delta, \sscpSmall{X^{(\delta)}}{\psi}_\delta) \hspace{-4cm} & \hspace{4cm} = \delta^{-4}\rCov\sscpSmall{X}{\varphi_\delta}, \sscpSmall{X}{\psi_\delta}) \\
			&= \frac{\sigma^2}{2}\delta^{-4}\delta^{-2-d}\int_0^\infty\int_0^\infty\int_{\abs{t-s}}^{t+s}\scp{\varphi\left(\frac{t}{\delta^2}, \frac{\cdot-x_0}{\delta}\right)}{S(r)\psi\left(\frac{s}{\delta^2}, \frac{\cdot-x_0}{\delta}\right)}\diff r\diff s\diff t \\
			&= \frac{\sigma^2}{2}\delta^{-6}\int_0^\infty\int_0^\infty\int_{\abs{t-s}}^{t+s}\scp{\varphi\left(\frac{t}{\delta^2}, \cdot\right)}{S_\delta\left(\frac{r}{\delta^2}\right)\psi\left(\frac{s}{\delta^2}, \cdot\right)}_\delta\diff r\diff s\diff t \\
			&= \frac{\sigma^2}{2}\int_0^\infty\int_0^\infty\int_{\abs{t-s}}^{t+s}\scp{\varphi(t, \cdot)}{S_\delta(r)\psi(s, \cdot)}_\delta\diff r\diff s\diff t.
	\end{align*}
\end{proof}

\subsubsection{Proofs from Section \ref{sec:Estimation}} \label{sec:proofs:estimation}

\begin{proof}[Proof of Lemma \ref{lem:errordecomposition}.] \label{proof:errordecomposition}
	First,
\begin{align*}
	\widehat X'_{\varepsilon, k, x} - \vartheta(x_0) \widehat X^\Delta_{\varepsilon, k, x}
		&= \sscpSmall{\partial_tX}{(K_{k, x})_{\varepsilon, x_0}} - \varepsilon\sscpSmall{\dot V}{\partial_t(K_{k, x})_{\varepsilon, x_0}} \\
		&\quad\quad - \vartheta(x_0)\sscpSmall{\Delta X}{(K_{k, x})_{\varepsilon, x_0}} - \vartheta(x_0)\varepsilon\sscpSmall{\dot V}{\Delta (K_{k, x})_{\varepsilon, x_0}} \\
		&\hspace{-0.5cm}= \left(\sigma\sscpSmall{\dot W}{(K_{k, x})_{\varepsilon, x_0}} - \sscpSmall{\dot V}{(\partial_tK_{k, x})_{\varepsilon, x_0}+(\Delta_{\vartheta(x_0)} K_{k, x})_{\varepsilon, x_0}}\right) \\
		&\hspace{-0.5cm}\quad\quad\quad + \sscpSmall{[\Delta_{\vartheta}-\Delta_{\vartheta(x_0)}] X}{(K_{k, x})_{\varepsilon, x_0}} \\
		&\hspace{-0.5cm}=: D_{\varepsilon, k, x} + E_{\varepsilon, k, x},
\end{align*}
where $(D_{\varepsilon, k, x})_{k\in\N_0}$ is for each $x\in\cX_\varepsilon$ a sequence of martingale differences with respect to the filtration $\cF_{\varepsilon, k}:=\sigma(\dot W(\varphi), \dot V(\varphi):\varphi\in L^2([0,\varepsilon(k+1))\times\cD))$, with
\begin{align*}
	\bE[(D_{\varepsilon, k, x})^2|\cF_{\varepsilon, k-1}]
		&= \sigma^2\norm{(K_{k, x})_{\varepsilon, x_0}}^2 + \norm{(\partial_tK_{k, x})_{\varepsilon, x_0}+(\Delta_{\vartheta(x_0)} K_{k, x})_{\varepsilon, x_0}}^2=\sigma_K^2,
\end{align*}
where $\sigma_K^2$ does neither depend on $k$, $x$ nor $\varepsilon$ (for $k\leq N_\varepsilon$).
This leads to the error decomposition
\begin{align*}
	\hat\vartheta_\varepsilon(x_0) - \vartheta(x_0) &= \frac{\sum_{k=1}^{N_\varepsilon}\sum_{x\in\cX_\varepsilon}w_\varepsilon^h(x)\widehat X^\Delta_{\varepsilon, k-1, x}[D_{\varepsilon, k, x}+E_{\varepsilon, k, x}]}{\sum_{k=1}^{N_\varepsilon}\sum_{x\in\cX_\varepsilon}w_\varepsilon^h(x)\widehat X^\Delta_{\varepsilon, k-1, x}\widehat X^\Delta_{\varepsilon, k, x}} = \frac{M_{\varepsilon, N_\varepsilon}+B_{\varepsilon, N_\varepsilon}}{I_{\varepsilon, N_\varepsilon}},
\end{align*}
which is as desired. Indeed, $(M_{\varepsilon, N})_{N\in\N}$ is a square-integrable martingale,
given as a sum (over $x\in\cX_\varepsilon$) of martingale transforms with integrator $D_{\varepsilon, k, x}$, having the claimed quadratic variation.
\end{proof}

\begin{lemma}\label{lem:conditionallyapunov}
	In the setting of Theorem \ref{thm:parametric}, and using notation from the proof of Lemma \ref{lem:errordecomposition}, we have
	\begin{align}\label{eq:conditionallyapunov}
		\bE\left[\sum_{k=1}^{N_\varepsilon}\left(\sum_{x\in\cX_\varepsilon}w_\varepsilon(x)\widehat X^\Delta_{\varepsilon, k-1, x}D_{\varepsilon, k, x}\right)^4\right]\lesssim N_\varepsilon\cW_\varepsilon^2.
	\end{align}
	In particular, the triangular martingale scheme $(N_\varepsilon^{-1/2}\cW_\varepsilon^{-1/2}M_{\varepsilon, N})_{N=0,...,N_\varepsilon}$ satisfies the Lyapunov condition with the fourth moment.
\end{lemma}

\begin{proof}
	The term on the left-hand side of \eqref{eq:conditionallyapunov} equals
	\begin{align*}
		\sum_{k=1}^{N_\varepsilon}\sum_{x, y, u, v\in\cX_\varepsilon}w_\varepsilon(x)w_\varepsilon(y)w_\varepsilon(u)w_\varepsilon(v)\bE[\widehat X^\Delta_{\varepsilon, k-1, x}\widehat X^\Delta_{\varepsilon, k-1, y}\widehat X^\Delta_{\varepsilon, k-1, u}\widehat X^\Delta_{\varepsilon, k-1, v}] & \\
		&\hspace{-3cm}\times\bE[D_{\varepsilon, k, x}D_{\varepsilon, k, y}D_{\varepsilon, k, u}D_{\varepsilon, k, v}],
	\end{align*}
	which is non-zero only if two pairs of the variables $x,y,u,v$ are equal, or all four of them.
	Note that uniformly in $1\leq k\leq N_\varepsilon$ and $x, y\in\cX_\varepsilon$,
	\begin{align*}
		\bE[\widehat X^\Delta_{\varepsilon, k-1, x}\widehat X^\Delta_{\varepsilon, k-1, y}] = \bE[\bar X^\Delta_{\varepsilon, k-1, x}\bar X^\Delta_{\varepsilon, k-1, y}] + \bE[\dot V^\Delta_{\varepsilon, k-1, x}\dot V^\Delta_{\varepsilon, k-1, y}]\lesssim 1
	\end{align*}
	due to Lemma \ref{lem:cov:bound:kl} and $\bE[\dot V^\Delta_{\varepsilon, k-1, x}\dot V^\Delta_{\varepsilon, k-1, y}] = \norm{\Delta K}^2\mathbbm{1}(x=y)$.
	Now if two pairs of $x, y, u, v$ are equal, we obtain a bound of the form
	\begin{align*}
		\sigma_K^4\sum_{k=1}^{N_\varepsilon}\sum_{x,y\in\cX_\varepsilon}w_\varepsilon(x)^2w_\varepsilon(y)^2\bE[(\widehat X^\Delta_{\varepsilon, k-1, x})^2(\widehat X^\Delta_{\varepsilon, k-1, y})^2]\lesssim \sum_{k=1}^{N_\varepsilon}\sum_{x,y\in\cX_\varepsilon}w_\varepsilon(x)^2w_\varepsilon(y)^2=N_\varepsilon\cW_\varepsilon^2,
	\end{align*}
	where we used Wick's formula for the mixed fourth moment.
	If all variables $x, y, u, v$ are equal, then
	\begin{align*}
		3\sigma_K^4\sum_{k=1}^{N_\varepsilon}\sum_{x\in\cX_\varepsilon}w_\varepsilon(x)^4\bE[(\widehat X^\Delta_{\varepsilon, k-1, x})^4]\lesssim \sum_{k=1}^{N_\varepsilon}\sum_{x\in\cX_\varepsilon}w_\varepsilon(x)^4\leq N_\varepsilon\cW_\varepsilon^2,
	\end{align*}
	again due to Wick's formula.
\end{proof}

\subsection{Splitting} \label{sec:Splitting}

For $t\geq 0$, define $\tilde X_t:=S(t)X_0$, and $\bar X_t:=X_t-\tilde X_t$. Then these processes are solutions to

\begin{align}
	\partial_t\bar X_t &= \Delta_\vartheta \bar X_t + \sigma\dot W_t, \hspace{1.4cm} \bar X_0=0, \\
	\partial_t\tilde X_t &= \Delta_\vartheta\tilde X_t,\hspace{2.5cm}\tilde X_0=X_0.
\end{align}
For $\varepsilon> 0$, we use an analogous decomposition $X^{(\varepsilon)}_t = \bar X^{(\varepsilon)}_t+\tilde X^{(\varepsilon)}_t$.
Define
\begin{align*}
	\begin{matrix}
		\bar X^\Delta_{\varepsilon, k, x} := \sscpSmall{\Delta\bar X}{(K_{k, x})_{\varepsilon, x_0}}, &
		\bar X^{\Delta-\Delta}_{\varepsilon, k, x} := \sscpSmall{[\Delta_{\vartheta}-\Delta_{\vartheta(x_0)}]\bar X}{(K_{k, x})_{\varepsilon, x_0}}, \\
		\tilde X^\Delta_{\varepsilon, k, x} := \sscpSmall{\Delta\tilde X}{(K_{k, x})_{\varepsilon, x_0}}, &
		\tilde X^{\Delta-\Delta}_{\varepsilon, k, x} := \sscpSmall{[\Delta_{\vartheta}-\Delta_{\vartheta(x_0)}]\tilde X}{(K_{k, x})_{\varepsilon, x_0}}, \\
		\dot V^\Delta_{\varepsilon, k, x} := \varepsilon\sscpSmall{\Delta\dot V}{(K_{k, x})_{\varepsilon, x_0}}. &
	\end{matrix}
\end{align*}
Further, $\scp{a}{b}_{w}=\sum_{k=1}^{N_\varepsilon}\sum_{x\in\cX_\varepsilon}w^h_\varepsilon(x)a_{k, x}b_{k, x}$ defines a bilinear form on
\newline $\R^{\{1,\dots,N_\varepsilon\}\times\cX_\varepsilon}$, and $\scp{\cdot}{\cdot}_{\abs{w}}$, $\scp{\cdot}{\cdot}_{w^2}$, defined analogously, are non-negative definite.
Set
\begin{align*}
		\bar I_{\varepsilon, N_\varepsilon} &:= \scpSmall{\bar X^\Delta_{\varepsilon, \cdot-1, \cdot}}{\bar X^\Delta_\varepsilon}_w, \\
		\langle\bar M_\varepsilon\rangle_{N_\varepsilon} &:= \sigma_K^2\scpSmall{\bar X^\Delta_{\varepsilon, \cdot-1, \cdot}}{\bar X^\Delta_{\varepsilon, \cdot-1, \cdot}}_{w^2}, \\
		\bar B_{\varepsilon, N_\varepsilon} &:= \scpSmall{\bar X^\Delta_{\varepsilon, \cdot-1, \cdot}}{\bar X^{\Delta-\Delta}_\varepsilon}_w.
\end{align*}

\begin{remark}
	Using the above defined terms, we will separately study the effects of the dynamic noise (via $\bar X$), the initial condition (via $\tilde X$) and the static noise (via $\dot V$). The notation reflects the fact that the first two quantities are governed by the dynamics of the heat semigroup.
	The next result allows to split the moment conditions from Assumption $(A_{\NonparIndex})$ into these three domains.
\end{remark}

\begin{proposition}[reduction of moments]\label{prop:splitting-reduction}
	Assume Assumption {\assumptionB} b).
	If
	for some $1\leq\NonparIndex<2$:
	\begin{enumerate}
		\item[a)] $\bE\bar I_{\varepsilon, N_\varepsilon}\sim N_\varepsilon$, $\bE\langle\bar M_\varepsilon\rangle_{N_\varepsilon} = \cO(N_\varepsilon\abs{\cX_\varepsilon}^{-1}h^{-d})$, and \\
		$\rVar(\normSmall{\bar X^{\Delta}_\varepsilon}_{\abs w}^2) + \rVar(\normSmall{\bar X^{\Delta}_{\varepsilon,\cdot-1,\cdot}}_{\abs w}^2)=\cO(h^2N_\varepsilon^{2+\eta})$ for any $\eta>0$,
		\item[b)] $\bE\bar B_{\varepsilon, N_\varepsilon} = \cO(h^{\NonparIndex} N_\varepsilon)$ and $\rVar(\normSmall{\bar X^{\Delta-\Delta}_\varepsilon}_{\abs w}^2) = \cO(h^{2\NonparIndex+2}N_\varepsilon^2)$,
		\item[c)] $\normSmall{\tilde X^\Delta_{\varepsilon, \cdot-1, \cdot}}_{\abs{w}}^2 + \normSmall{\tilde X^\Delta_{\varepsilon}}_{\abs{w}}^2 = o(h^{\NonparIndex-1}N_\varepsilon)$ and $\normSmall{\tilde X^{\Delta-\Delta}_{\varepsilon}}_{\abs{w}}^2 = \cO(h^{\NonparIndex+1}N_\varepsilon)$,
		\item[d)] $\bE\normSmall{\dot V^\Delta_{\varepsilon, \cdot-1, \cdot}}_{\abs{w}}^2 + \rVar(\normSmall{\dot V^\Delta_{\varepsilon, \cdot-1, \cdot}}_{\abs{w}}^2) + \rVar(\normSmall{\dot V^\Delta_{\varepsilon}}_{\abs{w}}^2) = \cO(N_\varepsilon)$,
	\end{enumerate}
	then $(I_{\varepsilon, N_\varepsilon}, B_{\varepsilon, N_\varepsilon}, \langle M_\varepsilon\rangle_{N_\varepsilon})$ satisfies Assumption $(A_{\NonparIndex})$.
\end{proposition}

\begin{proof} \
	Plug $Y = \bar X + \tilde X + \varepsilon\dot V$ into the definition of $\widehat X^\Delta_{\varepsilon, k, x}$ and expand the expected values and variances. Taking into account that $\bar X,\dot V$ are independent, and $\tilde X$ is deterministic, we verify the different parts of Assumption $(A_\NonparIndex)$.
	
	\emph{Assumption $(A_\NonparIndex)$ a):}
	$\bE I_{\varepsilon, N_\varepsilon} = \bE \bar I_{\varepsilon, N_\varepsilon} + \scpSmall{\tilde X^\Delta_{\varepsilon, \cdot-1, \cdot}}{\tilde X^\Delta_\varepsilon}_w$, and with $\bE \bar I_{\varepsilon, N_\varepsilon}\sim N_\varepsilon$ by a) and $\absSmall{\scpSmall{\tilde X^\Delta_{\varepsilon, \cdot-1, \cdot}}{\tilde X^\Delta_\varepsilon}_w}\leq\normSmall{\tilde X^\Delta_{\varepsilon, \cdot-1, \cdot}}_{\abs{w}}\normSmall{\tilde X^\Delta_{\varepsilon}}_{\abs{w}}=o(N_\varepsilon)$ by c), we have $\bE I_{\varepsilon, N_\varepsilon}\sim N_\varepsilon$.
		Next, using Lemma \ref{lem:abstract-variance-bounds} together with Young's inequality and a), c) and d),
		\begin{align*}
			\rVar(I_{\varepsilon, N_\varepsilon})
				&\lesssim \sum_{Z^{(1)}, Z^{(2)}\in \{\bar X^\Delta_\varepsilon, \tilde X^\Delta_\varepsilon, \dot V^\Delta_\varepsilon\}}\rVar(\scpSmall{Z^{(1)}_{\cdot-1, \cdot}}{Z^{(2)}}_w) \\
				&\lesssim \rVar(\normSmall{\bar X^{\Delta}_{\varepsilon,\cdot-1,\cdot}}_{\abs w}^2) + \rVar(\normSmall{\bar X^{\Delta}_{\varepsilon}}_{\abs w}^2) + \rVar(\normSmall{\dot V^{\Delta}_{\varepsilon,\cdot-1,\cdot}}_{\abs w}^2) \\
				&\quad\quad + \rVar(\normSmall{\dot V^{\Delta}_{\varepsilon}}_{\abs w}^2) + \normSmall{\tilde X^{\Delta}_{\varepsilon,\cdot-1,\cdot}}_{\abs w}^4 + \normSmall{\tilde X^{\Delta}_{\varepsilon}}_{\abs w}^4 \\
				&= o(N_\varepsilon^2).
		\end{align*}
	
	\emph{Assumption $(A_\NonparIndex)$ b):}
		We conclude from c), d) and Assumption {\assumptionB} b)
		\begin{align*}
			\normSmall{\tilde X^\Delta_{\varepsilon, \cdot-1, \cdot}}_{w^2}^2 + \bE\normSmall{\dot V^\Delta_{\varepsilon, \cdot-1, \cdot}}_{w^2}^2
				&\lesssim \sup_{x\in\cX_\varepsilon}\abs{w^h_\varepsilon(x)}\left(\normSmall{\tilde X^\Delta_{\varepsilon, \cdot-1, \cdot}}_{\abs{w}}^2 + \bE\normSmall{\dot V^\Delta_{\varepsilon, \cdot-1, \cdot}}_{\abs{w}}^2\right) \\
				&\lesssim \abs{\cX_\varepsilon}^{-1}h^{-d}N_\varepsilon,
		\end{align*}
		so $\bE\langle M_\varepsilon\rangle_{N_\varepsilon} = \bE\langle \bar M_\varepsilon\rangle_{N_\varepsilon} + \sigma_K^2\normSmall{\tilde X^\Delta_{\varepsilon, \cdot-1, \cdot}}_{w^2}^2 + \sigma_K^2\bE\normSmall{\dot V^\Delta_{\varepsilon, \cdot-1, \cdot}}_{w^2}^2\lesssim N_\varepsilon\abs{\cX_\varepsilon}^{-1}h^{-d}$ by a).
	
	\emph{Assumption $(A_\NonparIndex)$ c):}
		First, with c),
		\begin{align*}
			\abs{\scpSmall{\tilde X^\Delta_{\varepsilon, \cdot-1, \cdot}}{\tilde X^{\Delta-\Delta}_\varepsilon}_w}
				&\lesssim \left(\normSmall{\tilde X^\Delta_{\varepsilon, \cdot-1, \cdot}}_\abs{w}^2\normSmall{\tilde X^{\Delta-\Delta}_\varepsilon}_\abs{w}^2\right)^\frac{1}{2}
				\lesssim h^{\NonparIndex} N_\varepsilon,
		\end{align*}
		and 
		writing
		\begin{align}
			B_{\varepsilon, N_{\varepsilon}} = \scpSmall{\bar X_{\varepsilon, \cdot-1, \cdot}^\Delta+\tilde X_{\varepsilon, \cdot-1, \cdot}^\Delta+\dot V_{\varepsilon, \cdot-1, \cdot}^\Delta}{\bar X_{\varepsilon}^{\Delta-\Delta}+\tilde X_\varepsilon^{\Delta-\Delta}}_w,
		\end{align}
		we obtain $\absSmall{\bE B_{\varepsilon, N_\varepsilon}} = \absSmall{\bE \bar B_{\varepsilon, N_\varepsilon} + \scpSmall{\tilde X^\Delta_{\varepsilon, \cdot-1, \cdot}}{\tilde X^{\Delta-\Delta}_\varepsilon}_w}\lesssim h^{\NonparIndex} N_\varepsilon$ by b).
		Next, taking into account that the variance of a sum is bounded up to a constant by the sum of the variances of the separate terms, we obtain from Lemma \ref{lem:abstract-variance-bounds} and a), b), c) and d),
		\begin{align*}
			\rVar(B_{\varepsilon, N_\varepsilon})
				&\lesssim \sum_{Z\in \{\bar X^\Delta_\varepsilon, \tilde X^\Delta_\varepsilon, \dot V^\Delta_\varepsilon\}}\left(\rVar(\scpSmall{Z_{\cdot-1, \cdot}}{\bar X^{\Delta-\Delta}_\varepsilon}_w) + \rVar(\scpSmall{Z_{\cdot-1, \cdot}}{\tilde X^{\Delta-\Delta}_\varepsilon}_w)\right) \\
				&\lesssim \left(\sqrt{\rVar(\normSmall{\bar X^{\Delta}_{\varepsilon,\cdot-1,\cdot}}_{\abs w}^2)} + \sqrt{\rVar(\normSmall{\dot V^{\Delta}_{\varepsilon,\cdot-1,\cdot}}_{\abs w}^2)} + \normSmall{\tilde X^{\Delta}_{\varepsilon,\cdot-1,\cdot}}_{\abs w}^2\right) \\
				&\quad\quad \times \left(\sqrt{\rVar(\normSmall{\bar X^{\Delta-\Delta}_{\varepsilon}}_{\abs w}^2)} + \normSmall{\tilde X^{\Delta-\Delta}_{\varepsilon}}_{\abs w}^2\right) \\
				&= \cO((hN_\varepsilon^{1+\eta/2} + N_\varepsilon^{1/2} + h^{\NonparIndex-1}N_\varepsilon) \times (h^{\NonparIndex+1}N_\varepsilon + h^{\NonparIndex+1}N_\varepsilon))
		\end{align*}
		for any $\eta>0$, and this is in fact $\cO(h^{2\NonparIndex}N_\varepsilon^2)$
		since $N_\varepsilon^{1/2}\sim\delta N_\varepsilon\lesssim hN_\varepsilon\lesssim h^{\NonparIndex-1}N_\varepsilon$, and $hN_\varepsilon^{1+\eta/2} \lesssim h^{\NonparIndex-1}N_\varepsilon$ for $\NonparIndex<2$ and $\eta>0$ small enough.
				Finally, we conclude that $\bE B_{\varepsilon, N_\varepsilon}^2 = \rVar(B_{\varepsilon, N_\varepsilon}) + (\bE B_{\varepsilon, N_\varepsilon})^2 = \cO(h^{2\NonparIndex}N_\varepsilon^2)$.
\end{proof}

\subsection{Bounds for $\bar X_t$}

The proof of the next result is similar to the moment bounds in the proof of Theorem \ref{thm:parametric} that do not stem from the static noise.

\begin{lemma}[bounds related to the empirical Fisher information] \label{lem:parametric:moments}
	In the setting of Theorem \ref{thm:connect-conditions},
	it holds as $\varepsilon\rightarrow 0$:
	\begin{enumerate}
		\item[a)] $\bE \bar I_{\varepsilon, N_\varepsilon}\sim N_\varepsilon$,
		\item[b)] $\bE\langle \bar M_\varepsilon\rangle_{N_\varepsilon}\lesssim N_\varepsilon\abs{\cX_\varepsilon}^{-1}h^{-d}$,
		\item[c)] $\rVar(\normSmall{\bar X^{\Delta}_{\varepsilon,\cdot-1,\cdot}}_{\abs w}^2) + \rVar(\normSmall{\bar X^{\Delta}_\varepsilon}_{\abs w}^2)=\cO(h^2N_\varepsilon^{2+\eta})$ for any $\eta>0$.
	\end{enumerate}
\end{lemma}

\begin{remark}
	This lemma remains true if $K$ has $\Delta$-order one in $d=1$, and without restriction on the order of $K$ in $d\geq 2$.
\end{remark}

\begin{proof}[Proof of Lemma \ref{lem:parametric:moments}] \
	We will use that
	\begin{align*}
		\bar X^\Delta_{\varepsilon, k, x} = \sscpSmall{\Delta\bar X}{(K_{k, x})_{\varepsilon, x_0}} = \varepsilon^{-1}\sscpSmall{\bar X}{(\Delta K_{k, x})_{\varepsilon, x_0}} = \sscpSmall{\bar X^{(\varepsilon)}}{\Delta K_{k, x}}_\varepsilon.
	\end{align*}
	First,
	\begin{align*}
		\abs{\frac{\bE \bar I_{\varepsilon, N_\varepsilon}}{N_\varepsilon} - \frac{1}{N_\varepsilon}\sum_{k=1}^{N_\varepsilon}\bE[\sscpSmall{\bar X^{(0)}}{\Delta K_{k, 0}}_0\sscpSmall{\bar X^{(0)}}{\Delta K_{k-1, 0}}_0]} \hspace{-9cm} & \\
			&\leq \frac{1}{N_\varepsilon}\sum_{k=1}^{N_\varepsilon}\sum_{x\in\cX_\varepsilon}\abs{w_\varepsilon^h(x)}\abs{\bE[\bar X^\Delta_{\varepsilon, k, x}\bar X^\Delta_{\varepsilon, k-1, x}] - \bE[\sscpSmall{\bar X^{(0)}}{\Delta K_{k, x}}_0\sscpSmall{\bar X^{(0)}}{\Delta K_{k-1, x}}_0]} \\
			&\leq \sup_{1\leq k\leq N_\varepsilon}\sup_{x\in\cB[\varepsilon^{1/2}/h]}
			\left|\bE[\sscpSmall{\bar X^{(\varepsilon)}}{\Delta K_{k, x}}_\varepsilon\sscpSmall{\bar X^{(\varepsilon)}}{\Delta K_{k-1, x}}_\varepsilon] \right. \\
			&\hspace{5cm} \left.- \bE[\sscpSmall{\bar X^{(0)}}{\Delta K_{k, x}}_0\sscpSmall{\bar X^{(0)}}{\Delta K_{k-1, x}}_0]\right| \rightarrow 0
	\end{align*}
	due to Lemma \ref{lem:cov:conv:eps} a).
	Now Lemma \ref{lem:cov:conv:T} identifies the limit: $\bE \bar I_{\varepsilon, N_\varepsilon}/N_\varepsilon\rightarrow C_\infty(\Delta K, \Delta K_{-1, 0})$.
	Next,
	\begin{align*}
		\bE[\langle \bar M_\varepsilon\rangle_{N_\varepsilon}]
			&= \sigma_K^2\sum_{k=1}^{N_\varepsilon}\sum_{x\in\cX_\varepsilon}w_\varepsilon^h(x)^2\bE[\sscpSmall{\bar X^{(\varepsilon)}}{\Delta K_{k-1, x}}_\varepsilon^2] \\
			&\lesssim N_\varepsilon \sup_{x\in\cX_\varepsilon}\abs{w_\varepsilon^h(x)}\sup_{1\leq k\leq N_\varepsilon}\sup_{x\in\cB[\varepsilon^{1/2}/h]}\abs{\bE[\sscpSmall{\bar X^{(\varepsilon)}}{\Delta K_{k-1, x}}_\varepsilon^2]} \\
			&\lesssim N_\varepsilon\abs{\cX_\varepsilon}^{-1}h^{-d},
	\end{align*}
	where we used Lemma \ref{lem:cov:conv:eps} b) and Assumption {\assumptionB}.
	Next, with Lemma \ref{lem:cov:bound:kl} and $\kappa<d/2$,
	\begin{align*}
		\rVar(\normSmall{\bar X^{\Delta}_{\varepsilon}}_{\abs w}^2)
			&= 2\sum_{k,\ell=1}^{N_\varepsilon}\sum_{x, y\in\cX_\varepsilon}\abs{w_\varepsilon^h(x)w_\varepsilon^h(y)}\bE[\bar X^\Delta_{\varepsilon, k, x}\bar X^\Delta_{\varepsilon, \ell, y}]^2 \\
			&\lesssim \sum_{k,\ell=1}^{N_\varepsilon}\sum_{x, y\in\cX_\varepsilon}\abs{w_\varepsilon^h(x)w_\varepsilon^h(y)} ((1\wedge\abs{k-\ell}^{-\kappa}) + h)^2 \\
			&\lesssim (hN_\varepsilon)^2 + \sum_{k,\ell=1}^{N_\varepsilon}(1\wedge\abs{k-\ell}^{-2\kappa}).
	\end{align*}
	This is $\cO((hN_\varepsilon)^2 + N_\varepsilon^{1+\eta})$ for any $\eta>0$ (in $d\geq 2$, one could even choose $\eta=0$), and as $N_\varepsilon^{1+\eta}\sim\delta^2N_\varepsilon^{2+\eta}$, we obtain a final bound of the form $\cO(h^2N_\varepsilon^{2+\eta})$, as claimed.
	Finally, the argument for $\rVar(\normSmall{\bar X^{\Delta}_{\varepsilon,\cdot-1,\cdot}}_{\abs w}^2)$ is identical.
\end{proof}

Statement and proof of the next lemma are inspired by \cite{StrauchTiepner2024}.

\begin{lemma}[bounds related to the bias] \label{lem:moments:bias}
	In the setting of Theorem \ref{thm:connect-conditions}, let $\NonparIndex=\beta$ in $d\geq 3$, and $\NonparIndex=1$ otherwise. It holds as $\varepsilon\rightarrow 0$:
	\begin{enumerate}
		\item[a)] $\bE\bar B_{\varepsilon, N_\varepsilon} = \cO(h^{\NonparIndex} N_\varepsilon)$,
		\item[b)] $\rVar(\normSmall{\bar X^{\Delta-\Delta}_\varepsilon}_{\abs w}^2) = \cO(h^{2\NonparIndex+2}N_\varepsilon^2)$.
	\end{enumerate}
\end{lemma}

\begin{remark}
	The proof of this lemma uses that $K$ has $\Delta$-order $1$ in $d=2$, and $\Delta$-order $2$ otherwise.
\end{remark}

\begin{proof}[Proof of Lemma \ref{lem:moments:bias}] \
	\begin{enumerate}
		\item[a)] We write with $R=2$ and $J(s, t, k) = [\abs{t-s}, t+s+2k-2]$:
		\begin{align*}
			\bE \bar B_{\varepsilon, N_\varepsilon}
				\hspace{-1cm}&\hspace{1cm}
				= \bE\sum_{k=1}^{N_\varepsilon}\sum_{x\in\cX_\varepsilon}w_\varepsilon^h(x)\sscpSmall{\bar X^{(\varepsilon)}}{\Delta K_{k-1, x}}_\varepsilon\sscpSmall{[\Delta_{\vartheta_\varepsilon}-\Delta_{\vartheta_0}] \bar X^{(\varepsilon)}}{K_{k, x}}_\varepsilon \\
				&= \frac{\sigma^2}{2}\sum_{k=1}^{N_\varepsilon}\sum_{x\in\cX_\varepsilon}w_\varepsilon^h(x)\int_0^R\int_0^R\int_{J(s, t, k)}\scpSmall{\Delta \shift_xK_{t}}{S_\varepsilon(r)[\Delta_{\vartheta_\varepsilon}-\Delta_{\vartheta_0}]\shift_xK_{s-1}}_0\diff r\diff s\diff t \\
				&= \frac{\sigma^2}{2}\sum_{k=1}^{N_\varepsilon}\sum_{x\in\cX_\varepsilon}w_\varepsilon^h(x)\int_0^R\int_0^R\int_{J(s, t, k)}\scpSmall{\Delta \shift_xK_{t}}{S_0(r)[\Delta_{\vartheta_\varepsilon}-\Delta_{\vartheta_0}]\shift_xK_{s-1}}_0\diff r\diff s\diff t \\
				&\quad + \frac{\sigma^2}{2}\sum_{k=1}^{N_\varepsilon}\sum_{x\in\cX_\varepsilon}w_\varepsilon^h(x)\int_0^R\int_0^R\int_{J(s, t, k)} \\
				&\hspace{3cm} \langle\Delta \shift_xK_{t},[S_\varepsilon - S_0](r)[\Delta_{\vartheta_\varepsilon}-\Delta_{\vartheta_0} - \delta\Delta'_{(\nabla\vartheta)_0}]\shift_xK_{s-1}\rangle_0\diff r\diff s\diff t \\
				&\quad + \frac{\sigma^2}{2}\sum_{k=1}^{N_\varepsilon}\sum_{x\in\cX_\varepsilon}w_\varepsilon^h(x)\int_0^R\int_0^R\int_{J(s, t, k)} \\
				&\hspace{3cm} \scpSmall{\Delta \shift_xK_{t}}{[S_\varepsilon - S_0](r)\delta\Delta'_{(\nabla\vartheta)_0}\shift_xK_{s-1}}_0\diff r\diff s\diff t \\
				&= I + II + III.
		\end{align*}
		
		We bound these terms separately:
		\begin{itemize}
			\item[$I$:]
			With $\vartheta_\varepsilon^x(z) = \vartheta(\delta (z + x) + x_0)$,
			and using the compact support of $K$,
			\begin{align*}
				\scp{\Delta \shift_xK_{t}}{S_0(r)[\Delta_{\vartheta_\varepsilon} - \Delta_{\vartheta_0}]\shift_xK_{s-1}}_0 \hspace{-4cm} & \\
					&= \scp{\Delta K_{t}}{S_0(r)\nabla\cdot(\vartheta_\varepsilon^x-\vartheta_0)\nabla K_{s-1}}_{L^2(\R^d)} \\
					&= \scp{\nabla\Delta K_{t}}{S_0(r)(\vartheta_\varepsilon^x-\vartheta_0)\nabla K_{s-1}}_{L^2(\R^d)^d}.
			\end{align*}
			Next, for some $\xi^*\in[x_0, \delta(z+x)+x_0]\subset\R^d$ by the mean-value theorem
			\begin{align*}
				\sum_{x\in\cX_\varepsilon}w^h_\varepsilon(x)(\vartheta_\varepsilon^x(z)-\vartheta_0(z))\nabla K(s-1, z) \hspace{-5cm} & \\
					&= \sum_{x\in\cX_\varepsilon}w^h_\varepsilon(x)[\nabla\vartheta(\xi^*)\cdot\delta (z+x)]\nabla K(s-1, z) \\
					&= \delta\sum_{x\in\cX_\varepsilon}w^h_\varepsilon(x)[(\nabla\vartheta(\xi^*)-\nabla\vartheta(0))\cdot(z+x)]\nabla K(s-1, z) \\
					&\quad\quad + \delta\sum_{x\in\cX_\varepsilon}w^h_\varepsilon(x)[\nabla\vartheta(0)\cdot (z+x)]\nabla K(s-1, z).
			\end{align*}
			With respect to the last term, note that $\sum_{x\in\cX_\varepsilon}w_\varepsilon^h(x)x_i = 0$, and
			\begin{align*}
				\scp{\nabla\Delta K_{t}}{S_0(r)z_i\nabla K_{s-1}}_{L^2(\R^d)^d} = 0
			\end{align*}
			for $1\leq i\leq d$,
			because $K(v,\cdot)$ is an even or odd function along each spatial coordinate for $v\in\R$.
			Furthermore,
			\begin{align*}
				\norm{\delta[(\nabla\vartheta(\xi^*)-\nabla\vartheta(0))\cdot(z+x)]\nabla K(s-1, z)}_{L^2(\R^d)^d} \hspace{-5cm} & \\
					&\leq \delta \norm{\abs{\nabla\vartheta(\xi^*)-\nabla\vartheta(0)}\abs{z + x}\abs{\nabla K(s-1, z)}}_{L^2(\R^d)^d} \\
					&\lesssim \delta^\beta\norm{\abs{z + x}^{\beta}\abs{\nabla K(s-1, z)}}_{L^2(\R^d)^d} \\
					&\lesssim \delta^\beta(1 + \abs{x}^\beta) \lesssim h^\beta.
			\end{align*}
			Putting things together, using Lemma \ref{lem:bound:semigroup:heatkernel},
			\begin{align*}
				\left|\sum_{k=1}^{N_\varepsilon}\sum_{x\in\cX_\varepsilon}w^h_\varepsilon(x)\int_0^R\int_0^R\int_{\abs{t-s}}^{t+s+2k-2}\right. \hspace{-5cm }\\
					& \left.\scp{\nabla\Delta K_{t}}{S_0(r)(\vartheta_\varepsilon^x-\vartheta_0)\nabla K_{s-1}}_{L^2(\R^d)^d}\diff u\diff s\diff t\vphantom{\int_0^R}\right| \\
					&\lesssim h^{\beta} N_\varepsilon \sup_k\int_0^R\int_0^R\int_0^{2R+2k-2}\norm{S_0(r)\nabla\Delta K_{t}}_{L^2(\R^d)^d}\diff u\diff s\diff t \\
					&\lesssim h^\beta N_\varepsilon.
			\end{align*}
			
			\item[$II$:]
			Let $p=5$, $\alpha=0$ for $d\geq 3$, and $p=5$, $\alpha=1$ in $d=2$, and $p>1/(\beta-1)$, $\alpha=\beta$ for $d=1$.
			Define $q$ via $1/p+1/q=1$.
			We use Lemma \ref{lem:bound:semigroup:uniform} (noting that $K$ has $\Delta$-order at least one in $d=1$) and Lemma \ref{lem:bias:Laplacian} in order to obtain
			(writing $U = 2R+2k-2$):
			\begin{align*}
				\int_0^R\int_0^R\int_{\abs{t-s}}^{t+s+2k-2}\scpSmall{\Delta \shift_xK_{t}}{S_\varepsilon(r)[\Delta_{\vartheta_\varepsilon}-\Delta_{\vartheta_0}-\delta\Delta'_{(\nabla\vartheta)_0}]\shift_xK_{s-1}}_0\diff r\diff s\diff t \hspace{-12cm} & \\
					&\lesssim \int_0^R\int_0^R\int_{0}^{U}\norm{S_\varepsilon(r)\Delta \shift_xK_{t}}_p\norm{[\Delta_{\vartheta_\varepsilon}-\Delta_{\vartheta_0}-\delta\Delta'_{(\nabla\vartheta)_0}]\shift_xK_{s-1}}_q\diff r\diff s\diff t \\
					&\lesssim \delta^\beta\int_0^R\int_0^R\int_{0}^{U}(1\wedge r^{-\alpha/2-d/(2q)})\norm{\shift_xK_{t}}_{\bar W^{2, 1}_\alpha\cap \bar W^{2, p}_\alpha}\norm{\shift_xK_{s-1}}_{\bar W^{2, q}_\beta}\diff r\diff s\diff t \\
					&\lesssim \delta^\beta\sup_{t\geq 0}\norm{\shift_xK_{t}}_{\bar W^{2, 1}_\alpha\cap \bar W^{2, p}_\alpha}\sup_{s\geq 0}\norm{\shift_xK_{s-1}}_{\bar W^{2, q}_\beta}.
			\end{align*}
			The term involving $S_0$ is treated identically. Thus Lemma \ref{lem:bound:space:Sobolev} gives
			\begin{align*}
				\abs{II}
					&\lesssim N_\varepsilon h^\beta(h\delta^{-1})^\alpha\sim N_\varepsilon h^\beta
			\end{align*}
			as $h\sim\delta$ in $d\in\{1, 2\}$.
			
			\item[$III$:]
			Using Lemma \ref{lem:LaplaceDiff:properties} c),
			\begin{align*}
				\Delta'_{(\nabla\vartheta)_0}\shift_xK_{s-1}
					&= \shift_x \Delta \FuncY_{s-1} + (\nabla\vartheta(x_0)\cdot x)\shift_x\Delta K_{s-1}
			\end{align*}
			for some kernel $\FuncY_{s-1}$ having $\Delta$-order one less than $K$.
			We consider without loss of generality only the second term.
			With $\abs{x}\lesssim h\delta^{-1}$,
			\begin{align*}
				\abs{\int_0^R\int_0^R\int_{\abs{t-s}}^{t+s+2k-2}\scpSmall{\Delta \shift_xK_{t}}{[S_\varepsilon - S_0](r)\delta (\nabla\vartheta(x_0)\cdot x)\shift_x\Delta K_{s-1}}_0\diff r\diff s\diff t} \hspace{-11cm} & \\
					&\lesssim h\int_0^R\int_0^R\int_{0}^{2R+2k-2}\abs{\scpSmall{\Delta \shift_xK_{t}}{[S_\varepsilon - S_0](r)\Delta\shift_x K_{s-1}}_0}\diff r\diff s\diff t.
			\end{align*}
			Using Corollary \ref{cor:semigroup-for-statistics}, this is $\cO(h)$ in dimension one and two, and $\cO(h^\beta)$ in $d\geq 3$.
			Here we use that the $\Delta$-order of $K$ is at least two in $d\geq 3$, and at least one otherwise.
			Thus $\abs{III}\lesssim h^{\NonparIndex}N_\varepsilon$ in all cases.
		\end{itemize}
		
		\item[b)] Decompose
		\begin{align*}
			\bar X^{\Delta-\Delta}_{\varepsilon, k, x}
				&= \sscpSmall{[\Delta_{\vartheta}-\Delta_{\vartheta(x_0)}]\bar X}{(K_{k, x})_{\varepsilon, x_0}}
				= \sscpSmall{\bar X^{(\varepsilon)}}{[\Delta_{\vartheta_\varepsilon}-\Delta_{\vartheta_0}]K_{k, x}}_\varepsilon \\
				&= \sscpSmall{\bar X^{(\varepsilon)}}{[\Delta_{\vartheta_\varepsilon}-\Delta_{\vartheta_0} - \delta\Delta'_{(\nabla\vartheta)_0}] K_{k, x}}_\varepsilon + \delta\sscpSmall{\bar X^{(\varepsilon)}}{\Delta'_{(\nabla\vartheta)_0} K_{k, x}}_\varepsilon \\
				&=: \bar G_{k, x} + \delta \tilde G_{k, x}.
		\end{align*}
		With Lemma \ref{lem:abstract-variance-bounds} and Young's inequality, we obtain
		\begin{align*}
			\rVar(\normSmall{\bar X^{\Delta-\Delta}_\varepsilon}_{\abs w}^2) \lesssim \rVar(\normSmall{\bar G}_\abs{w}^2) + \rVar(\delta^2\normSmall{\tilde G}_\abs{w}^2).
		\end{align*}
		By Lemma \ref{lem:cov:DeltaApproximation}, with $g_d(N_\varepsilon)$ as defined there,
		we have
		\begin{align*}
			\bE[\bar G_{k, x}^2] \lesssim h^{2\beta}g_d(N_\varepsilon),
		\end{align*}
		and consequently
		\begin{align*}
			\rVar(\normSmall{\bar G}_\abs{w}^2)
				&=2\sum_{k,\ell=1}^{N_\varepsilon}\sum_{x,y\in\cX_\varepsilon}\abs{w_\varepsilon^h(x)w_\varepsilon^h(y)}\bE[\bar G_{k,x}\bar G_{\ell,y}]^2 \\
				&\leq 2\left(\sum_{k=1}^{N_\varepsilon}\sum_{x\in\cX_\varepsilon}\abs{w_\varepsilon^h(x)}\bE\left[\bar G_{k, x}^2\right]\right)^2 \lesssim h^{4\beta}N_\varepsilon^2g_d(N_\varepsilon)^2.
		\end{align*}
		This is $\cO(h^{4\NonparIndex}N_\varepsilon^2)$ (in particular, $\cO(h^{2\NonparIndex+2}N_\varepsilon^2)$) in all cases considered:
		First, in $d\geq 3$, we have $g_d(N_\varepsilon)\equiv 1$,
		whereas in $d=2$, $h^{4\beta}N_\varepsilon^2g_d(N_\varepsilon)^2=h^4N_\varepsilon^2(h^{2\beta-2}\ln(N_\varepsilon))^2\lesssim h^4N_\varepsilon^2$.
		Finally, in $d=1$, using $\beta\geq 3/2$, we have $h^{4\beta}N_\varepsilon^2g_d(N_\varepsilon)^2=h^4N_\varepsilon^2(h^{2\beta-2}N_\varepsilon^{1/2})^2\lesssim h^4N_\varepsilon^2$ due to $N_\varepsilon\sim\delta^{-2}\sim h^{-2}$.
		Now consider $\tilde G_{k, x}$. Lemma \ref{lem:cov:bound:kl} gives (using that $K$ has $\Delta$-order at least two in $d=1$, and at least one otherwise):
		\begin{align*}
			\delta^2\bE[\tilde G_{k, x}\tilde G_{\ell, y}] &\lesssim h^2((1\wedge\abs{k-1-\ell}^{-\kappa}) + h)
		\end{align*}
		for any $\kappa<d/2$, and therefore
		\begin{align*}
			\rVar(\delta^2\normSmall{\tilde G}_\abs{w}^2)
				&=2\delta^4\sum_{k,\ell=1}^{N_\varepsilon}\sum_{x,y\in\cX_\varepsilon}\abs{w_\varepsilon^h(x)w_\varepsilon^h(y)}\bE[\tilde G_{k,x}\tilde G_{\ell,y}]^2 \\
				&\lesssim h^4\sum_{k,\ell=1}^{N_\varepsilon}((1\wedge\abs{k-1-\ell}^{-\kappa}) + h)^2.
		\end{align*}
		This is $\cO(h^4N_\varepsilon^{1+\eta} + h^6N_\varepsilon^2)$
		for any $\eta>0$ in $d=1$, and $\eta=0$ otherwise.
		In $d\in\{1,2\}$, this is clearly $\cO(h^4N_\varepsilon^2)$.
		If $d\geq 3$, we have $h^4N_\varepsilon\lesssim h^6N_\varepsilon^2$ and $h^6N_\varepsilon^2\lesssim h^{2\beta+2}N_\varepsilon^2$ as $\beta<2$.
	\end{enumerate}
\end{proof}

\subsection{Initial Condition and Noise}

\begin{lemma}[bounds for the initial condition] \label{lem:bounds-initialcondition}
	In the setting of Theorem \ref{thm:connect-conditions}, with $\NonparIndex=\beta$ in $d\geq 3$ and $\NonparIndex=1$ otherwise, it holds with $h\sim \delta^{(2+d)/(2\NonparIndex+d)}$:
	\begin{enumerate}
		\item[a)] $\normSmall{\tilde X^\Delta_{\varepsilon, \cdot-1, \cdot}}_{\abs{w}}^2 + \normSmall{\tilde X^\Delta_{\varepsilon}}_{\abs{w}}^2 = o(h^{\NonparIndex-1}N_\varepsilon)$
		\item[b)] $\normSmall{\tilde X^{\Delta-\Delta}_{\varepsilon}}_{\abs{w}}^2 = \cO(h^{\NonparIndex+1}N_\varepsilon)$
	\end{enumerate}
\end{lemma}

\begin{proof}
	We will use that $X_0\in L^p(\cD)$ for some $p>2$ in $d\in\{1, 2\}$, and $p>p_0$ in $d\geq 3$, where
	\begin{align}\label{eq:p-initialcondition}
		p_0 = 2\times\frac{1+z}{1-z}\quad\quad\mathrm{with}\quad\quad z = \frac{6(\beta-1)}{d(d+1+\beta)}.
	\end{align}
	Note that for $1<\beta<2$, we have $2<p_0<14/3$.
	With
	\begin{align*}
		\tilde X^\Delta_{\varepsilon, k, x}
			&= \sscpSmall{\Delta\tilde X}{(K_{k, x})_{\varepsilon, x_0}}
			= \sscpSmall{\tilde X^{(\varepsilon)}}{\Delta K_{k, x}}_\varepsilon
			= \sscpSmall{S_\varepsilon(\cdot)X_0^{(\varepsilon)}}{\Delta K_{k, x}}_\varepsilon,
	\end{align*}
	and using
	\begin{align*}
		\normSmall{X_0^{(\varepsilon)}}_{L^p(\cD_\varepsilon)} = \varepsilon^{-1/2}\normSmall{X_0^{1/\varepsilon}}_{L^p(\cD_\varepsilon)} = \varepsilon^{-1/2}\varepsilon^{\frac{d}{4}-\frac{d}{2p}}\normSmall{X_0}_{L^p(\cD)}\sim \delta^{-1+\frac{d}{2}-\frac{d}{p}}
	\end{align*}
	whenever $X_0\in L^p$,
	we have with $\alpha=1$:
	\begin{align*}
		\abs{\tilde X^\Delta_{\varepsilon, k, x}}
			&= \abs{\int_0^\infty\scp{S_\varepsilon(t)X_0^{(\varepsilon)}}{\Delta K_{k, x}}_0\diff t}
			= \abs{\int_0^1\scp{S_\varepsilon(t+k)X_0^{(\varepsilon)}}{\Delta \shift_xK_t}_0\diff t} \\
			&\leq \normSmall{X_0^{(\varepsilon)}}_{L^p}\int_0^1\norm{S_\varepsilon(t+k)\Delta \shift_xK_t}_{L^q}\diff t \\
			&\lesssim \normSmall{X_0^{(\varepsilon)}}_{L^p}\sup_{t\geq 0; x\in\cB[\delta/h]}\norm{\shift_xK_t}_{\bar W^{2, 1}_\alpha\cap \bar W^{2, q}_\alpha}\int_0^1(1\wedge (t+k)^{-\alpha/2-d/(2p)})\diff t \\
			&\lesssim \delta^{-1+d/2-d/p}(h\delta^{-1})^{\alpha}(1\wedge k^{-\alpha/2-d/(2p)}),
	\end{align*}
	where we allow for $k=0$ (in this case $1\wedge\infty$ evaluates to one). Thus
	\begin{align*}
		\normSmall{\tilde X^\Delta_{\varepsilon}}_{\abs{w}}^2
			&= \sum_{k=1}^{N_\varepsilon}\sum_{x\in\cX_\varepsilon}\abs{w_\varepsilon^h(x)}(\tilde X^\Delta_{\varepsilon, k, x})^2
			\lesssim \delta^{-2+d-2d/p}(h\delta^{-1})^2\sum_{k=1}^{\infty}k^{-1-d/p}.
	\end{align*}
	In $d\in\{1, 2\}$, this is of order $\delta^{-2+d(1-2/p)}\sim h^{d(1-2/p)}N_\varepsilon = o(N_\varepsilon) = o(h^{\NonparIndex-1}N_\varepsilon)$, and in $d\geq 3$, choosing $p$ as in \eqref{eq:p-initialcondition}, the last panel is of order $\delta^{-2+d-2d/p}h^2N_\varepsilon = o(h^{\beta-1}N_\varepsilon)$.
	The bound for $\normSmall{\tilde X^\Delta_{\varepsilon, \cdot-1, \cdot}}_{\abs{w}}^2$ works analogously.
	Next, we have
	\begin{align*}
		\abs{\tilde X^{\Delta-\Delta}_{\varepsilon, k, x}}
			&= \abs{\int_0^1\scp{S_\varepsilon(t+k)X_0^{(\varepsilon)}}{[\Delta_{\vartheta_\varepsilon}-\Delta_{\vartheta_0}] \shift_xK_t}_0\diff t} \\
			&\leq \abs{\int_0^1\scp{S_\varepsilon(t+k)X_0^{(\varepsilon)}}{[\Delta_{\vartheta_\varepsilon}-\Delta_{\vartheta_0}-\delta\Delta'_{(\nabla\vartheta)_0}] \shift_xK_t}_0\diff t} \\
			&\quad\quad + \delta\abs{\int_0^1\scp{S_\varepsilon(t+k)X_0^{(\varepsilon)}}{\Delta'_{(\nabla\vartheta)_0} \shift_xK_t}_0\diff t}.
	\end{align*}
	With Lemma \ref{lem:LaplaceDiff:properties} c), the second term can be treated as before (because $K$ has $\Delta$-order $\geq 1$ by assumption), but with an additional factor $h^2$, leading to an overall bound of order $o(h^{\NonparIndex+1}N_\varepsilon)$.
	The remaining term is bounded by
	\begin{align*}
			&\normSmall{X_0^{(\varepsilon)}}_{L^p}\int_0^1\norm{S_\varepsilon(t+k)[\Delta_{\vartheta_\varepsilon}-\Delta_{\vartheta_0}-\delta\Delta'_{(\nabla\vartheta)_0}] \shift_xK_t}_{L^q}\diff t \\
			&\lesssim \hspace{-0.05cm}\normSmall{X_0^{(\varepsilon)}}_{L^p}\hspace{-0.25cm}\sup_{\tiny\begin{matrix}t\geq 0; \\ x\in\cB[\delta/h]\end{matrix}}\hspace{-0.15cm}\norm{[\Delta_{\vartheta_\varepsilon}-\Delta_{\vartheta_0}-\delta\Delta'_{(\nabla\vartheta)_0}]\shift_xK_t}_{\bar W^{0, 1}_0\cap \bar W^{0, q}_0}\hspace{-0.1cm}\int_0^1(1\wedge (t+k)^{-d/(2p)})\diff t \\
			&\lesssim \delta^{-1+d/2-d/p}h^\beta(1\wedge k^{-d/(2p)}),
	\end{align*}
	thus
	\begin{align*}
		\normSmall{\tilde X^{\Delta-\Delta}_{\varepsilon}}_{\abs{w}}^2
			&= \sum_{k=1}^{N_\varepsilon}\sum_{x\in\cX_\varepsilon}\abs{w_\varepsilon^h(x)}(\tilde X^{\Delta-\Delta}_{\varepsilon, k, x})^2
			\lesssim \delta^{-2+d-2d/p}h^{2\beta}\sum_{k=1}^{N_\varepsilon}k^{-d/p} + h^{\NonparIndex+1}N_\varepsilon.
	\end{align*}
	In $d\geq 3$, with $p=2$, the first term has order $h^{2\beta}\delta^{-2}\sim h^{2\beta}N_\varepsilon = o(h^{\beta+1}N_\varepsilon)$.
	For $d=2$, again with $p=2$ and using $\beta>1$, the same term has order $o(h^{\NonparIndex+1}N_\varepsilon)$.
	Finally, for $d=1$, with $\beta\geq 3/2$ and $p=2$, the term has the order $\cO(h^2N_\varepsilon) = \cO(h^{\NonparIndex+1}N_\varepsilon)$.
\end{proof}

The static noise terms are treated in a similar way as in the proof of Theorem \ref{thm:parametric}, although the weighted scalar product is different:

\begin{lemma}[bounds for the static noise] \label{lem:moments-noise}
	Under Assumption {\assumptionB} b), it holds
	\begin{align}
		\bE\normSmall{\dot V^\Delta_{\varepsilon, \cdot-1, \cdot}}_{\abs{w}}^2 + \rVar(\normSmall{\dot V^\Delta_{\varepsilon, \cdot-1, \cdot}}_{\abs{w}}^2) + \rVar(\normSmall{\dot V^\Delta_{\varepsilon}}_{\abs{w}}^2) = \cO(N_\varepsilon).
	\end{align}
\end{lemma}

\begin{proof} \
	First,
	\begin{align*}
		\bE\normSmall{\dot V^\Delta_\varepsilon}_\abs{w}^2
			&= \varepsilon^2\sum_{k=1}^{N_\varepsilon}\sum_{x\in\cX_\varepsilon}\absSmall{w^h_\varepsilon(x)}\norm{\Delta(K_{k, x})_{\varepsilon, x_0}}_{L^2(\cT\times\cD)}^2 \\
			&= \sum_{k=1}^{N_\varepsilon}\sum_{x\in\cX_\varepsilon}\absSmall{w^h_\varepsilon(x)}\norm{\Delta K_{k, x}}_{L^2(\cT_\varepsilon\times\cD_\varepsilon)}^2 = \sum_{k=1}^{N_\varepsilon}\sum_{x\in\cX_\varepsilon}\absSmall{w^h_\varepsilon(x)}\norm{\Delta K}_{L^2(\R\times\R^d)}^2 \\
			&\lesssim N_\varepsilon.
	\end{align*}
	With Lemma \ref{lem:abstract-variance-bounds} and $\bE[(\dot V^\Delta_{\varepsilon, k, x})^2]=\varepsilon^2\norm{\Delta(K_{k, x})_{\varepsilon, x_0}}_{L^2(\cT\times\cD)}^2=\norm{\Delta K}_{L^2(\R\times\R^d)}^2$,
	we see
	\begin{align*}
		\rVar(\normSmall{\dot V^\Delta_{\varepsilon}}_{\abs{w}}^2)
			&= 2\sum_{k,\ell=1}^{N_\varepsilon}\sum_{x, y\in\cX_\varepsilon}\abs{w^h_\varepsilon(x)w^h_\varepsilon(y)}\bE[\dot V^\Delta_{\varepsilon, k, x}\dot V^\Delta_{\varepsilon, \ell, y}]^2 \\
			&= 2\sum_{k=1}^{N_\varepsilon}\sum_{x\in\cX_\varepsilon}\abs{w^h_\varepsilon(x)}^2\norm{\Delta K}_{L^2(\R\times\R^d)}^2 \lesssim N_\varepsilon.
	\end{align*}
	This implies the claim as $\dot V$ is homogeneous in time.
\end{proof}

\subsection{The Covariance of $\bar X$}

In this section assume $\delta\lesssim h\lesssim 1$, with $\delta=\sqrt{\varepsilon}=o(1)$.
The notation $\bar X^{(\varepsilon)}$ and $\bar X^{(0)}$ refers to the rescaled localization of $\bar X$, as explained in Section \ref{sec:localization}.

\begin{lemma}[covariance under shift] \label{lem:cov:form}
	Let $\varphi, \psi\in \cC_\varepsilon$,
	and $k,\ell\in\N_0$, $x, y\in \cD_\varepsilon$.
	Assume that the support of $\varphi_{0, x}$, $\psi_{0, y}$ is contained in $[0,R]\times\cD_\varepsilon$ for some $R>0$.
	Then
	\begin{align*}
		\bE[\sscpSmall{\bar X^{(\varepsilon)}}{\varphi_{k, x}}_\varepsilon\sscpSmall{\bar X^{(\varepsilon)}}{\psi_{\ell, y}}_\varepsilon] \hspace{-2cm}& \\
			&= \frac{\sigma^2}{2}\int_0^R\int_0^R\int_{\abs{(t+k) - (s+\ell)}}^{(t+k) + (s+\ell)}\scp{\shift_x\varphi_t}{S_\varepsilon(r)\shift_y\psi_s}_0\diff r\diff s\diff t.
	\end{align*}
\end{lemma}

\begin{proof}
	Using Lemma \ref{lem:localization-X},
	we have
	\begin{align*}
		\bE[\sscpSmall{\bar X^{(\varepsilon)}}{\varphi_{k, x}}_\varepsilon\sscpSmall{\bar X^{(\varepsilon)}}{\psi_{\ell, y}}_\varepsilon] \hspace{-3cm} & \\
			&= \frac{\sigma^2}{2}\int_0^\infty\int_0^\infty\int_{\abs{t-s}}^{t+s}\scp{\varphi_{k, x}\left(t, \cdot\right)}{S_\varepsilon(r)\psi_{\ell, y}\left(s, \cdot\right)}_\varepsilon\diff r\diff s\diff t \\
			&= \frac{\sigma^2}{2}\int_0^\infty\int_0^\infty\int_{\abs{(t+k)-(s+\ell)}}^{(t+k)+(s+\ell)}\scp{\varphi\left(t, \cdot-x\right)}{S_\varepsilon(r)\psi\left(s, \cdot-y\right)}_\varepsilon\diff r\diff s\diff t,
	\end{align*}
	which implies the claim.
\end{proof}

\begin{lemma}[ergodic limit] \label{lem:cov:conv:T}
	Let $\varphi,\psi\in \cC_0$, and assume one of them is of $\Delta$-order one if $d\in\{1,2\}$.
	Then
	\begin{align}
		\lim_{k\rightarrow\infty}\sup_{x\in\R}\abs{\bE[\sscpSmall{\bar X^{(0)}}{\varphi_{k, x}}_0\sscpSmall{\bar X^{(0)}}{\psi_{k, x}}_0] - C_\infty(\varphi, \psi)} = 0
	\end{align}
	with
	$C_\infty(\varphi, \psi)$ as in \eqref{eq:limitingconstant}.
\end{lemma}

\begin{proof}
	We have
	\begin{align*}
		\bE[\sscpSmall{\bar X^{(0)}}{\varphi_{k, x}}_0\sscpSmall{\bar X^{(0)}}{\psi_{k, x}}_0] \hspace{-3cm} & \\
			&= \frac{\sigma^2}{2}\int_0^\infty\int_0^\infty\int_{\abs{t-s}}^{\infty}\mathbbm{1}_{r\leq t+s+2k}\scp{\shift_x\varphi_t}{S_0(r)\shift_x\psi_s}_0\diff r\diff s\diff t,
	\end{align*}
	and due to shift invariance of convolution with the heat kernel, we can assume $x=0$.
	The integrand converges in a pointwise sense for fixed $r,s,t$. Apply dominated convergence together with Lemma \ref{lem:bound:semigroup:heatkernel}, using that $\varphi$ or $\psi$ has $\Delta$-order one in $d\in\{1,2\}$.
\end{proof}

\begin{lemma}[uniform covariance approximation] \label{lem:cov:conv:eps}
	Let $\varphi, \psi\in \cC_0$
	such that \newline $\mathrm{supp}(\varphi_{0, x}), \mathrm{supp}(\psi_{0,x})\subseteq \cT_\varepsilon\times\cB^*[\varepsilon^{1/2}]$ for $x\in\cB[\varepsilon^{1/2}/h]$.
	In the setting of Theorem \ref{thm:parametric}, assuming that $\varphi$ or $\psi$ has $\Delta$-order one in $d\leq 2$, and in the setting of Theorem \ref{thm:connect-conditions}, assuming that $\varphi$ or $\psi$ has $\Delta$-order $(3-d)_+$, we have:
	\begin{enumerate}

	\item[a)]
	\begin{align*}
		\sup_{0\leq k\leq C\varepsilon^{-1}}\sup_{x\in\cB[\varepsilon^{1/2}/h]}\left|\bE[\sscpSmall{\bar X^{(\varepsilon)}}{\varphi_{k, x}}_\varepsilon\sscpSmall{\bar X^{(\varepsilon)}}{\psi_{k, x}}_\varepsilon] \right. \hspace{-3cm} & \\
		& \left. - \bE[\sscpSmall{\bar X^{(0)}}{\varphi_{k, x}}_0\sscpSmall{\bar X^{(0)}}{\psi_{k, x}}_0]\right|\rightarrow 0,
	\end{align*}
	
	\item[b)]
	\begin{align*}
		\sup_{0\leq k\leq C\varepsilon^{-1}}\sup_{x\in\cB[\varepsilon^{1/2}/h]}\abs{\bE[\sscpSmall{\bar X^{(\varepsilon)}}{\varphi_{k, x}}_\varepsilon\sscpSmall{\bar X^{(\varepsilon)}}{\psi_{k, x}}_\varepsilon]}<\infty.
	\end{align*}
	\end{enumerate}
\end{lemma}

\begin{proof} \
	Let $R>0$ such that the temporal support of $\varphi$, $\psi$ is contained in $[0,R]$.
	By Lemma \ref{lem:cov:form},
	it is true that
	\begin{align*}
		\sup_{0\leq k\leq C\varepsilon^{-1}}\sup_{x\in\cB[\varepsilon^{1/2}/h]}\left|\bE[\sscpSmall{\bar X^{(\varepsilon)}}{\varphi_{k, x}}_\varepsilon\sscpSmall{\bar X^{(\varepsilon)}}{\psi_{k, x}}_\varepsilon] \right. \hspace{-7.5cm} & \\
			&\hspace{5cm}\left. - \bE[\sscpSmall{\bar X^{(0)}}{\varphi_{k, x}}_0\sscpSmall{\bar X^{(0)}}{\psi_{k, x}}_0]\right| \hspace{-13cm} & \\
			&\lesssim \sup_{0\leq k\leq C\varepsilon^{-1}}\int_0^R\int_0^R\int_{0}^{t+s+2k}\sup_{x\in\cB[\varepsilon^{1/2}/h]}\abs{\scp{\shift_x\varphi_t}{S_\varepsilon(r)-S_0(r)\shift_x\psi_s}_0}\diff r\diff s\diff t \\
			&\lesssim \int_0^R\int_0^R\int_{0}^{t+s+2C\varepsilon^{-1}}\sup_{x\in\cB[\varepsilon^{1/2}/h]}\abs{\scp{\shift_x\varphi_t}{S_\varepsilon(r)-S_0(r)\shift_x\psi_s}_0}\diff r\diff s\diff t,
	\end{align*}
	which tends to zero by Corollary \ref{cor:semigroup-for-statistics}.
	Finally, b) is a consequence of a) and Lemma \ref{lem:cov:conv:T}.
\end{proof}

\begin{lemma}[uniform operator approximation] \label{lem:cov:DeltaApproximation}
	Let $\varphi\in \cC_0$.
	Let $g_1(N)=N^{1/2}$, $g_2(N)=\ln(N)$, and $g_d(N)=1$ for $d\geq 3$.
	Then we have
	\begin{align}
		\sup_{1\leq k\leq N_\varepsilon}\sup_{x\in\cB[\varepsilon^{1/2}/h]}\bE\left[\sscpSmall{[\Delta_{\vartheta_\varepsilon}-\Delta_{\vartheta_0}] \bar X^{(\varepsilon)}}{\varphi_{k, x}}_\varepsilon^2\right] &\lesssim h^2g_d(N_\varepsilon), \nonumber \\			
			\sup_{1\leq k\leq N_\varepsilon}\sup_{x\in\cB[\varepsilon^{1/2}/h]}\bE\left[\sscpSmall{[\Delta_{\vartheta_\varepsilon}-\Delta_{\vartheta_0} - \varepsilon^{1/2}\Delta'_{(\nabla\vartheta)_0}] \bar X^{(\varepsilon)}}{\varphi_{k, x}}_\varepsilon^2\right] &\lesssim h^{2\beta}g_d(N_\varepsilon). \nonumber
	\end{align}
\end{lemma}

\begin{proof}
	By an argument that is identical to the proof of Lemma \ref{lem:cov:form}, we find
		\begin{align*}
			\bE\left[\sscpSmall{[\Delta_{\vartheta_\varepsilon}-\Delta_{\vartheta_0}] \bar X^{(\varepsilon)}}{\varphi_{k, x}}_\varepsilon^2\right] \hspace{-5cm} &\\
				&= \frac{\sigma^2}{2}\int_0^R\int_0^R\int_{\abs{t-s}}^{t+s+2k}\scpSmall{[\Delta_{\vartheta_\varepsilon}-\Delta_{\vartheta_0}]\shift_x\varphi_{t}}{S_\varepsilon(r)[\Delta_{\vartheta_\varepsilon}-\Delta_{\vartheta_0}]\shift_x\varphi_s}_0\diff r\diff s\diff t \\
				&\lesssim \int_0^R\hspace{-0.1cm}\int_0^R\hspace{-0.1cm}\int_{\abs{t-s}}^{t+s+2k}\hspace{-0.15cm}\norm{S_\varepsilon(r/2)[\Delta_{\vartheta_\varepsilon}-\Delta_{\vartheta_0}]\shift_x\varphi_{t}}_0\norm{S_\varepsilon(r/2)[\Delta_{\vartheta_\varepsilon}-\Delta_{\vartheta_0}]\shift_x\varphi_{s}}_0\diff r\diff s\diff t.
		\end{align*}
		Now with Lemma \ref{lem:bound:semigroup:uniform}, Lemma \ref{lem:bias:Laplacian} and Lemma \ref{lem:bound:space:Sobolev}, with $\alpha=1$,
		we have
		\begin{align*}
			\norm{S_\varepsilon(r/2)[\Delta_{\vartheta_\varepsilon}-\Delta_{\vartheta_0}]\shift_x\varphi_{t}}_0
				&\lesssim (1\wedge r^{-d/4})\norm{[\Delta_{\vartheta_\varepsilon}-\Delta_{\vartheta_0}]\shift_x\varphi_{t}}_{L^1\cap L^2} \\
				&\lesssim \delta^{\alpha}(1\wedge r^{-d/4})\norm{\shift_x\varphi_{t}}_{\bar W^{2, 1}_\alpha\cap\bar W^{2, 2}_\alpha} \\
				&\lesssim h^{\alpha}(1\wedge r^{-d/4})\norm{\varphi_{t}}_{\bar W^{2, 1}_\alpha\cap\bar W^{2, 2}_\alpha},
		\end{align*}
		and equally for the other factor, thus
		\begin{align*}
			\sup_{1\leq k\leq N_\varepsilon}\sup_{x\in\cB[\varepsilon^{1/2}/h]}\bE\left[\sscpSmall{[\Delta_{\vartheta_\varepsilon}-\Delta_{\vartheta_0}] \bar X^{(\varepsilon)}}{\varphi_{k, x}}_\varepsilon^2\right] \hspace{-7cm} & \\
				&\lesssim h^{2\alpha}\sup_{1\leq k\leq N_\varepsilon}\int_0^R\int_0^R\int_{\abs{t-s}}^{t+s+2k}(1\wedge r^{-d/2})\norm{\varphi_{t}}_{\bar W^{2, 1}_\alpha\cap\bar W^{2, 2}_\alpha}\norm{\varphi_{s}}_{\bar W^{2, 1}_\alpha\cap\bar W^{2, 2}_\alpha}\diff r\diff s\diff t \\
				&\lesssim h^{2\alpha}g_d(N_\varepsilon).
		\end{align*}
		The second claim is proven analogously with $\alpha=\beta$.
\end{proof}

\begin{lemma}[asymptotic independence] \label{lem:cov:bound:kl}
	Let $0\leq k, \ell\leq N_\varepsilon$.
	Let $\kappa<d/2$.
	Then for all $\varphi,\psi\in \cC_0$,
	such that $\mathrm{supp}(\varphi_{0, x}), \mathrm{supp}(\psi_{0, x})\subseteq\cT_\varepsilon\times\cD_\varepsilon$ for $x\in\cB[\varepsilon^{1/2}/h]$ the following is true:
	
	\begin{enumerate}
		\item[a)]
	Let $\vartheta$ be constant and $h\sim 1$. Then we have
	\begin{align}\label{eq:cov:bound:kl:parametric}
			\sup_{x, y\in \cB[\varepsilon^{1/2}/h]}\abs{\bE[\sscpSmall{\bar X^{(\varepsilon)}}{\Delta\varphi_{k, x}}_\varepsilon\sscpSmall{\bar X^{(\varepsilon)}}{\Delta\psi_{\ell,y}}_\varepsilon]}
				&\lesssim (1\wedge\abs{k-\ell}^{-\kappa}).
	\end{align}
	
		\item[b)]
	In the setting of Theorem \ref{thm:connect-conditions}, assuming in $d=1$ that $\varphi$ or $\psi$ is of $\Delta$-order one, we have
	\begin{align}\label{eq:cov:bound:kl}
			\sup_{x, y\in \cB[\varepsilon^{1/2}/h]}\abs{\bE[\sscpSmall{\bar X^{(\varepsilon)}}{\Delta\varphi_{k, x}}_\varepsilon\sscpSmall{\bar X^{(\varepsilon)}}{\Delta\psi_{\ell,y}}_\varepsilon]}
				&\lesssim (1\wedge\abs{k-\ell}^{-\kappa}) + h.
	\end{align}
	Moreover, if $\varphi, \psi$ are of $\Delta$-order one (one of them of order two in $d=1$), then we have
	\begin{align}
			\varepsilon^{1/2}\sup_{x, y\in \cB[\varepsilon^{1/2}/h]}\abs{\bE[\sscpSmall{\bar X^{(\varepsilon)}}{\Delta\varphi_{k, x}}_\varepsilon\sscpSmall{\bar X^{(\varepsilon)}}{\Delta'_{(\nabla\vartheta)_0}\psi_{\ell,y}}_\varepsilon]} \hspace{-2cm} & \nonumber\\
				&\lesssim h((1\wedge\abs{k-\ell}^{-\kappa}) + h), \label{eq:cov:bound:kl:Delta} \\
			\varepsilon\sup_{x, y\in \cB[\varepsilon^{1/2}/h]}\abs{\bE[\sscpSmall{\bar X^{(\varepsilon)}}{\Delta'_{(\nabla\vartheta)_0}\varphi_{k, x}}_\varepsilon\sscpSmall{\bar X^{(\varepsilon)}}{\Delta'_{(\nabla\vartheta)_0}\psi_{\ell,y}}_\varepsilon]} \hspace{-2cm} & \nonumber\\
				&\lesssim h^2((1\wedge\abs{k-\ell}^{-\kappa}) + h). \label{eq:cov:bound:kl:DeltaDelta}
	\end{align}
	\end{enumerate}
\end{lemma}

\begin{proof} \
		The temporal support of $\shift_x\varphi, \shift_y\psi$ is contained in $[0,R]$ for some $R>0$.
		By an argument as in Lemma \ref{lem:cov:form}, we see
		\begin{align*}
			\bE[\sscpSmall{\bar X^{(\varepsilon)}}{\Delta\varphi_{k,x}}_\varepsilon\sscpSmall{\bar X^{(\varepsilon)}}{\Delta_{\vartheta_\varepsilon}(\psi_{\ell,y})}_\varepsilon] \hspace{-6cm} & \\
				&= \frac{\sigma^2}{2}\int_0^R\int_0^R\int_{\abs{(t+k) - (s+\ell)}}^{(t+k) + (s+\ell)}\scp{\Delta\shift_x\varphi_t}{S_\varepsilon(r)\Delta_{\vartheta_\varepsilon}\shift_y\psi_s}_0\diff r\diff s\diff t \\
				&= \frac{\sigma^2}{2}\int_0^R\int_0^R\scp{\Delta\shift_x\varphi_t}{[S_\varepsilon((t+k) + (s+\ell))-S_\varepsilon(\abs{(t+k) - (s+\ell)})]\shift_y\psi_s}_0\diff s\diff t,
		\end{align*}
		so for any $p,q>1$ with $1/p+1/q=1$, using Lemma \ref{lem:bound:semigroup:uniform} with $\alpha=0$, as well as Lemma \ref{lem:bound:space:Sobolev},
		we estimate
		\begin{align*}
			\abs{\bE[\sscpSmall{\bar X^{(\varepsilon)}}{\Delta\varphi_{k, x}}_\varepsilon\sscpSmall{\bar X^{(\varepsilon)}}{\Delta_{\vartheta_\varepsilon}(\psi_{\ell, y})}_\varepsilon]} \hspace{-6cm} & \\
				&\lesssim \int_0^R\int_0^R(\norm{S_\varepsilon(2((t+k)\wedge(s+\ell)))\shift_y\psi_s}_{L^q(\cD_\varepsilon)} + \norm{\shift_y\psi_s}_{L^q(\cD_\varepsilon)}) \\
				&\hspace{5cm} \times\norm{S_\varepsilon(\abs{t+k-s-\ell})\Delta\shift_x\varphi_t}_{L^p(\cD_\varepsilon)}\diff s\diff t \\
				&\lesssim \int_0^R\int_0^R\norm{S_\varepsilon(\abs{t+k-s-\ell})\Delta\shift_x\varphi_t}_{L^p(\cD_\varepsilon)}\diff s\diff t \\
				&\lesssim \sup_{t\geq 0}\norm{\shift_x\varphi_t}_{\bar W^{2, 1}_0\cap \bar W^{2, p}_0}\int_0^R\int_0^R1\wedge\abs{t+k-s-\ell}^{-d/(2q)}\diff s\diff t \\
				&\lesssim R^2(1\wedge(\abs{k-\ell}-R)^{-d/(2q)}) \lesssim 1\wedge \abs{k-\ell}^{-d/(2q)},
		\end{align*}
		where $q>1$ is arbitrary.
		Consequently, as $\vartheta$ is bounded away from zero, we have
		\begin{align*}
			\abs{\bE[\sscpSmall{\bar X^{(\varepsilon)}}{\Delta \varphi_{k, x}}_\varepsilon \sscpSmall{\bar X^{(\varepsilon)}}{\Delta \psi_{\ell, y}}_\varepsilon]} \hspace{-3cm} & \\
				&= \frac{1}{\vartheta(x_0)}\abs{\bE[\sscpSmall{\bar X^{(\varepsilon)}}{\Delta \varphi_{k, x}}_\varepsilon \sscpSmall{\bar X^{(\varepsilon)}}{\Delta_{\vartheta_0} \psi_{\ell, y}}_\varepsilon]} \\
				&\lesssim 1\wedge\abs{k-\ell}^{-d/(2q)} \\
				&\quad\quad + \abs{\bE[\sscpSmall{\bar X^{(\varepsilon)}}{\Delta \varphi_{k, x}}_\varepsilon \sscpSmall{\bar X^{(\varepsilon)}}{[\Delta_{\vartheta_0} - \Delta_{\vartheta_\varepsilon}] (\psi_{\ell, y})}_\varepsilon]}.
		\end{align*}
		This proves the parametric case, so assume the setting of Theorem \ref{thm:connect-conditions} for the rest of the proof.
		The last term is estimated as follows, with $\alpha=0$, $p=5$, $q=5/4$ in $d\geq 3$, and $\alpha=1$, $p=5$, $q=5/4$ in $d=2$, as well as $\alpha=\beta$, $1<q<1/(2-\beta)$, $p=q/(q-1)$ in $d=1$:
		\begin{align*}
			\abs{\bE[\sscpSmall{\bar X^{(\varepsilon)}}{\Delta \varphi_{k, x}}_\varepsilon \sscpSmall{\bar X^{(\varepsilon)}}{[\Delta_{\vartheta_0} - \Delta_{\vartheta_\varepsilon}] \psi_{\ell, y}}_\varepsilon]} \hspace{-7cm} & \\
				&= \frac{\sigma^2}{2}\abs{\int_0^R\int_0^R\int_{\abs{(t+k) - (s+\ell)}}^{(t+k) + (s+\ell)}\scp{\Delta\shift_x\varphi_t}{S_\varepsilon(r)[\Delta_{\vartheta_0} - \Delta_{\vartheta_\varepsilon}]\shift_y\psi_s}_0\diff r\diff s\diff t} \\
				&\lesssim \sup_{s\geq 0}\norm{[\Delta_{\vartheta_0} - \Delta_{\vartheta_\varepsilon}]\shift_y\psi_s}_{L^q(\R^d)}\int_0^R\int_0^R\int_{0}^{(t+k) + (s+\ell)}\norm{S_\varepsilon(r)\Delta\shift_x\varphi_t}_{L^p(\R^d)}\diff r\diff s\diff t \\
				&\lesssim \varepsilon^{1/2}\sup_{s\geq 0}\norm{\shift_y\psi_s}_{\bar W^{2, q}_{1}} \sup_{t\geq 0}\norm{\shift_x \varphi_t}_{\bar W^{2, 1}_\alpha\cap \bar W^{2, p}_\alpha}\int_0^\infty(1\wedge r^{-\alpha/2-d/(2q)})\diff r \\
				&\lesssim h(h\varepsilon^{-1/2})^\alpha,
		\end{align*}
		where we have used Lemma \ref{lem:bias:Laplacian}, Lemma \ref{lem:bound:semigroup:uniform} (here we need the $\Delta$-order of $\varphi$ to be one in $d=1$) and Lemma \ref{lem:bound:space:Sobolev}.
		Note that in all cases considered, $h(h\varepsilon^{-1/2})^\alpha\sim h$.
		This proves \eqref{eq:cov:bound:kl}.
		Next, write
		\begin{align*}
			\bE[\sscpSmall{\bar X^{(\varepsilon)}}{\Delta\varphi_{k, x}}_\varepsilon\sscpSmall{\bar X^{(\varepsilon)}}{\Delta'_{(\nabla\vartheta)_0}\psi_{\ell,y}}_\varepsilon] \hspace{-5cm} & \\
				&= \frac{\sigma^2}{2}\int_0^R\int_0^R\int_{\abs{(t+k) - (s+\ell)}}^{(t+k) + (s+\ell)}\scp{\Delta \shift_x \varphi_t}{S_\varepsilon(r)\Delta'_{(\nabla\vartheta)_0}\shift_y\psi_s}_0\diff r\diff s\diff t \\
				&= \frac{\sigma^2}{2}\int_0^R\int_0^R\int_{\abs{(t+k) - (s+\ell)}}^{(t+k) + (s+\ell)}\scp{\Delta \shift_x \varphi_t}{S_\varepsilon(r)\shift_y \Delta \bar\FuncY_{s}}_0\diff r\diff s\diff t \\
				&\quad\quad + \frac{\sigma^2}{2}\int_0^R\int_0^R\int_{\abs{(t+k) - (s+\ell)}}^{(t+k) + (s+\ell)}\scp{\Delta\shift_x \varphi_t}{S_\varepsilon(r)(\nabla\vartheta(x_0)\cdot y)\shift_y\Delta\psi_s}_0\diff r\diff s\diff t \\
				&= \bE[\sscpSmall{\bar X^{(\varepsilon)}}{\Delta\varphi_{k, x}}_\varepsilon\sscpSmall{\bar X^{(\varepsilon)}}{\Delta \bar\FuncY_{\ell,y}}_\varepsilon] \\
				&\hspace{3cm} + (\nabla\vartheta(x_0)\cdot y)\bE[\sscpSmall{\bar X^{(\varepsilon)}}{\Delta\varphi_{k, x}}_\varepsilon\sscpSmall{\bar X^{(\varepsilon)}}{\Delta \psi_{\ell,y}}_\varepsilon]
		\end{align*}
		with $\bar\FuncY$ determined by Lemma \ref{lem:LaplaceDiff:properties} c) (with $\ell=1$).
		If $\psi$ is of $\Delta$-order two, then $\bar\FuncY$ has $\Delta$-order one.
		Now \eqref{eq:cov:bound:kl:Delta} follows from \eqref{eq:cov:bound:kl}, taking into account $\varepsilon^{1/2}\abs{y}\lesssim h$.
		The proof of \eqref{eq:cov:bound:kl:DeltaDelta} works analogously.
\end{proof}

\section{Trotter--Kato Approximation}\label{sec:TrotterKato}

\subsection{Statement and Discussion}

Let $\vartheta\in C^\beta(\bar\cD)$ for $\beta\leq 2$.
We consider the case that $\delta\lesssim h\lesssim 1$.
The goal of this section is to prove:

\begin{theorem}[uniform Trotter--Kato theorem, general case] \label{thm:approx:semigroup}
	Let $p\geq 2$, $q\leq 2$ with $1=1/p+1/q$.
	Fix closed balls $\cB\subset\cB^*$ in $\cD$ centered at $x_0$.
	\begin{enumerate}
		\item[a)] Let $\ell_1,\ell_2\in\N$, and $0\leq\alpha\leq\beta$ with $d>(2-\alpha)q$.
		Assume $\ell_1\geq 2$ if $\alpha>1$,
		and $\ell_2>3/2-d/(2p)$.
		Then for $\FuncZ_1\in \bar W^{2\ell_1,1}_\alpha\cap \bar W^{2\ell_1,p}_\alpha$ and $\FuncZ_2\in \bar W^{2,1}_1\cap \bar W^{2\ell_2+2,q}_1$
		with $\mathrm{supp}(\shift_y\varphi_1),\mathrm{supp}(\shift_y\varphi_2)\subseteq\cB^*[\delta]$ for $y\in\cB[\delta/h]$, and $t\geq 0$:
		\begin{align*}
			\sup_{y_1, y_2\in\cB[\delta/h]}\abs{\scp{\Delta^{\ell_1}\shift_{y_1}\FuncZ_1}{(S_\delta(t) - S_0(t))\Delta^{\ell_2}\shift_{y_2}\FuncZ_2}_0} \hspace{-5cm} & \\
				&\lesssim \left(h(h\delta^{-1})^\alpha((t\wedge t^{1-\frac{\alpha}{2}-\frac{d}{2q}}) + (t\wedge t^{\frac{3}{2}-\ell_2-\frac{d}{2p}})) \right. \\
				&\hspace{3cm} \left. + [\delta^{2\ell_2+\frac{d}{q}} + \delta^{2\ell_2+2+\frac{d}{p}}(h\delta^{-1})^\alpha]e^{-\bar c\delta^{-2}t^{-1}}\right) \\
				&\quad\quad\quad\quad \times \norm{\FuncZ_1}_{\bar W^{2\ell_1,1}_\alpha\cap \bar W^{2\ell_1,p}_\alpha}\norm{\FuncZ_2}_{\bar W^{2,1}_1\cap \bar W^{2\ell_2+2,q}_1}.
		\end{align*}
		
		\item[b)] Let $\ell\in\N_0$, $\FuncZ\in \bar W^{2, 1}_1\cap \bar W^{2+2\ell, p}_1$
		with $\mathrm{supp}(\shift_y\varphi)\subseteq\cB^*[\delta]$ for $y\in\cB[\delta/h]$.
		Then
		\begin{align*}
			\sup_{0\leq t\leq T}\sup_{y\in\cB[\delta/h]}\norm{(S_\delta(t) - S_0(t))\Delta^\ell\shift_y\FuncZ}_{L^p(\R^d)}
			&\lesssim \delta^{2\ell+\frac{d}{q}-\frac{d-1}{p}}Te^{-\bar c\delta^{-2}T^{-1}}\normSmall{\FuncZ}_{L^1(\R^d)} \\
			&\quad\quad\quad + hC_T\normSmall{\FuncZ}_{\bar W^{2, 1}_1\cap \bar W^{2+2\ell, p}_1},
		\end{align*}
		with $C_T=\int_0^T(1\wedge s^{-\ell-\frac{d}{2q}})(1\vee s^\frac{1}{2})\diff s$, and
		in particular the right-hand side tends to zero for $h\rightarrow 0$.
		If $\vartheta$ is constant, it holds even
		\begin{align*}
			\sup_{0\leq t\leq T}\sup_{y\in\cB[\delta/h]}\norm{(S_\delta(t) - S_0(t))\Delta^\ell\shift_y\FuncZ}_{L^p(\R^d)}
			&\lesssim \delta^{2\ell+\frac{d}{q}-\frac{d-1}{p}}Te^{-\bar c\delta^{-2}T^{-1}}\normSmall{\FuncZ}_{L^1(\R^d)}.
		\end{align*}
	\end{enumerate}

\end{theorem}

\begin{proof}
	See Appendix \ref{sec:proof:approx:semigroup}, p. \pageref{proof:approx:semigroup}.
\end{proof}

\begin{remark} \
	\begin{enumerate}
		\item[a)] In order to obtain a bound that is integrable in time, it is necessary to consider the weak formulation in Theorem \ref{thm:approx:semigroup} a) compared to the norm estimate from b).
		\item[b)] We can understand the coefficients of the kernel norms appearing on the right-hand side in the panel in Theorem \ref{thm:approx:semigroup} a) as follows: $h(h\delta^{-1})^\alpha=\delta(h\delta^{-1})^{\alpha+1}$ comes from the approximation error of the generators of the semigroups of order $\delta$, together with a penalty $(h\delta^{-1})^{\alpha+1}$ for having estimates uniform over shifts in space up to distance $h\delta^{-1}$. The terms $t\wedge t^{1-\frac{\alpha}{2}-\frac{d}{2q}}$ and $t\wedge t^{\frac{3}{2}-\ell_2-\frac{d}{2p}}$ relate to contraction properties of the heat semigroup. The remaining terms appear due to boundary effects.
		\item[c)] The formulation involving a ball $\cB^*$ separated from the boundary of the domain is necessary only for the limit case $h\sim 1$, otherwise such condition is automatically satisfied as the active neighborhood of size $h$ shrinks. A similar effect occurs in \cite{StrauchTiepner2024}.
		\item[d)] The proof that we give is analytic in nature and should be contrasted to the stochastic proof of \cite[Proposition 3.5]{AltmeyerReiss2021}.
		There, the convergence of semigroups has been established via a stochastic representation, using a Feynman--Kac argument. For this, it is necessary to bound three terms, related to the inhomogeneity of $\vartheta$, the possible presence of reaction and advection terms, and the boundedness of the underlying domain. Such terms also appear in our approach (except that we do not consider lower order perturbations of the diffusion operator in this work).
		Notably, the error relating to the inhomogeneity of $\vartheta$ is handled using general approximation results for Feller processes stated in \cite{Kallenberg2002}, which are analytic in nature. In this sense, we provide a more detailed analysis by proving and using a quantitative version of the Trotter--Kato approximation.
	\end{enumerate}
\end{remark}

\subsection{Remaining Proofs from Section \ref{sec:semigroupapproximation}} \label{sec:proofs:semigroupapproximation}

The statements from Section \ref{sec:semigroupapproximation} are direct consequences of Theorem \ref{thm:approx:semigroup}. We provide their proofs here.

\begin{proof}[Proof of Theorem \ref{thm:semigroup-approximation-simplified}] \label{proof:semigroup-approximation-simplified} \
	\begin{enumerate}
		\item[a)] This follows immediately from Theorem \ref{thm:approx:semigroup} b).
		\item[b)] This follows from Theorem \ref{thm:approx:semigroup} a) with $\ell_1=2$, $\ell_2=3$, $\alpha<\beta$,
		and $q>1$ small enough, taking into account $\delta\lesssim 1\wedge t^{-1/2}$.
	\end{enumerate}
\end{proof}

\begin{proof}[Proof of Corollary \ref{cor:semigroup-for-statistics}] \label{proof:semigroup-for-statistics} \
	\begin{enumerate}
		\item[a)]
	By Theorem \ref{thm:semigroup-approximation-simplified} a),
	the integrand converges to zero as $\varepsilon\rightarrow 0$ for fixed $r, s, t\geq 0$.
	Then we can use dominated convergence, using Lemma \ref{lem:bound:semigroup:uniform} and Lemma \ref{lem:bound:semigroup:heatkernel}, together with Lemma \ref{lem:bound:space:Sobolev}:
	More precisely, in Lemma \ref{lem:bound:semigroup:uniform} set $\alpha=0$ for $d\geq 3$, $\alpha=1$ for $d=2$ and $\alpha=\beta$ for $d=1$, where in all cases $q>1$ is chosen small enough via H\"older's inequality.
	In $d\in\{1,2\}$, use that $h\sim\delta$, so $(h\delta^{-1})^\alpha\sim 1$ in all dimensions.
	
	In $d\geq 3$, Theorem \ref{thm:semigroup-approximation-simplified} b)
	implies that the left-hand side of \eqref{eq:uniform-covariance-convergence} is $\cO(h(h\delta^{-1})^{\alpha})$ for every $\alpha>1$,
	and this is $\cO(h^{\beta-1})$ because $\beta<2-2/(4+d)$ and $h\sim \delta^{\frac{2+d}{2\beta+d}}$ by assumption.
		
		\item[b)]
		For fixed $r, s, t\geq 0$, the integrand converges to zero by Theorem \ref{thm:semigroup-approximation-simplified} a). As before, we use dominated convergence, using H\"older's inequality, Lemma \ref{lem:bound:semigroup:parametric} and Lemma \ref{lem:bound:semigroup:heatkernel}. In $d\geq 3$, choose $p>d/(d-2)$, i.e. $q<d/2$, then the resulting bound is integrable in time. In $d\in\{1,2\}$, choose $p=q=2$ and $\ell=1$.
	\end{enumerate}
\end{proof}

\subsection{Proof of Theorem \ref{thm:approx:semigroup}} \label{sec:proof:approx:semigroup}

The main idea in the proof is to reduce the error in the semigroups to the error of their generators (up to boundary terms which need to be treated separately). The generators are easier to understand.
Our proceeding is inspired by the general theory in \cite{ItoKappel1998}, although we avoid the use of the resolvent that is crucial in the presentation therein.

\begin{lemma}[approximation error decomposition] \label{lem:approx:semigroup}
	For $\FuncZ\in W^{2,p}(\R^d)$, we have
	\begin{align*}
		S_\delta(t)\pi_\delta \FuncZ - S_0(t)\FuncZ
			&= R_1 + R_2 + R_3 + R_4 + R_5,
	\end{align*}
	where
	\begin{align*}
		R_1 &= (\pi_\delta - \mathrm{id})S_0(t)\FuncZ, \\
		R_2 &= S_\delta(t)(\pi_\delta-\bar\pi_\delta)\FuncZ, \\
		R_3 &= (\bar\pi_\delta-\pi_\delta)S_0(t)\FuncZ, \\
		R_4 &= \int_0^tS_\delta(t-s)(\pi_\delta-\bar\pi_\delta)S_0(s)\Delta_{\vartheta_0}\FuncZ\diff s, \\
		R_5 &= \int_0^tS_\delta(t-s)(\Delta_{\vartheta_\delta}\pi_\delta - \pi_\delta \Delta_{\vartheta_0})S_0(s)\FuncZ\diff s.
	\end{align*}
\end{lemma}

$R_1$ describes the discrepancy between the bounded and unbounded domain, $R_2$, $R_3$ and $R_4$ describe boundary effects, and $R_5$ is related to the inhomogeneity of $\vartheta$.
Note that if $\varphi$ is supported in $\cD_\delta$ (this is the standing assumption in Theorem \ref{thm:approx:semigroup}), then we can write $S_\delta(t)\pi_\delta \FuncZ - S_0(t)\FuncZ=S_\delta(t)\FuncZ - S_0(t)\FuncZ$. In this case, $R_2$ vanishes right away.

\begin{remark}\label{rem:approx:decomp:density}
	As $\bar\pi$ acts on $W^{2,p}(\R^d)$, the term $R_4$ is well-defined a priori only for $\FuncZ\in W^{4,p}(\R^d)$.
	However, since $S_\delta(t)\pi_\delta \FuncZ - S_0(t)\FuncZ$ and $R_1$, $R_2$, $R_3$, $R_5$ are well-defined in $L^2(\R^d)$ for all $\FuncZ\in W^{2,p}(\R^d)$, we can extend the domain of definition of $R_4$ to $W^{2,p}(\R^d)$ by a density argument.
\end{remark}

\begin{proof}[Proof of Lemma \ref{lem:approx:semigroup}]
	By Remark \ref{rem:approx:decomp:density}, we can restrict to $\FuncZ\in W^{4, p}(\R^d)$. We have
	\begin{align*}
		S_\delta(t)\pi_\delta \FuncZ-S_0(t)\FuncZ &= (S_\delta(t)\pi_\delta \FuncZ - \pi_\delta S_0(t)\FuncZ) + R_1 \\
			&= (S_\delta(t)\bar\pi_\delta \FuncZ - \bar\pi_\delta S_0(t)\FuncZ) + R_1 + R_2 + R_3.
	\end{align*}
	Define $E_\delta(t) := S_\delta(t)\bar\pi_\delta \FuncZ - \bar\pi_\delta S_0(t)\FuncZ$. Then $E_\delta(0)=0$ and
	\begin{align*}
		\partial_tE_\delta(t)
			&= A_\delta S_\delta(t)\bar\pi_\delta \FuncZ - \bar\pi_\delta A_0S_0(t)\FuncZ \\
			&= A_\delta E_\delta(t) + (A_\delta\bar\pi_\delta-\bar\pi_\delta A_0)S_0(t)\FuncZ,
	\end{align*}
	and consequently, by the variation-of-constants formula,
	\begin{align*}
		E_\delta(t) = \int_0^tS_\delta(t-s)(A_\delta\bar\pi_\delta - \bar\pi_\delta A_0)S_0(s)\FuncZ\diff s.
	\end{align*}
	Taking into account that $A_\delta\bar\pi_\delta = \Delta_{\vartheta_{\delta}}\bar\pi_\delta=\Delta_{\vartheta_\delta}\pi_\delta$ on $\cD_\delta$ by construction of $\bar\pi_\delta$, the claim now follows from
	\begin{align*}
		S_\delta(t-s)(A_\delta\bar\pi_\delta - \bar\pi_\delta A_0)S_0(s)\FuncZ
			&= S_\delta(t-s)(\Delta_{\vartheta_\delta}\pi_\delta - \pi_\delta\Delta_{\vartheta_0})S_0(s)\FuncZ \\
			&\quad\quad + S_\delta(t-s)(\pi_\delta-\bar\pi_\delta)\Delta_{\vartheta_0}S_0(s)\FuncZ.
	\end{align*}
\end{proof}

\begin{lemma}[domain and boundary effects] \label{lem:approx:domain} \
	Let $\ell\in\N_0$, and $p,q\geq 1$ with $1=1/p+1/q$.
	Let
	$\varphi\in W^{2\ell+2, p}(\R^d)\cap L^1(\R^d)$.
	Assume that there are closed balls $\cB\subset\cB^*$ in $\cD$ centered at $x_0$ with $\mathrm{supp}(\shift_y\varphi)\subseteq\cB^*[\delta]$ for $y\in\cB[\delta/h]$.
	Then:
	\begin{enumerate}
		\item[a)]
	For some $\bar c>0$ and all $t>0$:
	\begin{align*}
		\sup_{y\in\cB[\delta/h]}\norm{(\mathrm{id}-\pi_\delta)S_0(t)\Delta^\ell\shift_y\FuncZ}_{L^p(\R^d)}
			\lesssim t^{-\ell-\frac{d}{2q}+\frac{d-1}{2p}}e^{-\bar c\delta^{-2}t^{-1}}\normSmall{\FuncZ}_{L^1(\R^d)}.
	\end{align*}
	In particular, if $p\geq 2$,
	\begin{align*}
		\sup_{0\leq t\leq T}\sup_{y\in\cB[\delta/h]}\norm{(\mathrm{id}-\pi_\delta)S_0(t)\Delta^\ell\shift_y\FuncZ}_{L^p(\R^d)}
			\lesssim \delta^{2\ell+\frac{d}{q}-\frac{d-1}{p}}e^{-\bar c\delta^{-2}T^{-1}}\normSmall{\FuncZ}_{L^1(\R^d)}.
	\end{align*}
	
		\item[b)]
	\begin{align*}
		\sup_{0\leq t\leq T}\sup_{y\in\cB[\delta/h]}\norm{(\pi_\delta-\bar\pi_\delta)S_0(t)\Delta^\ell\shift_y\FuncZ}_{L^p(\cD_\delta)}
			\lesssim \delta^{2\ell+\frac{d}{q}}e^{-\bar c\delta^{-2}T^{-1}}\normSmall{\FuncZ}_{L^1(\R^d)}.
	\end{align*}
	\end{enumerate}
\end{lemma}

\begin{remark}
	By a standard density argument, the left-hand side of all panels appearing in Lemma \ref{lem:approx:domain} can be defined for $\FuncZ\in L^1(\R^d)$.
\end{remark}

\begin{proof}[Proof of Lemma \ref{lem:approx:domain}]
	First note that with Lemma \ref{lem:bound:properties:heatkernel} a), for $t>0$,
	\begin{align*}
		\abs{S_0(t)\Delta^\ell\shift_y\FuncZ(x)}
			&= \abs{(\Delta^\ell q_t*\shift_y \FuncZ)(x)} = \abs{\int_{\R^d}\Delta^\ell q_t(x-y-z) \FuncZ(z)\diff z} \\
			&\leq \sup_{z\in\mathrm{supp}(\FuncZ)} \abs{\Delta^\ell q_t(x-y-z)}\int_{\R^d}\absSmall{\FuncZ(z)}\diff z \\
			&\lesssim t^{-\ell}\sup_{z\in\mathrm{supp}(\FuncZ)} q_{2t}(x-y-z)\int_{\R^d}\absSmall{\FuncZ(z)}\diff z.
	\end{align*}
	As a last preparation, we write $C_*:=\mathrm{dist}(\partial\cD,\cB^*)>0$.
	
	\begin{enumerate}
		
		\item[a)]
		Let $R$ be the radius of $\cB^*$.
	For constants $c_i>0$, any unit vector $e\in S^{d-1}$, and taking into account $y\in\cB[\delta/h]$,
	\begin{align*}
		\norm{(\mathrm{id}-\pi_\delta)S_0(t)\Delta^\ell\shift_y\FuncZ}_{L^p(\R^d)}^p
			\hspace{-3cm}&\hspace{3cm}= \int_{\R^d\backslash\cD_\delta}\abs{S_0(t)\Delta^\ell\shift_y\FuncZ(x)}^p\diff x \\
			&\lesssim t^{-\ell p}\normSmall{\FuncZ}_{L^1(\R^d)}^p\int_{\R^d\backslash\cD_\delta}\sup_{z\in\mathrm{supp}(\FuncZ),y\in\cB[\delta/h]}q_{2t}(x-y-z)^p\diff x \\
			&\leq t^{-\ell p}\normSmall{\FuncZ}_{L^1(\R^d)}^p\int_{\{\abs{x}\geq (R+C_*)\delta^{-1}\}}\sup_{u\in\cB^*[\delta]}q_{2t}(x-u)^p\diff x \\
			&\leq t^{-\ell p}\normSmall{\FuncZ}_{L^1(\R^d)}^p\abs{S^{d-1}}\int_{(R+C_*)\delta^{-1}}^\infty q_{2t}((r-R\delta^{-1})e)^pr^{d-1}\diff r \\
			&\lesssim t^{-\ell p}\normSmall{\FuncZ}_{L^1(\R^d)}^p\int_{C_*\delta^{-1}}^\infty q_{2t}(re)^p(r+R\delta^{-1})^{d-1}\diff r \\
			&\lesssim \normSmall{\FuncZ}_{L^1(\R^d)}^pt^{-\ell p+d/2-pd/2}\int_{C_*\delta^{-1}}^\infty t^{-d/2}e^{-\frac{pr^2}{8t}}(r+R\delta^{-1})^{d-1}\diff r \\
			&\lesssim \normSmall{\FuncZ}_{L^1(\R^d)}^pt^{-\ell p+ d/2-pd/2}\int_{C_*\delta^{-1}t^{-1/2}}^\infty e^{-\frac{p}{8}r^2}(r\sqrt{t}+R\delta^{-1})^{d-1}\diff r \\
			&\lesssim \normSmall{\FuncZ}_{L^1(\R^d)}^pt^{-\ell p+d/2-pd/2}(t^{(d-1)/2}+(R\delta^{-1})^{d-1})e^{-c_1\delta^{-2}t^{-1}} \\
			&\lesssim \normSmall{\FuncZ}_{L^1(\R^d)}^pt^{-\ell p+d/2-pd/2}(t^{(d-1)/2}+(Rt^{1/2})^{d-1})e^{-c_2\delta^{-2}t^{-1}},
	\end{align*}
	which implies the first claim.
	For the second statement, use that for $\alpha\geq 0$, $$(\delta^{-2}t^{-1})^\alpha e^{-c_2\delta^{-2}t^{-1}}\lesssim e^{-c_3\delta^{-2}t^{-1}},$$ and further $e^{-c_3\delta^{-2}t^{-1}}\leq e^{-c_3\delta^{-2}T^{-1}}$.
	The case $t=0$ is trivial
	as $\shift_y\varphi$ has compact support in $\cD$.
		
		\item[b)]
	
	First observe that for any $\FuncY\in C^\infty(\R^d)$,
	\begin{align}\label{eq:approx:boundary:maximum-applied}
		\norm{(\pi_\delta-\bar\pi_\delta)\FuncY}_{L^p(\cD_\delta)} \lesssim \delta^{-d/p}\sup_{z\in\partial\cD_\delta}\abs{\FuncY(z)}.
	\end{align}
	Indeed, with $\FuncPotential\in C(\bar\cD)\cap C^2(\cD)$ being the solution of
	\eqref{eq:HarmonicPotential} with $\FuncZ=\FuncY$ (see \cite[Theorem 6.13]{GilbargTrudinger2001}),
	the maximum principle (see e.g. \cite[Theorem 3.5]{GilbargTrudinger2001}) implies that
	\begin{align*}
		\norm{(\pi_\delta-\bar\pi_\delta)\FuncY}_{L^p(\cD_\delta)}
			&= \norm{\FuncPotential}_{L^p(\cD_\delta)} \lesssim \delta^{-d/p}\sup_{z\in\cD_\delta}\abs{\FuncPotential(z)} \leq \delta^{-d/p}\sup_{z\in\partial\cD_\delta}\abs{\FuncPotential(z)},
	\end{align*}
	proving \eqref{eq:approx:boundary:maximum-applied}.
	Now, for $t>0$, set $\FuncY=S_0(t)\Delta^\ell\shift_y\FuncZ$.
	We get
	\begin{align*}
		\sup_{x\in\partial\cD_\delta}\abs{S_0(t)\Delta^\ell\shift_y\FuncZ(x)}
			&\lesssim t^{-\ell}\sup_{x\in\partial\cD_\delta}\sup_{z\in\mathrm{supp}(\FuncZ)}q_{2t}(x-y-z)\normSmall{\FuncZ}_{L^1(\R^d)}.
	\end{align*}
	Taking further the supremum over $y\in\cB[\delta/h]$, we see that
	\begin{align*}
		\sup_{y\in\cB[\delta/h]}\norm{(\pi_\delta-\bar\pi_\delta)S_0(t)\Delta^\ell\shift_y\FuncZ}_{L^p(\cD_\delta)} \hspace{-5cm} & \\
			&\lesssim \delta^{-d/p}t^{-\ell}\normSmall{\FuncZ}_{L^1(\R^d)}\sup_{x\in\partial\cD_\delta, z\in\mathrm{supp}(\FuncZ),y\in\cB[\delta/h]}q_{2t}(x-y-z) \\
			&\lesssim \delta^{-d/p}t^{-\ell-d/2}e^{-\frac{C_*^2\delta^{-2}}{8t}}\normSmall{\FuncZ}_{L^1(\R^d)} \\
			&\lesssim \delta^{2\ell+d-d/p}e^{-\bar c\delta^{-2}t^{-1}}\normSmall{\FuncZ}_{L^1(\R^d)} \\
			&\lesssim \delta^{2\ell+d/q}e^{-\bar c\delta^{-2}T^{-1}}\normSmall{\FuncZ}_{L^1(\R^d)}.
	\end{align*}
	Again, the case $t=0$ is included trivially
	as $\shift_y\varphi$ has compact support in $\cD$.
		
	\end{enumerate}
\end{proof}

\begin{proof}[Proof of Theorem \ref{thm:approx:semigroup}] \label{proof:approx:semigroup} \
	\begin{enumerate}
		\item[a)]
	Use Lemma \ref{lem:approx:semigroup} and bound the terms resulting from $R_1$ - $R_5$ separately.
	First, from $R_1$ we obtain
	\begin{align*}
		\scp{\Delta^{\ell_1}\shift_{y_1}\FuncZ_1}{(\pi_\delta-\mathrm{id})S_0(t)\Delta^{\ell_2}\shift_{y_2}\FuncZ_2}_0 \hspace{-3cm} & \\
			&= \scp{(\pi_\delta-\mathrm{id})\Delta^{\ell_1}\shift_{y_1}\FuncZ_1}{S_0(t)\Delta^{\ell_2}\shift_{y_2}\FuncZ_2}_0 = 0
	\end{align*}
	as the support of $\shift_{y_1}\FuncZ_1$ is contained in $\cD_\delta$.
	The term containing $R_2$ vanishes for the same reason, as $\mathrm{supp}(\shift_{y_2}\FuncZ_2)\subset\cD_\delta$. For $R_3$,
	\begin{align*}
		\sup_{y_1, y_2\in\cB[\delta/h]}\abs{\scp{\Delta^{\ell_1}\shift_{y_1}\FuncZ_1}{(\bar\pi_\delta-\pi_\delta)S_0(t)\Delta^{\ell_2}\shift_{y_2}\FuncZ_2}_0} \hspace{-6cm} & \\
			&\lesssim \norm{\FuncZ_1}_{W^{2\ell_1, q}(\R^d)}\sup_{y_2\in\cB[\delta/h]}\norm{(\bar\pi_\delta-\pi_\delta)S_0(t)\Delta^{\ell_2}\shift_{y_2}\FuncZ_2}_{L^p(\bar\cD_\delta)} \\
			&\lesssim \delta^{2\ell_2+d/q}e^{-\bar c\delta^{-2}t^{-1}}\norm{\FuncZ_1}_{W^{2\ell_1, q}(\R^d)}\norm{\FuncZ_2}_{L^1(\R^d)} \\
			&\lesssim \delta^{2\ell_2+d/q}e^{-\bar c\delta^{-2}t^{-1}}\norm{\FuncZ_1}_{W^{2\ell_1, 1}(\R^d)\cap W^{2\ell_1, p}(\R^d)}\norm{\FuncZ_2}_{L^1(\R^d)},
	\end{align*}
	where we have used Lemma \ref{lem:approx:domain} b) (with $T=t$) and interpolation of $L^p$ spaces.
	Similarly, using $\ell_1\geq 2$ for $\alpha>1$, the term resulting from $R_4$ yields the bound
	\begin{align*}
		\sup_{y_1, y_2\in\cB[\delta/h]}\abs{\int_0^t\scp{\Delta^{\ell_1}\shift_{y_1}\FuncZ_1}{S_\delta(t-s)(\pi_\delta-\bar\pi_\delta)S_0(s)\Delta^{\ell_2+1}\shift_{y_2}\FuncZ_2}_0\diff s} \hspace{-11cm} & \\
			&\lesssim \sup_{y_1, y_2\in\cB[\delta/h]}\int_0^t\norm{S_\delta(t-s)\Delta^{\ell_1}\shift_{y_1}\FuncZ_1}_{L^{p}(\cD_\delta)} \\
			&\hspace{5cm} \times\norm{(\bar\pi_\delta-\pi_\delta)S_0(s)\Delta^{\ell_2+1}\shift_{y_2}\FuncZ_2}_{L^q(\bar\cD_\delta)}\diff s \\
			&\lesssim \delta^{2(\ell_2+1)+d/p}e^{-\bar c\delta^{-2}t^{-1}}\norm{\FuncZ_2}_{L^1(\R^d)} \\
			&\hspace{3cm} \times\sup_{y_1\in\cB[\delta/h]}\int_0^t\norm{S_\delta(t-s)\Delta^{\ell_1}\shift_{y_1}\FuncZ_1}_{L^{p}(\cD_\delta)}\diff s \\
			&\lesssim \delta^{2(\ell_2+1)+d/p}e^{-\bar c\delta^{-2}t^{-1}}\norm{\FuncZ_2}_{L^1(\R^d)}\sup_{y_1\in\cB[\delta/h]}\norm{\shift_{y_1}\FuncZ_1}_{\bar W^{2\ell_1, 1}_\alpha\cap \bar W^{2\ell_1, p}_\alpha} \\
			&\hspace{3cm} \times\int_0^t1\wedge(t-s)^{-\alpha/2-d/(2q)}\diff s \\
			&\lesssim \delta^{2(\ell_2+1)+d/p}e^{-\bar c\delta^{-2}t^{-1}}\norm{\FuncZ_2}_{L^1(\R^d)}(h\delta^{-1})^\alpha\norm{\FuncZ_1}_{\bar W^{2\ell_1, 1}_\alpha\cap \bar W^{2\ell_1, p}_\alpha},
	\end{align*}
	where the integral is bounded due to $d>(2-\alpha)q$.
	Consider the last remaining term $R_5$:
	Noting that $\ell_2+d/(2p)-1/2>0$ as $\ell_2\geq 1$,
	and using Lemma \ref{lem:bias:Laplacian},
	together with the semigroup bounds from Lemma \ref{lem:bound:semigroup:heatkernel} and Lemma \ref{lem:bound:semigroup:uniform}, and the shift operator bounds from Lemma \ref{lem:bound:space:Sobolev},
	\begin{align*}
		\sup_{y_1, y_2\in\cB[\delta/h]}\abs{\int_0^t\scp{\Delta^{\ell_1}\shift_{y_1}\FuncZ_1}{S_\delta(t-s)(\Delta_{\vartheta_\delta}\pi_\delta - \pi_\delta\Delta_{\vartheta_0})S_0(s)\Delta^{\ell_2}\shift_{y_2}\FuncZ_2}_0\diff s} \hspace{-12cm} & \\
			&\lesssim \sup_{y_1, y_2\in\cB[\delta/h]}\int_0^t\norm{S_\delta(t-s)\Delta^{\ell_1}\shift_{y_1}\FuncZ_1}_{L^{p}(\cD_\delta)} \\
			&\hspace{5cm} \times\norm{(\Delta_{\vartheta_\delta}\pi_\delta - \pi_\delta\Delta_{\vartheta_0})S_0(s)\Delta^{\ell_2}\shift_{y_2}\FuncZ_2}_{L^q(\cD_\delta)}\diff s \\
			&\lesssim \delta\sup_{y_1, y_2\in\cB[\delta/h]}\int_0^t\norm{S_\delta(t-s)\Delta^{\ell_1}\shift_{y_1}\FuncZ_1}_{L^{p}(\cD_\delta)}\norm{S_0(s)\Delta^{\ell_2}\shift_{y_2}\FuncZ_2}_{\bar W^{2, q}_{1}}\diff s \\
			&\lesssim \delta\sup_{y_1, y_2\in\cB[\delta/h]}\norm{\shift_{y_1}\FuncZ_1}_{\bar W^{2\ell_1, 1}_\alpha\cap \bar W^{2\ell_1, p}_\alpha}\norm{\shift_{y_2}\FuncZ_2}_{\bar W^{2, 1}_{1}\cap\bar W^{2+2\ell_2, q}_{1}} \\
			&\quad\quad\quad\quad \times \int_0^t(1\wedge(t-s)^{-\alpha/2-d/(2q)})(1\wedge s^{-\ell_2-d/(2p)+1/2})\diff s \\
			&\lesssim \delta((t\wedge t^{1-\alpha/2-d/(2q)}) + (t\wedge t^{3/2-\ell_2-d/(2p)}))\\
			&\quad\quad\quad\quad \times \sup_{y_1, y_2\in\cB[\delta/h]}\norm{\shift_{y_1}\FuncZ_1}_{\bar W^{2\ell_1, 1}_\alpha\cap \bar W^{2\ell_1, p}_\alpha}\norm{\shift_{y_2}\FuncZ_2}_{\bar W^{2, 1}_{1}\cap\bar W^{2+2\ell_2, q}_{1}} \\
			&\lesssim \delta((t\wedge t^{1-\alpha/2-d/(2q)}) + (t\wedge t^{3/2-\ell_2-d/(2p)}))(h\delta^{-1})^{\alpha+1} \\
			&\hspace{6.5cm} \times\norm{\FuncZ_1}_{\bar W^{2\ell_1, 1}_\alpha\cap \bar W^{2\ell_1, p}_\alpha}\norm{\FuncZ_2}_{\bar W^{2, 1}_{1}\cap\bar W^{2+2\ell_2, q}_{1}}.
	\end{align*}

		\item[b)] Again, bound the different terms from Lemma \ref{lem:approx:semigroup}, this time directly in $L^p$-norm.
		$R_1$, $R_2$ and $R_3$ together give (by Lemma \ref{lem:approx:domain}) a bound of the order $$\delta^{2\ell+d/q-(d-1)/p}e^{-\bar c\delta^{-2}T^{-1}}\normSmall{\FuncZ}_{L^1(\R^d)}.$$
		Next, $R_4$ yields a bound of the order $$T\delta^{2(\ell+1)+d/q-(d-1)/p}e^{-\bar c\delta^{-2}T^{-1}}\normSmall{\FuncZ}_{L^1(\R^d)}.$$
		$R_5$ is treated as in a),
		resulting in a bound of the order $$hC_T\normSmall{\FuncZ}_{\bar W^{2, 1}_1\cap \bar W^{2+2\ell, p}_1}.$$
		In case $\vartheta$ is constant, $\Delta_{\vartheta_\delta}\pi_\delta - \pi_\delta\Delta_{\vartheta_0}=0$, so in this case even $R_5=0$,
		and the approximation error from the last panel does not appear.
		
	\end{enumerate}
	
\end{proof}

\section{Semigroup bounds}

In this section, we gather various bounds for the heat semigroups on bounded and unbounded domains. The exposition in this section is based on \cite[Appendix A.2]{AltmeyerReiss2021}.
We start with recalling some well-known contraction properties of the heat kernel, see e.g. \cite[Chapter 9.2]{Friedman2008} for similar results for more general parabolic equations.

\begin{lemma}[heat kernel bounds] \label{lem:bound:properties:heatkernel}
	Let $\ell\in\N_0$ and $\alpha\geq 0$.
	\begin{enumerate}
		
		\item[a)] For $t>0$, $x\in\R^d$,
		\begin{align}
			\abs{\abs{x}^\alpha\Delta^\ell q_t(x)}\lesssim t^{\frac{\alpha}{2}-\ell}q_{2t}(x).
		\end{align}
		
		\item[b)] For $p,q\geq 1$ with $1=1/p+1/q$,
		\begin{align}
			\norm{\abs{x}^\alpha\Delta^\ell q_t}_{L^p(\R^d)} \lesssim t^{\frac{\alpha}{2}-\ell-\frac{d}{2q}}.
		\end{align}
		
	\end{enumerate}
\end{lemma}

\begin{proof} \
	\begin{enumerate}
		
		\item[a)] An inductive argument shows that for $\ell\in\N_0$,
		\begin{align}
			\Delta^\ell q(x) = q(x)p_\ell(\abs{x}^2)
		\end{align}
		with some polynomial $p_\ell$ of degree not larger than $\ell$, so
		\begin{align}
			\abs{\abs{x}^\alpha\Delta^\ell q(x)} &= (2\pi)^{-d/2}e^{-\abs{x}^2/4}(e^{-\abs{x}^2/4}\abs{x}^\alpha\absSmall{p_\ell(\abs{x}^2)})\lesssim q(x/\sqrt{2}),
		\end{align}
		and
		\begin{align*}
			(2t)^{\ell-\frac{\alpha}{2}}\abs{\abs{x}^\alpha\Delta^\ell q_t(x)} &= (2t)^{-d/2}\abs{\abs{\cdot}^\alpha\Delta^\ell q}((2t)^{-1/2}x) \\
				&\leq (2t)^{-d/2}q((2t)^{-1/2}x/\sqrt{2}) = 2^{d/2}q_{2t}(x).
		\end{align*}
		
		\item[b)] This follows from a) together with
		\begin{align*}
			\norm{q_t}_p &\lesssim t^{-\frac{d}{2}+\frac{d}{2p}}\left(\int_{\R^d}\frac{1}{(2t)^{d/2}}e^{-\frac{p\abs{x}^2}{4t}}\diff x\right)^\frac{1}{p} \lesssim t^{-\frac{d}{2q}}.
		\end{align*}
		Note that $\norm{q_t}_p\lesssim t^{-d/(2q)}$ is valid in the border cases $p,q\in\{1,\infty\}$, too.
	\end{enumerate}
\end{proof}

\begin{lemma}[semigroup bounds on $\R^d$] \label{lem:bound:semigroup:heatkernel}
	Let $p,q\geq 1$ with $1=1/p+1/q$, and $k,\ell\in\N_0$.
	\begin{enumerate}
		
		\item[a)]
		For $\FuncZ\in W^{k, 1}(\R^d)\cap W^{k+2\ell, p}(\R^d)$ and $t>0$,
		\begin{align}
			\norm{e^{t\Delta}\Delta^\ell \FuncZ}_{W^{k, p}(\R^d)}\leq C(1\wedge t^{-\ell-d/(2q)})\norm{\FuncZ}_{W^{k,1}(\R^d)\cap W^{k+2\ell,p}(\R^d)}.
		\end{align}
		
		\item[b)] Let $\alpha\geq 0$. For $\FuncZ\in \bar W_\alpha^{0,1}\cap\bar W_\alpha^{2\ell,p}$:
		\begin{align}
			\norm{\abs{x}^\alpha e^{t\Delta}\Delta^\ell \FuncZ}_{L^p(\R^d)} \lesssim (1\wedge t^{-\ell-d/(2q)})(1\vee t^{\alpha/2})\norm{\FuncZ}_{\bar W_\alpha^{0,1}\cap\bar W_\alpha^{2\ell,p}},
		\end{align}
		and in particular, for $\FuncZ\in \bar W_\alpha^{k,1}\cap\bar W_\alpha^{k+2\ell,p}$:
		\begin{align}
			\norm{e^{t\Delta}\Delta^\ell \FuncZ}_{\bar W^{k, p}_\alpha} \lesssim (1\wedge t^{-\ell-d/(2q)})(1\vee t^{\alpha/2})\norm{\FuncZ}_{\bar W_\alpha^{k,1}\cap\bar W_\alpha^{k+2\ell,p}}.
		\end{align}
		
	\end{enumerate}
\end{lemma}

\begin{proof} \
	\begin{enumerate}
		
		\item[a)] Let first $k=0$. With Young's inequality and Lemma \ref{lem:bound:properties:heatkernel} (with $\alpha=0$),
			\begin{align*}
				\norm{e^{t\Delta}\Delta^\ell \FuncZ}_{p}
					&= \norm{\Delta^\ell(q_t*\FuncZ)}_p \leq \norm{q_t}_1\norm{\Delta^\ell \FuncZ}_p\wedge\norm{\Delta^\ell q_t}_p\norm{\FuncZ}_1 \\
					&\lesssim (1\wedge t^{-\ell-d/(2q)})\norm{\FuncZ}_{L^1\cap W^{2\ell,p}}.
			\end{align*}
			Next, for arbitrary $k\in\N_0$, the $W^{k,p}$-norm is a sum of $L^p$-norms of weak partial derivatives of the argument, which commute with $e^{t\Delta}\Delta^\ell$. It suffices to apply the previous panel to $\partial_\alpha \FuncZ$, $\abs{\alpha}\leq k$.
		
		\item[b)]
		First note that
		\begin{align}
			\norm{\abs{x}^\alpha(\FuncX*\FuncY)}_p\lesssim \norm{(\abs{x}^\alpha \abs{\FuncX})*\abs{\FuncY}}_p + \norm{\abs{\FuncX}*(\abs{x}^\alpha \abs{\FuncY})}_p
		\end{align}
		for all $\FuncX,\FuncY:\R^d\rightarrow\R$ such that these norms are well-defined.
		Indeed,
		\begin{align*}
			\norm{\abs{x}^\alpha(\FuncX*\FuncY)}_p
				&= \left(\int_{\R^d}\abs{\int_{\R^d}\abs{x}^\alpha \FuncX(x-y)\FuncY(y)\diff y}^p\diff x\right)^\frac{1}{p} \\
				&\lesssim \left(\int_{\R^d}\abs{\int_{\R^d}\abs{x-y}^\alpha \abs{\FuncX(x-y)}\abs{\FuncY(y)}\diff y}^p\diff x\right)^\frac{1}{p} \\
				&\quad\quad\quad\quad + \left(\int_{\R^d}\abs{\int_{\R^d}\abs{\FuncX(x-y)}\abs{y}^\alpha\abs{\FuncY(y)}\diff y}^p\diff x\right)^\frac{1}{p} \\
				&= \norm{(\abs{x}^\alpha \abs{\FuncX})*\abs{\FuncY}}_p + \norm{\abs{\FuncX}*(\abs{x}^\alpha \abs{\FuncY})}_p
		\end{align*}
		and an analogous argument for $p=\infty$.
		With $\FuncX=\Delta^\ell q_{t}$ and $\FuncY=\FuncZ$,
		\begin{align*}
			\norm{\abs{x}^\alpha e^{t\Delta}\Delta^\ell \FuncZ}_p
				&= \norm{\abs{x}^\alpha(\FuncX*\FuncY)}_p \lesssim \norm{(\abs{x}^\alpha \abs{\FuncX})*\abs{\FuncY}}_p + \norm{\abs{\FuncX}*(\abs{x}^\alpha \abs{\FuncY})}_p \\
				&\lesssim \norm{\abs{x}^\alpha\Delta^\ell q_t}_p\norm{\FuncZ}_1 + \norm{\Delta^\ell q_t}_p\norm{\abs{x}^\alpha \FuncZ}_1 \\
				&\lesssim t^{-\ell-\frac{d}{2q}}(1\vee t^\frac{\alpha}{2})\norm{\FuncZ}_{\bar W^{0, 1}_\alpha}
		\end{align*}
		again by Lemma \ref{lem:bound:properties:heatkernel}, and with $\FuncX=q_t$ and $\FuncY=\Delta^\ell \FuncZ$,
		\begin{align*}
			\norm{\abs{x}^\alpha e^{t\Delta}\Delta^\ell \FuncZ}_p
				&\lesssim \norm{\abs{x}^\alpha q_t}_1\norm{\Delta^\ell \FuncZ}_p + \norm{q_t}_1\norm{\abs{x}^\alpha \Delta^\ell \FuncZ}_p \\
				&\lesssim (1\vee t^\frac{\alpha}{2})\norm{\FuncZ}_{\bar W^{2\ell, p}_\alpha}.
		\end{align*}
		These estimates imply the first claim. The second claim is a direct consequence.
	\end{enumerate}
\end{proof}

\begin{lemma}[semigroup bounds on bounded domains] \label{lem:bound:semigroup:uniform} \
	Let $\vartheta\in C^\beta(\bar\cD)$ for $\beta\leq 2$, and $p\geq 2$, $q\leq 2$ with $1/p+1/q=1$.
	Then we have for $T>0$, $\delta>0$, $0\leq t\leq T\delta^{-2}$:
	\begin{enumerate}
		\item[a)] If $\FuncZ\in \bar W^{2\ell, 1}_\alpha\cap\bar W^{2\ell, p}_\alpha$ with support in $\cD_\delta$,
		under either of the combinations
		\addtocounter{equation}{1}
		\begin{align}
			&\ell\geq 0 \quad\mathrm{and}\quad \alpha=0, \tag{\theequation a} \label{eq:bound:semigroup:uniform:a} \\
			&\ell\geq 1 \quad\mathrm{and}\quad 0\leq\alpha\leq 1, \tag{\theequation b} \label{eq:bound:semigroup:uniform:b} \\
			&\ell\geq 2 \quad\mathrm{and}\quad 0\leq\alpha\leq\beta, \tag{\theequation c} \label{eq:bound:semigroup:uniform:c}
		\end{align}
		it follows that
		\begin{align*}
			\norm{S_\delta(t)\Delta^\ell \FuncZ}_{L^p(\cD_\delta)} \lesssim (1\wedge t^{-\alpha/2-d/(2q)})\normSmall{\FuncZ}_{\bar W^{2\ell, 1}_\alpha\cap\bar W^{2\ell, p}_\alpha}.
		\end{align*}
		
		\item[b)] If $\FuncZ\in \bar W^{2\ell, 1}_\alpha\cap\bar W^{2\ell, p}_\alpha$ with support in $\cD_\delta$, under either of the combinations
		\addtocounter{equation}{1}
		\begin{align}
			&\ell\geq 1 \quad\mathrm{and}\quad 0\leq\alpha\leq 1, \tag{\theequation a} \label{eq:bound:semigroup:uniform:diff:a} \\
			&\ell\geq 2 \quad\mathrm{and}\quad 0\leq\alpha\leq\beta, \tag{\theequation b} \label{eq:bound:semigroup:uniform:diff:b}
		\end{align}
		it holds that
		\begin{align*}
			\norm{S_\delta(t)[\Delta_{\vartheta_\delta}-\Delta_{\vartheta_0}]\Delta^{\ell-1}\FuncZ}_{L^p(\cD_\delta)} &\lesssim (1\wedge t^{-\alpha/2-d/(2q)})\normSmall{\FuncZ}_{\bar W^{2\ell, 1}_\alpha\cap\bar W^{2\ell, p}_\alpha}.
		\end{align*}
		
		\item[c)] If $\FuncZ\in \bar W^{2, 1}_\alpha\cap\bar W^{2, p}_\alpha$ with support in $\cD_\delta$, where $0\leq\alpha\leq 1$,
		\begin{align*}
			\norm{S_\delta(t)[\Delta_{\vartheta_\delta}-\Delta_{\vartheta_0}]\FuncZ}_{L^p(\cD_\delta)} &\lesssim \delta^\alpha(1\wedge t^{-d/(2q)})\norm{\FuncZ}_{\bar W^{2, 1}_\alpha\cap\bar W^{2, p}_\alpha}.
		\end{align*}
		
		\item[d)] All bounds remain true if $S_\delta$ is replaced by $S_0$.
		
	\end{enumerate}
\end{lemma}

\begin{proof} \
	\begin{itemize}
		\item[a)] \emph{Case \eqref{eq:bound:semigroup:uniform:a}.}
		Lemma \ref{lem:AR21:35} implies
		\begin{align}
			\norm{S_{\delta}(t)\FuncY}_{L^p(\cD_\delta)}\lesssim e^{CT}\norm{S_0(ct)\abs{\FuncY}}_{L^p(\cD_\delta)}
		\end{align}
		for $\FuncY\in C(\bar \cD_\delta)$, thus for $\FuncY\in L^p(\cD_\delta)$ by a density argument. Now set $\FuncY=\Delta^\ell \FuncZ$ and apply Lemma \ref{lem:bound:semigroup:heatkernel} a) with $k=\ell=0$ to obtain
		\begin{align*}
			\norm{S_\delta(t)\Delta^\ell \FuncZ}_{L^p(\cD_\delta)}
				&\lesssim (1\wedge t^{-d/(2q)})\normSmall{\Delta^\ell \FuncZ}_{L^1(\R^d)\cap L^p(\R^d)} \\
				&\lesssim (1\wedge t^{-d/(2q)})\normSmall{\FuncZ}_{W^{2\ell, 1}(\R^d)\cap W^{2\ell, p}(\R^d)}.
		\end{align*}
		
		\item[c)] Apply a) with $\ell=0$ and $\alpha=0$ to $[\Delta_{\vartheta_\delta}-\Delta_{\vartheta_0}]\FuncZ$ and use Lemma \ref{lem:bias:Laplacian} in $L^1$ and $L^p$.
		
		\item[b)] \emph{Case \eqref{eq:bound:semigroup:uniform:diff:a}.}
		The case $\ell=1$ follows from c) via $\delta\leq T^{1/2}t^{-1/2}$, and for $\ell>1$, replace $\FuncZ$ by $\Delta^{\ell-1}\FuncZ$ and use $\normSmall{\Delta^{\ell-1}\FuncZ}_{\bar W^{2,1}_\alpha\cap\bar W^{2, p}_\alpha}\lesssim \norm{\FuncZ}_{\bar W^{2\ell,1}_\alpha\cap\bar W^{2\ell, p}_\alpha}$.
		
		\item[a)] \emph{Case \eqref{eq:bound:semigroup:uniform:b}.}
		Abbreviate $\FuncY:=\Delta^{\ell-1}\FuncZ$.
		First note that from Lemma \ref{lem:bias:Laplacian} with $\alpha=0$,
		\begin{align*}
			\norm{\Delta_{\vartheta_\delta}\FuncY}_{L^p(\cD_\delta)}
				&\leq \norm{[\Delta_{\vartheta_\delta}-\Delta_{\vartheta_0}]\FuncY}_{L^p(\cD_\delta)} + \norm{\Delta_{\vartheta_0}\FuncY}_{L^p(\cD_\delta)} \\
				&\lesssim \norm{\FuncY}_{\bar W^{2, p}_0} + \norm{\FuncY}_{W^{2, p}(\R^d)} \\
				&\lesssim \norm{\FuncY}_{W^{2, p}(\R^d)}.
		\end{align*}
		Next, by Lemma \ref{lem:ACP:17} and a) with $\ell=0$, $\alpha=0$,
		\begin{align*}
			\norm{S_\delta(t)\Delta_{\vartheta_\delta}\FuncY}_{L^p(\cD_\delta)}
				&\lesssim t^{-1}\norm{S_\delta(t/2)\FuncY}_{L^p(\cD_\delta)} \wedge \norm{\Delta_{\vartheta_\delta}\FuncY}_{L^p(\cD_\delta)} \\
				&\lesssim t^{-1-d/(2q)}\norm{\FuncY}_{L^1(\R^d)\cap L^p(\R^d)}\wedge \norm{\FuncY}_{W^{2, p}(\R^d)} \\
				&\lesssim (1\wedge t^{-1-d/(2q)})\norm{\FuncY}_{L^1(\R^d)\cap W^{2, p}(\R^d)}.
		\end{align*}
		Write $\Delta_{\vartheta_0}\FuncY = [\Delta_{\vartheta_0}-\Delta_{\vartheta_\delta}]\FuncY+\Delta_{\vartheta_{\delta}}\FuncY$, then with b):
		\begin{align*}
			\norm{S_\delta(t)\Delta^\ell \FuncZ}_{L^p(\cD_\delta)}
				&\lesssim \norm{S_\delta(t)[\Delta_{\vartheta_0}-\Delta_{\vartheta_\delta}]\FuncY}_{L^p(\cD_\delta)} + \norm{S_\delta(t)\Delta_{\vartheta_\delta}\FuncY}_{L^p(\cD_\delta)} \\
				&\hspace{-1.2cm}\lesssim (1\wedge t^{-\alpha/2-d/(2q)})\norm{\FuncZ}_{\bar W^{2\ell, 1}_\alpha\cap\bar W^{2\ell, p}_\alpha} + (1\wedge t^{-1-d/(2q)})\norm{\FuncY}_{L^1(\R^d)\cap W^{2, p}(\R^d)} \\
				&\hspace{-1.2cm}\lesssim (1\wedge t^{-\alpha/2-d/(2q)})\norm{\FuncZ}_{\bar W^{2\ell,1}_\alpha\cap\bar W^{2\ell, p}_\alpha}.
		\end{align*}
		
		\item[b)] \emph{Case \eqref{eq:bound:semigroup:uniform:diff:b}.}
		It suffices to consider the case $1<\alpha\leq\beta$. Again, abbreviate $\FuncY=\Delta^{\ell-1}\FuncZ$.
		Write $$[\Delta_{\vartheta_\delta}-\Delta_{\vartheta_0}]\FuncY=[\Delta_{\vartheta_\delta}-\Delta_{\vartheta_0}-\delta\Delta'_{(\nabla\vartheta)_0}]\FuncY + \delta\Delta'_{(\nabla\vartheta)_0}\FuncY$$ and apply a) and Lemma \ref{lem:bias:Laplacian} to the first term, using $\delta\leq T^{1/2}t^{-1/2}$:
		\begin{align*}
			\norm{S_\delta(t)[\Delta_{\vartheta_\delta}-\Delta_{\vartheta_0}-\delta\Delta'_{(\nabla\vartheta)_0}]\FuncY}_{L^p(\cD_\delta)} \hspace{-4cm} & \\
				&\lesssim (1\wedge t^{-d/(2q)})\norm{[\Delta_{\vartheta_\delta}-\Delta_{\vartheta_0}-\delta\Delta'_{(\nabla\vartheta)_0}]\FuncY}_{L^1(\R^d)\cap L^p(\R^d)} \\
				&\lesssim \delta^\alpha (1\wedge t^{-d/(2q)}) \norm{\FuncY}_{\bar W^{2, 1}_\alpha\cap\bar W^{2, p}_\alpha} \\
				&\lesssim (1\wedge t^{-\alpha/2-d/(2q)}) \norm{\FuncZ}_{\bar W^{2\ell, 1}_\alpha\cap\bar W^{2\ell, p}_\alpha},
		\end{align*}
		and with $\Delta'_{(\nabla\vartheta)_0}\FuncY=\Delta \FuncX$ from Lemma \ref{lem:LaplaceDiff:properties} b) with $\ell=1$,
		\begin{align*}
			\norm{S_\delta(t)\delta\Delta'_{(\nabla\vartheta)_0}\FuncY}_{L^p(\cD_\delta)}
				&= \delta\norm{S_\delta(t)\Delta \FuncX}_{L^p(\cD_\delta)} \\
				&\lesssim \delta(1\wedge t^{-(\alpha-1)/2-d/(2q)})\norm{\FuncX}_{\bar W^{2, 1}_{\alpha-1}\cap\bar W^{2, p}_{\alpha-1}} \\
				&\lesssim (1\wedge t^{-\alpha/2-d/(2q)})\normSmall{\Delta^{\ell-2}\FuncZ}_{\bar W^{4, 1}_\alpha\cap\bar W^{4, p}_\alpha} \\
				&\lesssim (1\wedge t^{-\alpha/2-d/(2q)})\normSmall{\FuncZ}_{\bar W^{2\ell, 1}_\alpha\cap\bar W^{2\ell, p}_\alpha}.
		\end{align*}
		
		\item[a)] \emph{Case \eqref{eq:bound:semigroup:uniform:c}.}
		The proof works verbatim as in the case
		\eqref{eq:bound:semigroup:uniform:b}.
		
		\item[d)] For $S_0$, the statement a) is just a restriction of Lemma \ref{lem:bound:semigroup:heatkernel} a). This can be substituted in the proof of b) and c) wherever needed.
		
	\end{itemize}
\end{proof}

If $\vartheta$ is constant, this can be improved to the rates expected from Lemma \ref{lem:bound:semigroup:heatkernel}.

\begin{lemma}\label{lem:bound:semigroup:parametric}
	Let $\vartheta$ be constant and fix $\ell\in\N_0$ and $p\geq 2$, $q\leq 2$ with $1/p+1/q=1$. Then we have for $\delta>0$, $t\geq 0$ and $\varphi\in L^1(\R^d)\cap W^{2\ell, p}(\R^d)$ with support in $\cD_\delta$
	\begin{align*}
		\norm{S_\delta(t)\Delta^\ell\varphi}_{L^p(\cD_\delta)}\lesssim(1\wedge t^{-\ell-d/(2q)})\norm{\varphi}_{L^1(\R^d)\cap W^{2\ell, p}(\R^d)}.
	\end{align*}
\end{lemma}

\begin{proof}
	The case $\ell=0$ works verbatim as in Lemma \ref{lem:bound:semigroup:uniform} a), where the condition $0\leq t\leq T\delta^{-2}$ is not used. For $\ell\geq 1$, noting that $\vartheta_\delta$ is constant and using Lemma \ref{lem:ACP:17},
	\begin{align*}
		\normSmall{S_\delta(t)\Delta^\ell\varphi}_{L^p(\cD_\delta)}
			&\lesssim \normSmall{S_\delta(t)\Delta_{\vartheta_\delta}^\ell\varphi}_{L^p(\cD_\delta)}
			\lesssim t^{-\ell}\norm{S_\delta(t/2)\varphi}_{L^p(\cD_\delta)} \wedge \normSmall{\Delta_{\vartheta_\delta}^\ell\varphi}_{L^p(\cD_\delta)} \\
			&\lesssim t^{-\ell-d/(2q)}\norm{\varphi}_{L^1(\R^d)\cap L^p(\R^d)}\wedge \norm{\varphi}_{W^{2\ell, p}(\R^d)},
	\end{align*}
	where we applied the case $\ell=0$ in the last estimate. This finishes the proof.
\end{proof}

\section{Auxiliary Results}

\begin{lemma}[Taylor expansion of $\delta\mapsto\Delta_{\vartheta_\delta}$] \label{lem:bias:Laplacian}
	Let $p\geq 1$ and $\vartheta\in C^\beta(\bar\cD)$ with $\beta\leq 2$.
	\begin{enumerate}
		\item[a)]
		We have $\pi_\delta\Delta_{\vartheta_0}=\Delta_{\vartheta_0}\pi_\delta$ and $\pi_\delta\Delta'_{(\nabla\vartheta)_0}=\Delta'_{(\nabla\vartheta)_0}\pi_\delta$, considered as mappings $W^{2, p}(\R^d)\rightarrow L^p(\cD_\delta)$ and $\bar W^{2, p}_1\rightarrow L^p(\cD_\delta)$, respectively.
		
		\item[b)]
		For $\FuncZ\in\bar W^{2, p}_\alpha$, where $0\leq\alpha\leq 1$,
		\begin{align}
			\norm{[\Delta_{\vartheta_\delta}-\Delta_{\vartheta_0}]\pi_\delta \FuncZ}_{L^p(\cD_\delta)} \lesssim \delta^\alpha \norm{\FuncZ}_{\bar W^{2, p}_\alpha},
		\end{align}
		and if $\FuncZ\in \bar W^{2, p}_\alpha$, where $1<\alpha\leq\beta$,
		\begin{align}
			\norm{[\Delta_{\vartheta_\delta}-\Delta_{\vartheta_0} - \delta \Delta'_{(\nabla\vartheta)_0}]\pi_\delta \FuncZ}_{L^p(\cD_\delta)} \lesssim \delta^\alpha\norm{\FuncZ}_{\bar W^{2, p}_\alpha}.
		\end{align}
	\end{enumerate}
\end{lemma}

\begin{proof} \
	\begin{enumerate}
		\item[a)] This is trivial.
		
		\item[b)] Using $\Delta_{\vartheta_\delta}=\vartheta_\delta\Delta + \nabla\vartheta_\delta\cdot\nabla$, we see that for $x\in \cD_\delta$:
		\begin{align*}
			[\Delta_{\vartheta_\delta}-\Delta_{\vartheta_0}]\pi_\delta \FuncZ(x)
				&= (\vartheta(\delta x + x_0)-\vartheta(x_0))\Delta \FuncZ(x) \\
				&\quad\quad\quad\quad + \delta\nabla\vartheta(\delta x + x_0)\cdot\nabla \FuncZ(x),
		\end{align*}
		thus
		\begin{align*}
			\abs{[\Delta_{\vartheta_\delta}-\Delta_{\vartheta_0}]\pi_\delta \FuncZ(x)}
				&\lesssim \delta^\alpha\abs{x}^\alpha\abs{\Delta \FuncZ(x)} + \delta \sup_{z\in\cD}\abs{\nabla\vartheta(z)}\abs{\nabla \FuncZ(x)},
		\end{align*}
		which yields the first claim.
		Next,
		\begin{align*}
			[\Delta_{\vartheta_\delta}-\Delta_{\vartheta_0} - \delta \Delta'_{(\nabla\vartheta)_0}]\pi_\delta \FuncZ(x)
				&= (\vartheta(\delta x + x_0)-\vartheta(x_0)-\delta\nabla\vartheta(x_0)\cdot x)\Delta \FuncZ(x) \\
				&\quad\quad\quad\quad + \delta(\nabla\vartheta(\delta x+x_0)-\nabla \vartheta(x_0))\cdot\nabla \FuncZ(x).
		\end{align*}
		For some $\xi^*\in[x_0, x_0+\delta x]$, we have $\vartheta(\delta x + x_0)-\vartheta(x_0) = \delta \nabla\vartheta(\xi^*)\cdot x$, so using $(\beta-1)$-H\"older continuity (and a forteriori $(\alpha-1)$-H\"older continuity) on $\bar\cD$ of $\nabla\vartheta$, we obtain
		\begin{align*}
			\abs{[\Delta_{\vartheta_\delta}-\Delta_{\vartheta_0} - \delta \Delta'_{(\nabla\vartheta)_0}]\pi_\delta \FuncZ(x)}
				&\leq \delta\abs{\nabla\vartheta(\xi^*)-\nabla\vartheta(x_0)}\abs{x}\abs{\Delta \FuncZ(x)} \\
				&\quad\quad + \delta\abs{\nabla\vartheta(\delta x+x_0)-\nabla\vartheta(x_0)}\abs{\nabla \FuncZ(x)} \\
				&\leq C\delta^\alpha(\abs{x}^\alpha\abs{\Delta \FuncZ(x)}+\abs{x}^{\alpha-1}\abs{\nabla \FuncZ(x)}),
		\end{align*}
		implying the second claim.
	\end{enumerate}
\end{proof}

\begin{lemma}[properties of $\Delta'_{(\nabla\vartheta)_0}$] \label{lem:LaplaceDiff:properties} \
	\begin{enumerate}
		\item[a)] For $p\geq 2$ and $\FuncZ\in W^{2, p}(\R^d)$ with compact support and $\int_{\R^d}\FuncZ(x)\diff x=0$, there is $\FuncY\in W^{2, p}(\R^d)$ with compact support such that $\Delta'_{(\nabla\vartheta)_0}\FuncZ=\Delta \FuncY$.
		
		\item[b)]
		Let $\ell\in\N_0$, $p\geq 1$, and $\FuncZ\in \bar W^{2\ell+2, p}_1$. Then there is $\FuncY^{(\ell)}\in W^{2\ell, p}(\R^d)$ such that
		\begin{align}
			\Delta'_{(\nabla\vartheta)_0}(\Delta^\ell \FuncZ) = \Delta^\ell \FuncY^{(\ell)}.
		\end{align}
		For $k\in\N_0$, $p\geq 1$ and $\alpha\geq 0$, $\normSmall{\FuncY^{(\ell)}}_{\bar W^{k, p}_\alpha}\lesssim\norm{\FuncZ}_{\bar W^{k+2, p}_{\alpha+1}}$.
		If $\FuncZ$ has compact support, then $\FuncY^{(\ell)}$ has compact support. If $\FuncZ\in C^\infty(\R^d)$, then $\FuncY^{(\ell)}\in C^\infty(\R^d)$.
		
		\item[c)] In the setting of b), for $y\in\R^d$, it holds
		\begin{align}
			\Delta'_{(\nabla\vartheta)_0}\shift_y(\Delta^\ell \FuncZ)
				&= \shift_y\Delta^\ell \FuncY^{(\ell)} + (\nabla\vartheta(x_0)\cdot y)\shift_y\Delta^{\ell+1} \FuncZ.
		\end{align}
	\end{enumerate}
\end{lemma}

\begin{proof} \
	\begin{enumerate}
		\item[a)] The case $p=2$ is covered by Lemma \ref{lem:AR21:A5}. For $p>2$ it holds that $\FuncZ\in W^{2,2}(\R^d)$ due to its compact support (contained in some smoothly bounded domain $\cK\subset\R^d$, say), so there is $\FuncY\in W^{2, 2}(\R^d)$ with $\Delta \FuncY=\Delta'_{(\nabla\vartheta)_0}\FuncZ$. In particular, $-\FuncY$ solves the Poisson problem with inhomogeneity $\Delta'_{(\nabla\vartheta)_0}\FuncZ$ and Dirichlet boundary conditions on $\cK$, so $\norm{\FuncY}_{W^{2, p}(\R^d)}=\norm{\FuncY}_{W^{2, p}(\cK)}\lesssim \normSmall{\Delta'_{(\nabla\vartheta)_0}\FuncZ}_{L^p(\cK)}\lesssim \norm{\FuncZ}_{W^{2, p}(\cK)}<\infty$ (cf. \cite[Theorem 9.15]{GilbargTrudinger2001}), and therefore $\FuncY\in W^{2, p}(\R^d)$.
		
		\item[b)] Inductively, one shows that
		\begin{align*}
			\Delta^\ell[(\nabla\vartheta(x_0)\cdot x)\Delta \FuncZ]
				&= (\nabla\vartheta(x_0)\cdot x)\Delta^{\ell+1} \FuncZ + \sum_{i=1}^da_{i}^\ell\partial_i \Delta^\ell \FuncZ
		\end{align*}
		for certain coefficients $a_i^\ell\in\R$.
		Indeed, the case $\ell=0$ is trivial, and for the step $\ell\mapsto \ell+1$ use
		\begin{align*}
			\Delta[(\nabla\vartheta(x_0)\cdot x) f]
				&= (\nabla\vartheta(x_0)\cdot x) \Delta f + 2\nabla\vartheta(x_0)\cdot\nabla f
		\end{align*}
		with $f=\Delta^{\ell+1}\FuncZ$.
		Now
		\begin{align*}
			\Delta'_{(\nabla\vartheta)_0}\Delta^\ell \FuncZ
				&= (\nabla\vartheta(x_0)\cdot x)\Delta^{\ell+1} \FuncZ + \nabla\vartheta(x_0)\cdot\nabla\Delta^\ell \FuncZ \\
				&= \Delta^\ell[(\nabla\vartheta(x_0)\cdot x)\Delta \FuncZ] - \sum_{i=1}^da_{i}^\ell\partial_i \Delta^\ell \FuncZ + \nabla\vartheta(x_0)\cdot\nabla\Delta^\ell \FuncZ \\
				&= \Delta^\ell\left[(\nabla\vartheta(x_0)\cdot x)\Delta \FuncZ + \nabla\vartheta(x_0)\cdot\nabla \FuncZ - \sum_{i=1}^da_{i}^\ell\partial_i \FuncZ\right] =: \Delta^\ell \FuncY^{(\ell)}.
		\end{align*}
		The properties of $\FuncY^{(\ell)}$ follow immediately.
		
		\item[c)] By b), it is true that
		\begin{align*}
			\Delta'_{(\nabla\vartheta)_0}\shift_y(\Delta^\ell \FuncZ)
				&= \shift_y\Delta'_{(\nabla\vartheta)_0}(\Delta^\ell \FuncZ) + (\nabla\vartheta(x_0)\cdot y)\shift_y\Delta^{\ell+1} \FuncZ \\
				&= \shift_y\Delta^\ell \FuncY^{(\ell)} + (\nabla\vartheta(x_0)\cdot y)\shift_y\Delta^{\ell+1} \FuncZ.
		\end{align*}
	\end{enumerate}
\end{proof}

\begin{lemma}[shift operator bounds] \label{lem:bound:space:Sobolev}
	Let $\delta\lesssim h\lesssim 1$, and $k\in\N_0$, $p\geq 1$, $\alpha\geq 0$.
		For $\FuncZ\in\bar W^{k, p}_\alpha$,
		we have
		\begin{align}
			\sup_{y\in\cB[\delta/h]}\norm{\shift_y\FuncZ}_{\bar W^{k, p}_\alpha}\lesssim h^\alpha\delta^{-\alpha}\norm{\FuncZ}_{\bar W^{k, p}_\alpha}.
		\end{align}
\end{lemma}

\begin{proof}
	We see that
	with multiindex $\abs{\multiindex}\leq k$,
	\begin{align*}
		\left(\int_{\R^d}\abs{x}^{\alpha p}\abs{\partial_{\multiindex} \shift_y\FuncZ}^p\diff x\right)^\frac{1}{p}
			&= \left(\int_{\R^d}\abs{x+y}^{\alpha p}\abs{\partial_{\multiindex} \FuncZ}^p\diff x\right)^\frac{1}{p} \\
			&\lesssim \norm{\abs{x}^\alpha\partial_{\multiindex} \FuncZ}_{L^p} + \abs{y}^\alpha\norm{\partial_{\multiindex} \FuncZ}_{L^p},
	\end{align*}
	where $\abs{y}\leq Ch\delta^{-1}$ for some $C>0$, and $h\delta^{-1}\gtrsim 1$.
\end{proof}

\section{Further Results from the Literature}

For reference, we state some results from \cite{AltmeyerReiss2021} and \cite{AltmeyerCialencoPasemann2023}.

\begin{lemma}[see Proposition 3.5 in \cite{AltmeyerReiss2021}]\label{lem:AR21:35}
	There are $c_1,c_2,c_3>0$ such that for $\delta>0$, $\FuncZ\in C(\bar\cD_\delta)$, $x\in\cD_\delta$ and $t>0$
	we have
	\begin{align*}
		\abs{[S_\delta(t)\FuncZ](x)} \leq c_1e^{c_2\delta^2t}[S_0(c_3t)\abs{\FuncZ}](x).
	\end{align*}
\end{lemma}

\begin{lemma}[see Lemma A.5 in \cite{AltmeyerReiss2021}]\label{lem:AR21:A5}
	Let $\FuncZ\in W^{2, 2}(\R^d)$ have compact support.
		If $\int_{\R^d}\FuncZ(x)\diff x=0$, then there is $\FuncY\in W^{2, 2}(\R^d)$ with compact support such that
		\begin{align*}
			\Delta'_{(\nabla\vartheta)_0}\FuncZ = \Delta \FuncY.
		\end{align*}
\end{lemma}

\begin{lemma}[see Proposition 17 in \cite{AltmeyerCialencoPasemann2023}]\label{lem:ACP:17}
	Let $\alpha\geq 0$, $p\geq 2$. Then we have
	\begin{align}
		\sup_{0<\delta\leq 1,t> 0}\norm{(-t\Delta_{\vartheta_\delta})^\alpha S_\delta(t)}_{L^p(\cD_\delta)}\leq C_\alpha.
	\end{align}
\end{lemma}

\begin{proof}
	The proof works verbatim as in Proposition 17 in \cite{AltmeyerCialencoPasemann2023}.
	There, $\vartheta$ is constant, but this is not used in the proof.
\end{proof}

\section{Variance Bounds}\label{sec:VarianceBounds}

In this section, we provide a useful variance bound for Gaussian random elements.
A similar statement can be found in \cite[Lemma C.1]{ReissStrauchTrottner2023}. We simplify the proof therein by a stochastic argument.

\begin{lemma}[separation of variances]\label{lem:abstract-variance-bounds}
	Let $A, B$ be centered jointly Gaussian random elements in $\R^{\{1,\dots,N_\varepsilon\}\times\cX_\varepsilon}$, and $c\in \R^{\{1,\dots,N_\varepsilon\}\times\cX_\varepsilon}$. Then
	\begin{align}
		\rVar(\norm{A}_w^2) &= 2\sum_{k,\ell=1}^{N_\varepsilon}\sum_{x, y\in\cX_\varepsilon}w^h_\varepsilon(x)w^h_\varepsilon(y)\bE[A_{k, x}A_{\ell, y}]^2, \\
		\rVar(\scp{A}{B}_w) &\leq \sqrt{\rVar(\norm{A}_{\abs{w}}^2)}\sqrt{\rVar(\norm{B}_{\abs{w}}^2)},\\
		\rVar(\scp{A}{c}_w) &\leq \frac{1}{\sqrt{2}}\norm{c}_{\abs{w}}^2\sqrt{\rVar(\norm{A}_{\abs{w}}^2)}.
	\end{align}
\end{lemma}

\begin{proof}
		\begin{align*}
			\rVar(\scp{A}{B}_w)
				\hspace{-1.7cm}&\hspace{1.7cm}= \sum_{k,\ell=1}^{N_\varepsilon}\sum_{x, y\in\cX_\varepsilon}w^h_\varepsilon(x)w^h_\varepsilon(y)\rCov(A_{k, x}B_{k, x}, A_{\ell, y}B_{\ell,y}) \\
				&= \sum_{k,\ell=1}^{N_\varepsilon}\sum_{x, y\in\cX_\varepsilon}w^h_\varepsilon(x)w^h_\varepsilon(y)\left(\bE[A_{k, x}B_{k, x}A_{\ell, y}B_{\ell,y}] - \bE[A_{k, x}B_{k, x}]\bE[A_{\ell, y}B_{\ell,y}]\right) \\
				&= \sum_{k,\ell=1}^{N_\varepsilon}\sum_{x, y\in\cX_\varepsilon}w^h_\varepsilon(x)w^h_\varepsilon(y)\left(\bE[A_{k, x}A_{\ell, y}]\bE[B_{k, x}B_{\ell,y}] + \bE[A_{k, x}B_{\ell, y}]\bE[A_{\ell, y}B_{k, x}]\right),
		\end{align*}
		proving the first claim for $A=B$. Continuing the general case, the Cauchy-Schwarz inequality in $\R^{\{1,\dots,N_\varepsilon\}^2\times\cX_\varepsilon^2}$ yields
		\begin{align*}
			\rVar(\scp{A}{B}_w)
				\hspace{-2.1cm}&\hspace{2.1cm} \leq \\
				&\left(\sum_{k,\ell=1}^{N_\varepsilon}\sum_{x, y\in\cX_\varepsilon}\hspace{-0.1cm}\abs{w^h_\varepsilon(x)w^h_\varepsilon(y)}\bE[A_{k, x}A_{\ell, y}]^2\right)^{\hspace{-0.1cm}\frac{1}{2}}\hspace{-0.1cm}\left(\sum_{k,\ell=1}^{N_\varepsilon}\sum_{x, y\in\cX_\varepsilon}\hspace{-0.1cm}\abs{w^h_\varepsilon(x)w^h_\varepsilon(y)}\bE[B_{k, x}B_{\ell, y}]^2\right)^{\hspace{-0.1cm}\frac{1}{2}} \\
				&\quad\quad\quad\quad  + \sum_{k,\ell=1}^{N_\varepsilon}\sum_{x, y\in\cX_\varepsilon}\abs{w^h_\varepsilon(x)w^h_\varepsilon(y)}\bE[A_{k, x}B_{\ell, y}]^2 \\
				&= \frac{1}{2}\sqrt{\rVar(\norm{A}_{\abs{w}}^2)}\sqrt{\rVar(\norm{B}_{\abs{w}}^2)} \\
				&\hspace{3cm} + \bE\left(\sum_{k=1}^{N_\varepsilon}\sum_{x\in\cX_\varepsilon}\abs{w^h_\varepsilon(x)}A_{k, x}A_{k, x}'\right)\left(\sum_{k=1}^{N_\varepsilon}\sum_{x\in\cX_\varepsilon}\abs{w^h_\varepsilon(x)}B_{k, x}B_{k, x}'\right)
		\end{align*}
		for an independent copy $(A', B')$ of $(A, B)$. Apply the Cauchy-Schwarz inequality in $L^2(\Omega)$ to the last term, and note that
		\begin{align*}
			\sqrt{\bE\left(\sum_{k=1}^{N_\varepsilon}\sum_{x\in\cX_\varepsilon}\abs{w^h_\varepsilon(x)}A_{k, x}A_{k, x}'\right)^2} = \sqrt{\sum_{k,\ell=1}^{N_\varepsilon}\sum_{x, y\in\cX_\varepsilon}\abs{w^h_\varepsilon(x)w^h_\varepsilon(y)}\bE[A_{k, x}A_{\ell, y}]^2},
		\end{align*}
		and equally for $B$, to prove the bound for $\rVar(\scp{A}{B}_w)$.
		Finally,
		\begin{align*}
			\rVar(\scp{A}{c}_w)
				\hspace{-2cm}&\hspace{2cm}= \sum_{k,\ell=1}^{N_\varepsilon}\sum_{x, y\in\cX_\varepsilon}w^h_\varepsilon(x)w^h_\varepsilon(y)c_{k, x}c_{\ell, y}\rCov(A_{k, x}, A_{\ell, y}) \\
				&\leq \left(\sum_{k,\ell=1}^{N_\varepsilon}\sum_{x, y\in\cX_\varepsilon}\abs{w^h_\varepsilon(x)w^h_\varepsilon(y)}\bE[A_{k, x}A_{\ell, y}]^2\right)^{\hspace{-0.1cm}\frac{1}{2}}\hspace{-0.1cm}\left(\sum_{k,\ell=1}^{N_\varepsilon}\sum_{x, y\in\cX_\varepsilon}\abs{w^h_\varepsilon(x)w^h_\varepsilon(y)}c_{k, x}^2c_{\ell, y}^2\right)^{\hspace{-0.1cm}\frac{1}{2}} \\
				&= \frac{1}{\sqrt{2}}\norm{c}_{\abs{w}}^2\sqrt{\rVar(\norm{A}_{\abs{w}}^2)}.
		\end{align*}
\end{proof}

\section*{Acknowledgements}

This research has been partially funded by the Deutsche Forschungsgemeinschaft (DFG) via SFB 1294, project-ID 318763901, and TRR 388, project-ID 516748464.
The authors like to thank Anton Tiepner and Eric Ziebell for helpful discussions, and two anonymous referees for their helpful reports.
Numerical simulations have been partially performed on the servers of the Humboldt Lab for Empirical and Quantitative Research.

\bibliographystyle{amsalpha}
\bibliography{references}

\end{document}